\documentclass[12pt]{amsart}

\usepackage{graphicx}
\usepackage{amsfonts}
\usepackage{epsf}
\usepackage{amssymb}
\usepackage{amsmath}
\usepackage{amscd}
\usepackage{hyperref}
\usepackage{color}
\usepackage{bbm}
\usepackage{pinlabel}
\usepackage{subcaption}
\usepackage{tikz}

\newtheorem{theorem}{Theorem}[section]
\newtheorem{proposition}[theorem]{Proposition}
\newtheorem{lemma}[theorem]{Lemma}

\newtheorem{corollary}[theorem]{Corollary}

\theoremstyle{remark}
\newtheorem{definition}[theorem]{Definition}
\newtheorem{example}[theorem]{Example}
\newtheorem{remark}{Remark}[section] 

\newtheorem{conjecture}[theorem]{Conjecture}

\newcommand{\eps}{\epsilon}

\input xy
\xyoption{all}

\begin{document}

\title[]{On a Nonorientable Analogue of the Milnor Conjecture}
\begin{abstract}
The nonorientable 4-genus 
$\gamma_4(K)$ of a knot $K$ is the smallest first Betti number of any nonorientable surface properly
embedded in the 4-ball, and bounding the knot $K$. We study a conjecture proposed by Batson about the value of $\gamma_4$ for
torus knots, which can be seen as a nonorientable analogue of Milnor's Conjecture for the orientable 4-genus of torus
knots. We prove the conjecture for many infinite families of torus knots, by relying on a lower bound for $\gamma_4$ formulated
by Ozsv\'ath, Stipsicz, and Szab\'o. As a side product we obtain new closed formulas for the signature of torus knots.
\end{abstract}
%
%\subjclass[2010]{57M25 and 57M27} 
%

%\email{gallasp2@unr.nevada.edu}
\author{Stanislav Jabuka}
\author{Cornelia A. Van Cott}

%\email{jabuka@unr.edu}
%\address{Department of Mathematics and Statistics, University of Nevada, Reno NV 89557, USA.}
\maketitle
%%%%%
%%%%%
%%%%%
%%%%%
%%%%%
\section{Introduction} \label{SectionIntroduction}
%%%
%%%
\subsection{Background} In 1968, Milnor \cite[Remark 10.9]{Milnor:1968} famously asked if the unknotting number $u(T(p,q))$ of the $(p,q)$ torus knot $T(p,q)$, equals $\frac{1}{2}(p-1)(q-1)$. Since the smooth 4-genus $g_4$ of $T(p,q)$ is bounded from above by $u(T(p,q))$, and since an explicit unknotting of $T(p,q)$ by $\frac{1}{2}(p-1)(q-1)$ crossing changes exists, proving that $g_4(T(p,q)) = \frac{1}{2}(p-1)(q-1)$ affirmatively answers Milnor's question. The conjecture that 
$$g_4(T(p,q))=\frac{1}{2}(p-1)(q-1)$$ 
has been called {\em Milnor's Conjecture} by multiple authors, a choice of historically somewhat incorrect nomenclature that we too adopt here, in the hope that it will not cause confusion. Milnor's Conjecture was first verified by Kronheimer and Mrowka \cite{KronheimerMrowka1:1993, KronheimerMrowka2:1995} in 1993 using Donaldson style gauge theory. It has since received several other proofs, most notable is the purely combinatorial proof given by Rasmussen \cite{Rasmussen:2010}. 

In this work we discuss a nonorientable analogue of this conjecture, first formulated by Batson \cite{Batson:2014}. It concerns the value of the nonorientable smooth 4-genus $\gamma_4$ of torus knots. For any knot $K$, the {\em nonorientable smooth 4-genus $\gamma_4(K)$} is defined as the smallest first Betti number of any nonorientable surface $\Sigma$ smoothly and properly embedded in the 4-ball and with $\partial \Sigma = K$. 

The knot invariant $\gamma_4(K)$ was introduced by Murakami and Yasuhara \cite{MurakamiYasuhara:2000} only in the year 2000, and relatively little is known about it (for an overview of existing results on $\gamma_4$  see \cite{GilmerLivingston:2011, JabukaKelly:2018}). Two of the main tools for computing $\gamma_4$ are lower bounds coming from Heegaard Floer homology. To state these, let $S_r^3(K)$ denote the 3-manifold resulting from $r$-framed Dehn surgery on the knot $K\subset S^3$, with $r\in \mathbb Q$. Batson \cite{Batson:2014} proved that 
%%%
%%%
%
\begin{equation} \label{BatsonsBoundOnGamma4}
\textstyle \frac{1}{2}\sigma(K) - d(S^3_{-1}(K))  \le \gamma_4(K), 
\end{equation}
where $\sigma(K)$ is the signature of the knot $K$, and for $Y$ an oriented, integral homology 3-sphere, $d(Y)$ is its Heegaard Floer correction term \cite{OzsvathSzabo3:2003}. We shall refer to inequality \eqref{BatsonsBoundOnGamma4} as {\em Batson's bound}. Using this bound, Batson was able to show that $\gamma_4$ is an unbounded function. Prior to his work the largest known value of $\gamma_4$ was 3 \cite{GilmerLivingston:2011}. 

Ozsv\'ath, Stipsicz, and Szab\'o \cite{OSS:2017} proved the inequality
\begin{equation} \label{OSSBoundOnGamma4}
|\upsilon(K) - \textstyle \frac{1}{2}\sigma(K)| \le \gamma_4(K), 
\end{equation}
where $\upsilon(K)$ is the upsilon invariant of the knot $K$ as defined in \cite{OSS:2017}. This inequality, henceforth referred to as the {\em OSS bound}, can also be used to show that $\gamma_4$ is an unbounded function. Indeed, $\upsilon-\frac{1}{2}\sigma$ is a concordance invariant, and so if $K$ is any knot with $(\upsilon-\frac{1}{2}\sigma)(K)\ne 0$, then $\gamma_4(\#^mK)$ grows without bound as $m\to \infty$. 

\vskip3mm
To state the nonorientable analogue of Milnor's Conjecture, we first recall the notion of a nonorientable band move.

\begin{definition}
A {\em nonorientable band move} on an oriented knot $K$ is the operation of attaching an oriented band $h=[0,1]\times [0,1]$ to $K$ along $[0,1]\times \partial [0,1]$ in such a way that the orientation of the knot agrees with that of $[0,1]\times \{0\}$ and disagrees with that of $[0,1]\times \{1\}$ (or vice versa), and then performing surgery on $h$, that is replacing the arcs $[0,1]\times \partial [0,1]\subseteq K$ by the arcs $\partial [0,1]\times [0,1]$.  The resulting knot $K'$ is said to have been {\em obtained from $K$ by a nonorientable band move}, and we write $K' = K\#h$ to indicate this operation. 
\end{definition}

%Note that if $K'=K\#h$ then $K=K\#h'$ where $h'$ is the {\em dual band of $h$}, obtained from $h$ by swapping the order of the two factors $[0,1]$.   
%%%
%%%
%Note that if $K'$ was obtained from $K$ by a non-oriented band move and $K'=K\#h$, then the knot $K$ is also obtained from $K'$ by a non-oriented band move and $K=K'\#h'$ where $h'$ is the ``dual band''  of $h$, see Figure~\ref{DualBands}. 

If we represent the torus knot $T(p,q)$ as a collection of parallel strands with slope $p/q$ on a square with opposite edges identified, then there is a natural \lq\lq simplest\rq\rq choice of a nonorientable band move, obtained by placing the band $h$ between a pair of neighboring strands, as in  Figure \ref{FigurePinchingMove}. We shall refer to this type of nonorientable band move on $T(p,q)$ as a {\em pinch move}. A pinch move clearly yields another torus knot as the newly created knot still lies on a torus. 
%%%
%%%
\begin{figure}   %[b]{0.27\textwidth}
\centering
\labellist
\small\hair 2pt
\pinlabel (a) at 150 10
\pinlabel (b) at 460 10
\pinlabel (c) at 760 10
\endlabellist
\includegraphics[width=15cm]{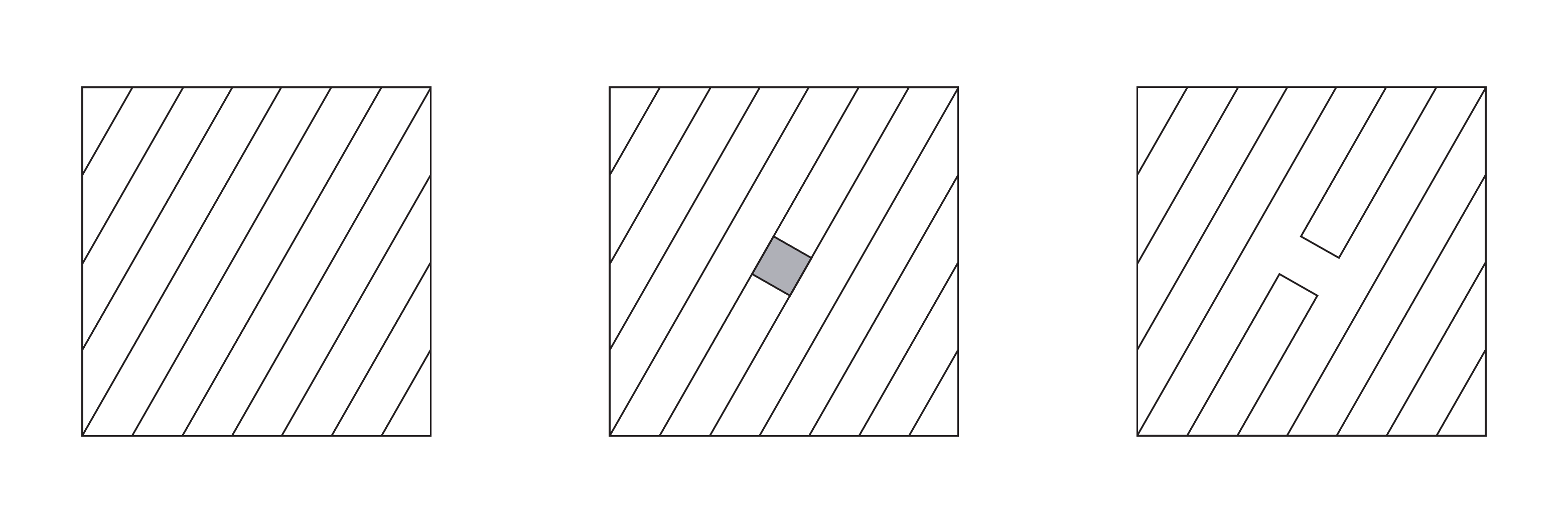}
\vskip2mm
\caption{A nonorientable band move on the torus knot $T(7,4)$ (Subfigure (a)) using the handle $h$ (shaded band in Subfigure (b)) yields the torus knot $T(3,2)$ (Subfigure (c)). A nonorientable band move resulting from placing the band between a pair of neighboring strands, is referred to as a {\em pinch move}.  }
\label{FigurePinchingMove}
\end{figure}
%%%
%%%
Batson \cite{Batson:2014} states (without proof, however see Lemma~\ref{LemmaOnTheProofOfBatsonsPinchingMoveFormula} for a justification) that the knot $T(r,s)$ obtained by a pinch move on $T(p,q)$ (with $p,q>0$) is (up to orientation) the knot $T(r,s)$ with  
\begin{align} \label{Definition Of (r,s) From (p,q)}
(r,s) =(|p-2t|, |q-2h|) 
\end{align}
where $t$ and $h$ are the integers uniquely determined by the requirements 
\begin{align} \label{EquationFormulasForTAndH}
t & \equiv -q^{-1}\,(\text{mod }p)\qquad\text{ and } \qquad   t\in \{0,\dots, p-1\},\cr
h & \equiv \phantom{-} p^{-1}\,(\text{mod }q) \qquad \text{ and } \qquad h\in \{0,\dots, q-1\}.
\end{align}
%
%Observe that the parities of $r$ and $s$ coincide with those of $p$ and $q$ respectively. Moreover, we will show (Lemma~\ref{MagnitudeOfRandS}) the relative magnitudes of the parameters are preserved (meaning, for example, if $p>q$, then $r\geq s$, with equality only if $r = s = 1$). We will also show (Lemma \ref{LemmaAboutTheSameSignaOfRAndS}) that the signs of $p-2t$ and $q - 2h$ never differ from each other. We shall say that $(r,s)$ in \eqref{Definition Of (r,s) From (p,q)} was obtained from $(p,q)$ by a {\em positive pinch move} (or a {\em $+1$-move}) if $p-2t, q-2h \ge 0$, and that $(r,s)$ was obtained from $(p,q)$ by a {\em negative pinch move} (or a {\em $-1$-move}) if $p-2t, q-2h \le 0$. 
%%%
%%%
%\begin{remark} \label{RemarkAboutWellDefinednessOfTheSignOfAPinchingMove}
%Note that the sign of the pinch move $\eps = \eps(p,q)$ is only well defined as a function of the ordered pair $(p,q)$, but not as a function of the unordered pair $\{p,q\}$. This is because \eqref{EquationFormulasForTAndH} is not symmetric with respect to swapping $p$ and $q$, as the formula for $t$ carries a minus sign, while the formula for $h$ does not. We leave it as an easy exercise for the reader to check that $\eps(p,q) = - \eps(q,p)$. 

%Since the knots $T(p,q)$ and $T(q,p)$ are isotopic, and since we'd like $\eps$ to be associated to the knot type of $T(p,q)$, this presents an ambiguity which we rectify by adopting the convention that if $pq$ is even, then $p$ is the even parameter, while if $pq$ is odd, we demand that $p>q$. 
%\end{remark}
%%%
We shall symbolically write 
\begin{align}\label{PinchMoveParameters}
T(p,q) \stackrel{\eps}{\longrightarrow} T(r,s),
\end{align}
to indicate said pinch move from $T(p,q)$ to $T(r,s)$, where $\eps = Sign (p-2t)$, and we shall say that {\em $T(r,s)$ was obtained by an $\eps$-move from $T(p,q)$} (see Section \ref{SectionOnUndoingPinchingMoves} for a more in-depth discussion about pinch moves). 
%We shall see later that $\eps$-moves preserve our convention from Remark \ref{RemarkAboutWellDefinednessOfTheSignOfAPinchingMove}. An alternative characterization of $\eps$ in terms of continued fractions is given in Theorem~\ref{PositiveNegativePinchMove}.

As with any band move, a pinch move provides a cobordism between the two associated knots. In this case, the cobordism is from $T(p,q)$ to $T(r,s)$, and the cobordism can be embedded in $T^2 \times [-\delta, \delta]$ in $S^3$. Note that $0\le r<p$ and $0\le s<q$, showing that the knot $T(r,s)$ obtained from $T(p,q)$ is \lq\lq smaller.\rq\rq Hence, a finite sequence of pinch moves applied firs to $T(p,q)$, then to $T(r,s)$, and so on, will eventually result in the unknot $U$. Concatenating the cobordisms, we have a cobordism in $S^3$ between $T(p,q)$ and $U$. Along these lines, we make the following definition, borrowed from~\cite{Batson:2014}.

\begin{definition}[The Surface $F_{p,q}$] \label{DefinitionOfTheSurfaceFpq}
Let $p,q>1$ be relatively prime integers, where we take $q$ to be odd and $p>q$ if $p$ is also odd. Apply the minimum number $n$ of pinch moves necessary to the torus knot $T(p,q)$ so that the result is an unknot $T(p_0,1)$ for some $p_0 \ge 0$. To the associated cobordism from $T(p,q)$ to $T(p_0, 1)$, glue a disk along $T(p_0, 1)$ in $B^4$. Push the interior of the thus created surface into $B^4$ to produce a nonorientable surface $F_{p,q} \subset B^4$ with $\partial F_{p,q} = T(p,q)$ and with $b_1(F_{p,q}) =n$. 
\end{definition}

%See Figure~\ref{SurfaceFigure} for a visualization of examples of $F_{p,q}$ immersed in $S^3$. In the special case that $p<q$ with $p$ even, the surface $F_{p,q}$ embeds into $S^3$ (see, for example, Figure~\ref{SurfaceFigure-F-5-4}). 
%
%\begin{figure}   %[b]{0.27\textwidth}
%\centering
%\includegraphics[width=4cm]{Pic2}
%\hspace{.2cm}
%\includegraphics[width=4cm]{Pic3}
%\hspace{.2cm}
%\includegraphics[width=3cm]{Pic4}
%
%\vskip2mm
%\caption{The nonorientable surfaces $F_{4,3}$, $F_{5,3}$, and $F_{7,3}$ with boundary $T(4,3)$, $T(5,3)$, $T(7,3)$, respectively, immersed in $S^3$. Observe that each surface is a M\"obius band. Proposition~\ref{TorusKnotsThatBoundMobiusBands} gives a complete list of torus knots for which $F_{p,q}$ is a M\"obius band. }
%\label{SurfaceFigure}
%\end{figure}
%%%
%%%
\subsection{A Nonorientable Analogue of Milnor's Conjecture}
%%%
%%%%
With these preliminaries understood, the next conjecture is due to Batson \cite{Batson:2014}.
%%%
\begin{conjecture}[Nonorientable Analogue of Milnor's Conjecture] \label{NonorientableAnalogueOfMilnorsConjecture}
 \label{ConjectureNonorientableMilnorConjecture}
Let $p,q>1$ be relatively prime integers and let $n$ be the minimum number of pinch moves needed to reduce $T(p,q)$ to the unknot. Then $\gamma_4(T(p,q)) = n$, and the value of $\gamma_4(T(p,q))$ is realized by the first Betti number of the nonorientable surface $F_{p,q}$ from Definition \ref{DefinitionOfTheSurfaceFpq}. 
\end{conjecture}
%%%
%%%
A recent result of Lobb's \cite{Lobb} shows that the full scope of this conjecture is not true.
\begin{theorem}[Lobb \cite{Lobb}, 2019]
Conjecture \ref{NonorientableAnalogueOfMilnorsConjecture} is false. Specifically, the torus knot $T(4,9)$ bounds a M\"obius band smoothly and properly embedded in the 4-ball, showing that $\gamma_4(T(4,9)) = 1$, while the value of $\gamma_4(T(4,9))$ predicted by Conjecture \ref{NonorientableAnalogueOfMilnorsConjecture} is 2. 
\end{theorem}
%%%
%%%
The failure of the torus knot $T(4,9)$ to comply with Conjecture \ref{NonorientableAnalogueOfMilnorsConjecture} can be used to construct infinitely many more counterexamples, something that has also been observed by Lobb \cite{Lobb2}.
%%%
%%%
\begin{corollary} \label{CorollaryToLobbsResult}
For every $n\ge 3$ there exist infinitely many torus knots for which the value of their nonorientable 4-genus $\gamma_4 $ predicted by Conjecture \ref{NonorientableAnalogueOfMilnorsConjecture} equals $n$, but for which the actual value of $\gamma_4$ is at most $n-1$.  
\end{corollary}
%%%
%%%
The counterexamples from the previous corollary all rely on Lobb's counterexample of $T(4,9)$, but it is not unlikely that there are other counterexamples entirely, and identifying them and calculating their values of $\gamma_4$ remains a challenge for future work. Our present work aims to show that despite not being true in complete generality, Conjecture \ref{NonorientableAnalogueOfMilnorsConjecture} nevertheless holds true for many torus knots, including for each $n\in \mathbb N$, for infinite families of torus knots with $\gamma_4$ equal to $n$.   

The construction of the surface $F_{p,q}$ shows that $\gamma_4(T(p,q)) \le b_1(F_{p,q})$, providing an upper bound on $\gamma_4(T(p,q))$. For all torus knots $T(p,q)$ where the surface $F_{p,q}$ is a M\"obius band, Conjecture~\ref{ConjectureNonorientableMilnorConjecture} is obviously true, for in that case the upper bound equals 1. Such torus knots are characterized by the next proposition. 
%%%
%%%
\begin{proposition}\label{TorusKnotsThatBoundMobiusBands}
The surface $F_{p,q}$ is a M\"obius band if and only if $T(p,q)$ is up to passing to its mirror knot, given as $T(qm\pm 2, q)$ where $q \ge 3$ is odd and $m\ge 0$. Hence
$$\gamma_4(T(qm\pm 2, q)) = 1.$$
\end{proposition}
%
%See Figure~\ref{SurfaceFigure} for some examples of these torus knots. This proposition is a simple consequence of a Theorem~\ref{TheoremOnReversingPinchingMoves} below. 
%
%  
For other torus knots, we make use of lower bounds on $\gamma_4$ to pinpoint the invariant's value. Going forward, our choice of lower bound for $\gamma_4(T(p,q))$ will be the OSS bound \eqref{OSSBoundOnGamma4}, leading to this double inequality:
\begin{equation} \label{EquationDoubleInequalityForGamma4}
|\upsilon (T(p,q)) - \textstyle \frac{1}{2}\sigma(T(p,q))| \le \gamma_4(T(p,q)) \le b_1(F_{p,q}).
\end{equation}
%%%
%%%
%%%
The purpose of this paper is to compute the OSS lower bound for all torus knots. The starting point for this computation is a series of pinch moves that convert $T(p,q)$ to the unknot. Thus, relying on notation from \eqref{PinchMoveParameters}, write 
\begin{equation} \label{EquationSequenceOfPinchingMovesFirst}
T(p,q) = T(p_n,q_n) \stackrel{\eps_n}{\longrightarrow} T(p_{n-1},q_{n-1}) \stackrel{\eps_{n-1}}{\longrightarrow}\dots \stackrel{\eps_2}{\longrightarrow} T(p_1,q_1) \stackrel{\eps_1}{\longrightarrow}T(p_0,1),
\end{equation}   
for integers $n, p_n,q_n,\dots, p_1,q_1\ge 1$, $p_0\ge 0$ and signs $\eps_1,\dots, \eps_n\in \{\pm 1\}$.  
We state our results by distinguishing the two cases of $p$ being odd or even.  
%%%
%%%
\begin{theorem} \label{MainResultForP0Odd}
Let $p,q>1$ be a pair of relatively prime odd integers with $p>q$, and let $n$, $\{p_k\}_{k=0}^n$, $\{q_k\}_{k=1}^n$ and $\{\eps_k\}_{k=1}^n$ be as in \eqref{EquationSequenceOfPinchingMovesFirst}. For $k=1,\dots, n-1$, let 
$$m_k = (p_{k+1}+\eps_{k}\eps_{k+1}p_{k-1})/p_k.$$
Then each $m_k$ is an even integer. If $n=1$, let $\mathcal I=\emptyset$, and if $n\ge 2$ let  
$$\mathcal I = \{k\in \{1,\dots, n-1\}\,|\, m_k\equiv 0\,(\text{mod }4)\},$$ 
and write $\mathcal I=\{k_{1}, k_{2}, \dots, k_{\ell}\}$ with $k_{1}<k_{2}<\dots<k_{\ell}$, if $\mathcal I \ne \emptyset$. Then 
$$\textstyle (\upsilon -\frac{1}{2}\sigma)(T(p,q)) = \left\{
\begin{array}{rl}
n - \sum _{k_i\in \mathcal I} (-1)^{i-1} (n-k_{i}) & \quad ; \quad q_1-\eps_1 \equiv 0\,(\text{mod }4), \cr & \cr
\sum _{k_i\in \mathcal I} (-1)^{i-1} (n-k_{i}) & \quad ; \quad q_1-\eps_1 \equiv 2\,(\text{mod }4).
\end{array}
\right.$$
The sum appearing on the right-hand sides above equals zero if $\mathcal I = \emptyset$. 
\end{theorem}
%%%
%%%
\begin{theorem}  \label{MainResultForP0Even}
Let $p,q>1$ be a pair of relatively prime integers of which $q$ is odd and $p$ is even, and let $n$, $\{p_k\}_{k=0}^n$, $\{q_k\}_{k=1}^n$ and $\{\eps_k\}_{k=1}^n$ be as in \eqref{EquationSequenceOfPinchingMovesFirst}. If $n=1$, then 
$$\textstyle (\upsilon -\frac{1}{2}\sigma)(T(p,q)) = \left\{
\begin{array}{cl}
1 & \quad ; \quad \eps_1=1, \cr
0 & \quad ; \quad \eps_1=-1.
\end{array}
\right.$$
If $n=2$ let $\mathcal J=\emptyset$, and for $n\ge 3$ define the set $\mathcal J$ as   
$$\mathcal J = \{k\in \{2,\dots, n-1\}\,|\, \eps_{k}\eps_{k+1} = -1\}$$ 
and write $\mathcal J=\{k_{1}, k_{2}, \dots, k_{\ell}\}$ with $k_{1}<k_{2}<\dots<k_{\ell}$, if $\mathcal J \ne \emptyset$. Then 
$$\textstyle (\upsilon -\frac{1}{2}\sigma)(T(p,q)) = \left\{
\begin{array}{rll}
n - \sum _{k_j\in \mathcal J} (-1)^{j-1} (n-k_{j}) & \quad ; \quad \eps_1=1, &\eps_2=1, \cr
1 +\sum _{k_j\in \mathcal J} (-1)^{j-1} (n-k_{j}) & \quad ; \quad \eps_1=1, &\eps_2=-1, \cr
(n-1) - \sum _{k_j\in \mathcal J} (-1)^{j-1} (n-k_{j}) & \quad ; \quad \eps_1=-1, &\eps_2=1, \cr
\sum _{k_j\in \mathcal J} (-1)^{j-1} (n-k_{j}) & \quad ; \quad \eps_1=-1, &\eps_2=-1. \cr
\end{array}
\right.$$
The sum appearing on the right-hand sides above equals zero if $\mathcal J = \emptyset$. 
\end{theorem}
%%%
%%%
It will be shown in Section \ref{SectionWithProofsOfMainTheorems} (Remarks \ref{FirstRemarkAboutTheValueOfTheSumFromMainTheorem} and \ref{SecondRemarkAboutTheValueOfTheSumFromMainTheorem}) that the sums appearing in Theorems \ref{MainResultForP0Odd} and \ref{MainResultForP0Even}, attain values between $0$ and $n-1$, leading to this result.   
%%%
\begin{corollary}
Let $T(p,q)$ be any torus knot with $p, q$ nonnegative, then  
$$ \upsilon (T(p,q)) \ge \textstyle \frac{1}{2}\sigma (T(p,q)).$$
\end{corollary}
%%%

The next theorem and its corollary, both of which are a direct consequence of Theorems \ref{MainResultForP0Odd}, \ref{MainResultForP0Even} and the OSS bound \eqref{OSSBoundOnGamma4}, add infinitely many new families of torus knots to the list of knots for which Conjecture~\ref{ConjectureNonorientableMilnorConjecture} is true, including for every $n\in \mathbb N$, infinite families of torus knots with $\gamma_4$ equal to $n$. 
%%%
%%%
\begin{theorem} \label{TheoremOnPartialSolutionOfNonorientableMilnorConjecture}
Let $p,q>1$ be a pair of relatively prime integers of which $q$ is odd. Let $n$, $\{p_k\}_{k=0}^n$, $\{q_k\}_{k=1}^n$ and  $\{\eps_k\}_{k=1}^n$ be as in \eqref{EquationSequenceOfPinchingMovesFirst}. For $k=1,\dots, n-1$ set $m_k = (p_{k+1}+\eps_{k}\eps_{k+1}p_{k-1})/p_k$ as in Theorem \ref{MainResultForP0Odd}, and recall that $m_k$ is an even integer. 
\begin{itemize}
\item[(a)] If $p$ is odd and $p>q$ then 
$$\upsilon (T(p,q)) -\textstyle \frac{1}{2}\sigma (T(p,q)) = n$$
if and only if $q_1 \equiv \eps_1 \,(\text{mod }4)$ and $m_k\equiv 2\,(\text{mod }4)$ for all $k=1,\dots, n-1$. For any such torus knot $T(p,q)$ one obtains $\gamma_4(T(p,q)) = n$.%, showing that Conjecture~\ref{ConjectureNonorientableMilnorConjecture} holds true in this case.  
%%%
\vskip2mm
\item[(b)] If $p$ is even then 
$$\upsilon (T(p,q)) -\textstyle \frac{1}{2}\sigma (T(p,q)) = n$$
if and only if $\eps_k=1$ for all $k=1,\dots, n$. For any such torus knot $T(p,q)$ one obtains $\gamma_4(T(p,q)) = n$.%, showing that Conjecture~\ref{ConjectureNonorientableMilnorConjecture}  holds true in this case as well.  
\end{itemize}
\end{theorem}
%%%
%%%
\begin{corollary} \label{CorollaryAboutValiditiyOfNonOrientableMilnorConjecture}
Let $T(p,q)$ be any torus knot satisfying the conditions from part (a) or part (b) from Theorem \ref{TheoremOnPartialSolutionOfNonorientableMilnorConjecture}. Then $T(p,q)$ has nonorientable 4-genus $\gamma_4$ as predicted by the Nonorientable Analogue of the Milnor Conjecture \ref{NonorientableAnalogueOfMilnorsConjecture}. Specifically, for any such torus knot $T(p,q)$ one obtains 
$$\gamma_4(T(p,q)) = b_1(F_{p,q})$$
with $F_{p,q}$ the nonorientable surface from Definition \ref{DefinitionOfTheSurfaceFpq}. Additionally, for any $m\in \mathbb N$ it follows that 
$$\gamma_4(\#^mT(p,q)) = m\cdot  \gamma_4(T(p,q)).$$
\end{corollary}
In Proposition \ref{PropositionCharacterizingTheCaseOfUpsilonMinusSigmaOverTwoEqualsN-1} we provide a complete list of torus knots $T(p,q)$ with  $(\upsilon-\frac{1}{2}\sigma)(T(p,q)) = n-1$ and $b_1(F_{p,q}) = n$. All such knots obey the double inequality $n-1 \le \gamma_4(T(p,q)) \le n$. 

\begin{example} \label{ExampleOfConcreteFamiliesOfTorusKnotsPassingTheNonorientableMilnorConjecture}
Many families of torus knots satisfying the conditions of Theorem \ref{TheoremOnPartialSolutionOfNonorientableMilnorConjecture} (and hence the conclusions of Corollary \ref{CorollaryAboutValiditiyOfNonOrientableMilnorConjecture}) can be explicitly described. The simplest among them are the knots
%obtained from the choices of $m_k=2$ for $k=1,\dots, n-1$, $\eps_k=\eps\in\{\pm1\}$ for all $k=1,\dots, n$, and with $p_0\ge 1$ and $q_1\ge 3$ odd, being arbitrary parameters. By solving the recursive relations \eqref{EquationSequenceOfPinchingMoves} one easily finds that   
%
$$T(p_k,q_k)  =  T( p_0+k(p_0(q_1-1)-2\eps)\, , \,1+k(q_1-1)),$$
with $\eps \in \{\pm 1\}$, $q_1\ge 3$ odd and $p_0\ge 1$. These choices lead to $m_k=2$ for all $k\ge 1$. Thus, if either $p_0\ge 2$ is even and $\eps=1$, or $p_0$ is odd and $q_1$ is congruent to $\eps$ modulo 4, then $\gamma_4(T(p_n,q_n)) = n$, showing that $T(p_n,q_n)$ satisfies Conjecture~\ref{ConjectureNonorientableMilnorConjecture}. 

Simplifying our choices further by picking $p_0=2$ and $q_1=3$ leads to 
$$T(p_k,q_k) = T(2(k+1) , 2(k+1) - 1), $$
which incidentally is the one-parameter family of knots used by Batson \cite{Batson:2014} to first prove that $\gamma_4$ is an unbounded function. Thus, his examples appear as one of the simplest subfamilies of our examples for which Conjecture~\ref{ConjectureNonorientableMilnorConjecture}  holds. \end{example}

%\begin{example}
%It is easy to see that $\gamma_4(T(2,q))=1$ for any $q$ odd, and that $\gamma_4(T(p,3))=1$ (from $T(3k\pm 1, 3)\stackrel{\pm 1}{\longrightarrow}T(k\pm 1, 1)$).  
%\end{example}

%The computation of $\upsilon -\frac{1}{2}\sigma$ for torus knots uses recursive formulas for $\upsilon$ and $\sigma$ that can be found in \cite{FellerKrcatovich:2017} and \cite{GordonLitherlandMurasugi:1981} respectively (see also Sections \ref{SectionOnComputingTheUpsilonInvariant} and \ref{SectionForSignatureComputations}). While the recursive formula for $\upsilon$ allows for a complete computation of $\upsilon (T(p,q))$ for all torus knots $T(p,q)$ (see next theorem) and is exceedingly simple, the recursive formula for $\sigma$ is more delicate, and we have not been able to use it to derive a closed formula for $\sigma(T(p,q))$ directly. Indeed, it is %only because of 
%the similarity of the two recursive formulas for $\upsilon$ and $\sigma$, that leads to many canceling terms when computing $\upsilon -\frac{1}{2}\sigma$, and to the ultimate success in computing $\upsilon -\frac{1}{2}\sigma$ for all torus knots. It was only after we had computed $\upsilon$ and $\upsilon -\frac{1}{2}\sigma$, that we were able to obtain the new, closed formula for the signature of torus knots presented in Corollary \ref{CorollaryClosedSignatureFormulas} (other algorithms for computing $\sigma(T(p,q))$ can be found in \cite{BorodzikOleszkiewicz:2010, GordonLitherlandMurasugi:1981}).  

A side product of our computations of the OSS lower bound for $\gamma_4$ are explicit formulas for $\upsilon$ and $\sigma$ for all torus knots. 

\begin{theorem} \label{TheoremAboutTheValueOfUpsilonForAllTorusKnots}
Let $p,q>1$ be relatively prime integers with $q$ odd, and either with $p$ even or with $p>q$. Suppose that $T(p,q)$ reduces to the unknot $T(p_0,1)$ after $n$ pinch moves. Then 
$$\upsilon(T(p,q)) = \frac{1}{2} n+\frac{1}{4} ( p_0 - pq).$$
In particular, since $\upsilon(T(p,q))\in \mathbb Z$, it follows that $n\equiv \frac{1}{2}(p_0-pq)\, (\text{mod }2)$.
\end{theorem}
%%%
%%%
%The value of $p_0$, which is used in the formula for $\upsilon$ above, is pinned down succintly in Theorem~\ref{p0}. 
The next statement is a direct consequence of Theorems \ref{MainResultForP0Odd}, \ref{MainResultForP0Even}, and \ref{TheoremAboutTheValueOfUpsilonForAllTorusKnots}.

\begin{corollary} \label{CorollaryClosedSignatureFormulas}
Let $p,q>1$ be relatively prime integers with $q$ odd and either with $p$ even or $p>q$. Let $n$, $\{p_k\}_{k=0}^n$, $\{q_k\}_{k=1}^n$ and $\{\eps_k\}_{k=1}^n$ be as in \eqref{EquationSequenceOfPinchingMoves} and set $m_k = (p_{k+1}+\eps_{k}\eps_{k+1}p_{k-1})/p_k$ for $k=1,\dots, n-1$, as in Theorem \ref{MainResultForP0Odd}. 
%%%
\begin{itemize}
\item[(i)] 
Suppose $p$ is odd. If $n=1$ set $\mathcal I=\emptyset$, and for $n\ge 2$ let $\mathcal I=\{k_1,\dots, k_\ell\}$ be the set from Theorem \ref{MainResultForP0Odd} of all indices $k$ for which $m_k\equiv 0\,(\text{mod }4)$, listed in increasing order. Then 
$$\sigma(T(p,q)) = \left\{
\begin{array}{rl}
\frac{1}{2} (p_0 - pq) -  \left[n - 2\sum _{k_i\in \mathcal I} (-1)^{i-1} (n-k_i)\right] &  ;  q_1-\eps_1 \equiv 0\,(\text{mod }4),  \cr & \cr
\frac{1}{2}(p_0 -pq) +\left[n - 2\sum _{k_i\in \mathcal I} (-1)^{i-1} (n-k_i)\right] &  ;   q_1-\eps_1 \equiv2\,(\text{mod }4),
\end{array}
\right. $$
where it is understood that the empty sum (case of $\mathcal I=\emptyset$) equals 0. 
%%%
%%%
\vskip3mm
\item[(ii)] Suppose $p$ is even. If $n=1$ then $\sigma(T(p,q)) = -\eps_1+\frac{1}{2}(p_0-pq).$ If $n=2$ let $\mathcal J=\emptyset$ and for $n\ge 3$ let $\mathcal J=\{k_1,\dots, k_\ell\}$ be the set from Theorem \ref{MainResultForP0Even} of all indices $k$ for which $\eps_{k}\eps_{k+1} = -1$, listed in increasing order. Then 
$$\sigma(T(p,q)) = \left\{
\begin{array}{rl}
\frac{1}{2} (p_0 - pq) -  \eps_1 \left[n - 2\sum _{k_j\in \mathcal J} (-1)^{j-1} (n-k_j)\right] & ;  \eps_1\eps_2 =1,\cr  & \cr
-2\eps_1+ \frac{1}{2} (p_0 - pq) +\eps_1  \left[n - 2\sum _{k_j\in \mathcal J} (-1)^{j-1} (n-k_j)\right] &  ; \eps_1\eps_2=-1,
\end{array}
\right. $$
where it is understood that the empty sum (case of $\mathcal J=\emptyset$) equals 0. 
\end{itemize}
\end{corollary}
\subsection{Organization}  Section \ref{SectionOnUndoingPinchingMoves} begins with the proof of Batson's formula for the parameters produced in a pinch move (Equation~\ref{Definition Of (r,s) From (p,q)}) and proves several subsequent relationships among these parameters, culminating in Theorem \ref{TheoremOnReversingPinchingMoves}. Section \ref{SectionOnUndoingPinchingMoves} also provides a proof of Proposition \ref{TorusKnotsThatBoundMobiusBands}. 

In Section \ref{SectionOnComputingTheUpsilonInvariant} we calculate $\upsilon(T(p,q))$ for all torus knots $T(p,q)$ and provide a proof of Theorem \ref{TheoremAboutTheValueOfUpsilonForAllTorusKnots}. 

In Section \ref{SectionForSignatureComputations}, which is largely technical and computation heavy,  we obtain a formula for $(\upsilon-\frac{1}{2}\sigma)(T(p,q))$ for all torus knots $T(p,q)$. These formulas are applied in Section \ref{SectionWithProofsOfMainTheorems} to provide proofs of Theorems \ref{MainResultForP0Odd}, \ref{MainResultForP0Even}, and \ref{TheoremOnPartialSolutionOfNonorientableMilnorConjecture}.  

\vskip5mm
{\bf Acknowledgements}  We wish to thank Josh Batson and Andrew Lobb for helpful email correspondence. The second author also thanks Robert Lipshitz for discerning questions that led to further insights. S. Jabuka was partially supported by a grant from the Simons Foundation, Award ID 524394.

\section{Relationships among the pinch move parameters} \label{SectionOnUndoingPinchingMoves}
%%%%%%%%%%%%%%%%%%%%%%%%%%%%%%%%%%%%%%%%%%%%%%%%%%%%%%%%%%%%%%%%%%%%%%%%%%%%%%%%%%%%%%%%%%%%%%%%%%%%%%%%%%%%%%%%%%%%%%%%%%%%%%%%%%%%%%%%%%%%
%%%%%%%%%%%%%%%%%%%%%%%%%%%%%%%%%%%%%%%%%%%%%%%%%%%%%%%%%%%%%%%%%%%%%%%%%%%%%%%%%%%%%%%%%%%%%%%%%%%%%%%%%%%%%%%%%%%%%%%%%%%%%%%%%%%%%%%%%%%%
The goal of this section is to describe several of the relationships among the parameters present in pinch moves $T(p,q) \stackrel{\eps}{\longrightarrow} T(r,s)$ as in \eqref{PinchMoveParameters}. We begin by proving the formula for $r$ and $s$ in terms of $p$ and $q$, given in Equation~\ref{Definition Of (r,s) From (p,q)}. We then proceed to prove several useful inter-relationships among the values of $p, q, r, s,$ and $\eps$, culminating with Theorem~\ref{TheoremOnReversingPinchingMoves}.% In fact, we prove a stronger result than Theorem~\ref{TheoremOnReversingPinchingMoves} (see Theorem~\ref{TheoremOnReversingPinchingMoves}). Theorem~\ref{TheoremOnReversingPinchingMoves} is a consequence of this stronger result.

\begin{lemma} \label{LemmaOnTheProofOfBatsonsPinchingMoveFormula}
Begin with a diagram of the torus knot $T(p, q)$ on the flat torus. Apply a pinch move to $T(p,q)$. The resulting torus knot (up to orientation) is $T(p-2t, q-2h)$ where $t$ and $h$ are the integers uniquely determined by the requirements 
\begin{align*}
t & \equiv -q^{-1}\,(\text{mod }p)\qquad\text{ and } \qquad   t\in \{0,\dots, p-1\},\cr
h & \equiv \phantom{-} p^{-1}\,(\text{mod }q) \qquad \text{ and } \qquad h\in \{0,\dots, q-1\}.
\end{align*}

\end{lemma}

\begin{proof}
Consider the torus knot  $T(p,q)$ drawn on the flat torus as in Figure~\ref{bandmove}(a). In the figure, without loss of generality, we depict a torus knot with $p< q$. We remark that if one travels along the curve $T(p,q)$ in direction of the given orientation in Figure~\ref{bandmove}(a), the knot meets the points on the horizontal edge in a sequence of labels \!\!\!\!\! $\mod{q}$ that increase by $p\pmod{q}$ at each step. Similarly, the knot meets the boundary points on the vertical edge in a sequence of labels \!\!\!\!$\mod{p}$ that increase by $-q \pmod{p}$ at each step. This observation will be used later.

Perform a pinch move to $T(p,q)$ along the indicated band in Figure~\ref{bandmove}(a). The curve resulting from the pinch move is depicted in Figure~\ref{bandmove}(b). We choose an orientation for the resulting curve as shown in the figure. (Note that this choice of orientation determines the signs of the resulting knot's parameters. Namely, if the torus knot with the chosen orientation is $T(r,s)$, then the torus knot corresponding to choosing the opposite orientation will be $T(-r, -s)$.) 

To identify this torus knot that results from the band move, we must count the signed intersection of this new curve with a meridian and with a longitude of the torus, which we take to be the vertical edge and horizontal edge of the rectangle in Figure~\ref{bandmove}(b), respectively.

First, we count the signed intersection of the curve with the horizontal edge of the flat torus. Let us begin at Point $P$ in Figure~\ref{bandmove}(b) and proceed along the curve {\em against} the direction of the chosen orientation. By our earlier remark, if one proceeds in this way, the intersections with the horizontal edge of the flat torus occur at points with labels $p\pmod{q}$, $2p \pmod{q}$, $3p \pmod{q}$, and so on. Since $p$ and $q$ are relatively prime, one must eventually reach the intersection point labeled 1 on the horizontal edge. This will happen in $h$ steps where $h = p^{-1}\pmod{q}$ and $0\le  h<q$ since $hp = 1\pmod{q}$.  All of these intersections up to this point are negatively oriented. After the curve meets the intersection point labeled by $1$ on the horizontal edge, the curve turns around (see Figure~\ref{bandmove}(b)), and all remaining intersection points with the horizontal edge are positively oriented. Therefore, the curve has $h$  negatively oriented intersection points and $q - h$ positively oriented intersection points with the horizontal edge. So the curve's signed intersection with the horizontal edge is $(q - h) - h = q - 2h$.

Next, we count the signed intersection of the curve with the vertical edge of the flat torus. Let's again start at Point $P$ in Figure~\ref{bandmove}(b) and proceed along the curve in the same manner as before, against the direction of chosen orientation. By our earlier remark, if one proceeds in this way, the intersections with the vertical edge of the flat torus occur at points with labels $-q\pmod{p}$, $-2q \pmod{p}$, $-3q \pmod{p}$, and so on, until reaching the intersection point labeled $1$ on the vertical edge. This will happen in $t$ steps where $t = -q^{-1}\pmod{p}$ and $0\le t<p$ since $-tq = 1\pmod{p}$.  All of the intersection points in this sequence are negatively oriented. After reaching the point labeled with $1$ on the vertical edge, the curve turns around (see Figure~\ref{bandmove}(b)), and all remaining intersection points with the vertical edge of the flat torus are positively oriented. So then, there are $t$ negatively oriented intersection points and $p - t$ positively oriented intersection points with the vertical edge. Thus, the curve's signed intersection with the vertical edge is $(p - t) - t = p - 2t$. 

From this, we conclude that up to orientation, the curve resulting from the pinch move on the torus knot $T(p,q)$ is the the torus knot $T(p - 2t,q - 2h)$, where $t \equiv-q^{-1} \pmod{p}$ with $0 \le  t < p$ and $h \equiv p^{-1} \pmod{q}$ with $0 \le h < q$.  
\end{proof}

In the following result, we find alternate characterizations for the values of $t$ and $h$ defined in Lemma~\ref{LemmaOnTheProofOfBatsonsPinchingMoveFormula}. 
%%%
\begin{lemma} \label{AuxiliaryLemmaOnHowToFindTAndH}
For $p,q>1$ relatively prime integers, the associated values of $t$ and $h$ from Lemma~\ref{LemmaOnTheProofOfBatsonsPinchingMoveFormula} are characterized as the unique integers $t$ and $h$ satisfying the equation $ph-qt=1$ with $h\in\{1,\dots, q-1\}$. 
\end{lemma}
%%%
%%%
\begin{proof}
 Let $t,h \in \mathbb Z$ satisfy the equation $ph-qt=1$ where $h\in \{1,\dots, q-1\}$. We want to show that $t$ and $h$ satisfy the requirements of Lemma~\ref{LemmaOnTheProofOfBatsonsPinchingMoveFormula}. First we verify that $t\in \{0,\dots, p-1\}.$ Observe that 
$$t=\frac{ph-1}{q} = \frac{ph}{q}- \frac{1}{q} \leq \frac{p(q-1)}{q}- \frac{1}{q} <p-\frac{1}{q}.$$
Since $t$ is an integer and strictly less than $p$, then $t\le p-1$. To see that $t\ge 1$, note that the negation of this inequality, namely $t\le 0$, leads to $ph\le 1$ which is not possible since $p>1$ and $h\ge 1$.

Lastly, the mod $q$ and mod $p$ reductions of the equation $ph-qt=1$ show that $h\equiv p^{-1}\,(\text{mod }q)$ and $t\equiv -q^{-1}\,(\text{mod }p)$, implying that $t$ and $h$ satisfy the requirements of Lemma~\ref{LemmaOnTheProofOfBatsonsPinchingMoveFormula}.

Conversely, suppose that $t$ and $h$ satisfy the requirements in Lemma~\ref{LemmaOnTheProofOfBatsonsPinchingMoveFormula}. Namely, $t$ and $h$ are the unique integers satisfying $t \equiv -q^{-1}\,(\text{mod }p)$, with $t\in \{0,\dots, p-1\},$ and $h \equiv \phantom{-} p^{-1}\,(\text{mod }q)$, with $h\in \{0,\dots, q-1\}$. We need to show that $ph-qt=1$. Observe that the value $\frac{ph - 1}{q}$ satisfies the requirements for $t$, and since the value of $t$ is unique, $t = \frac{ph-1}{q}$. Hence $ph - qt = 1$, as desired. 
\end{proof}
%%%
\begin{figure}
\centering
\labellist
\small\hair 2pt
\pinlabel $0$ at 88 -30
\pinlabel $1$ at 155 -30
\pinlabel $q-1$ at 750 -30
\pinlabel $p-1$ at 820 50
\pinlabel $0$ at -20 380
\pinlabel $1$ at -20 315
\pinlabel $2$ at -20 250
\pinlabel $p$ at 545 505
\pinlabel $q-1$ at 745 505
\endlabellist
\begin{subfigure}[b]{0.4\textwidth}
        \includegraphics[width=\textwidth]{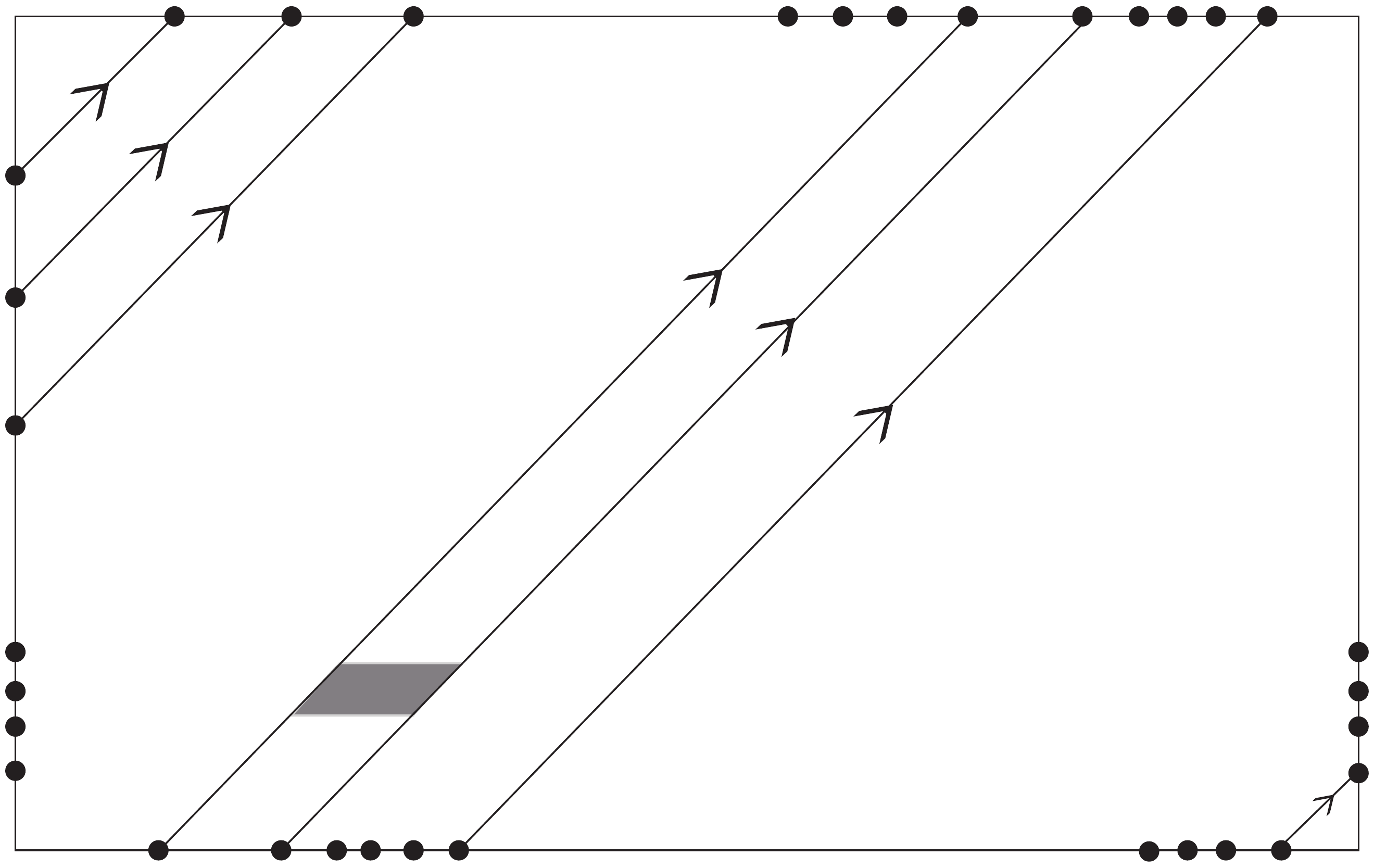}
        \caption{}
\end{subfigure}
\quad \quad \quad 
%%%
\labellist
\small\hair 2pt
\pinlabel $0$ at 88 -30
\pinlabel $1$ at 155 -30
\pinlabel $q-1$ at 750 -30
\pinlabel $p-1$ at 820 50
\pinlabel $0$ at -20 380
\pinlabel $1$ at -20 315
\pinlabel $2$ at -20 250
\pinlabel $p$ at 545 505
\pinlabel $q-1$ at 745 505
\pinlabel $\textcolor{black}{P}$ at 230 205
\endlabellist
\begin{subfigure}[b]{0.4\textwidth}
        \includegraphics[width=\textwidth]{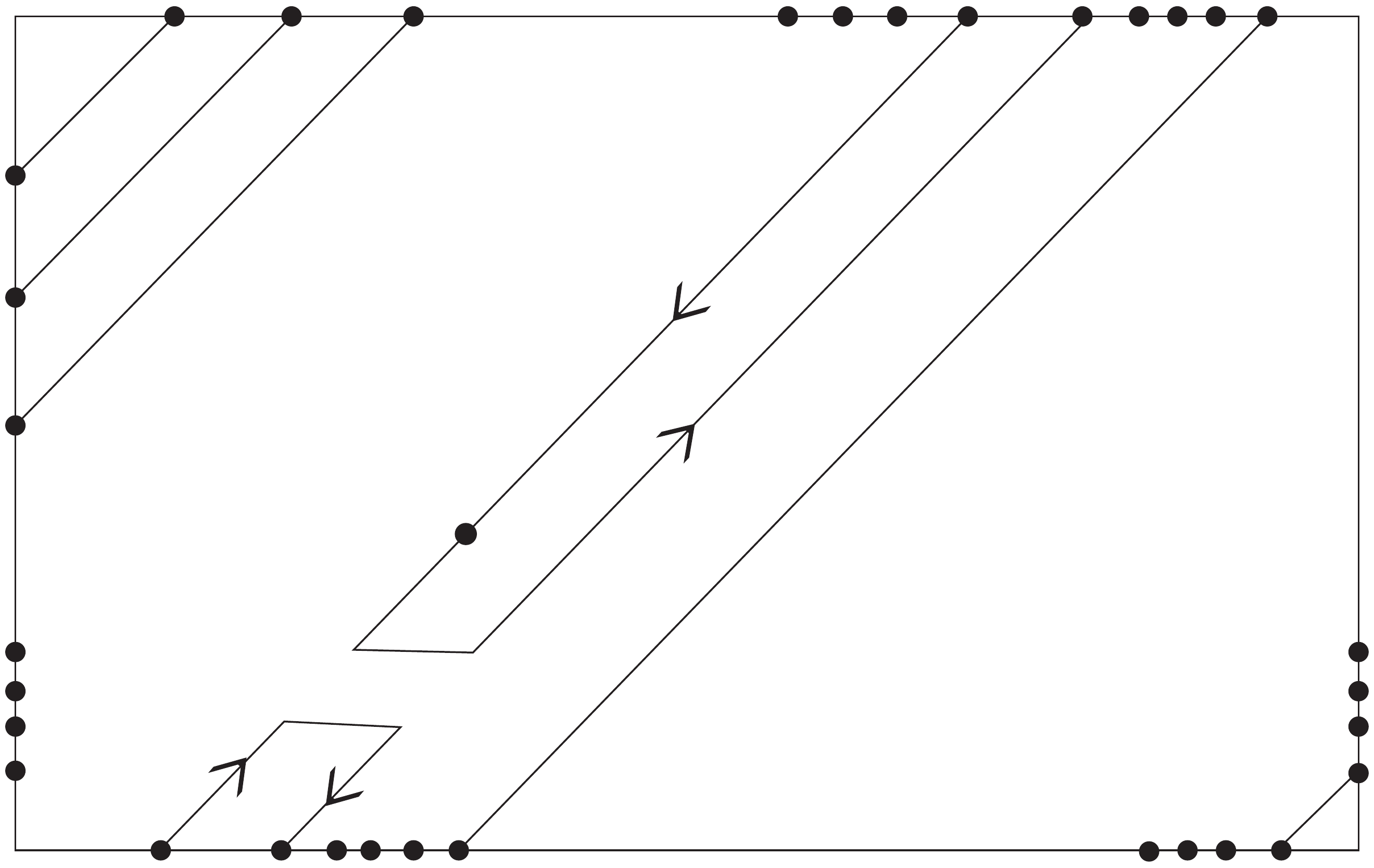}
         \caption{}
\end{subfigure}
\vskip3mm
\caption{Beginning with the knot $T(p,q)$, we perform one band move denoted by the gray band in (a). The resulting curve, denoted in (b), is again a torus knot $T(r,s)$ for some $r$ and $s$. The values of $r$ and $s$ are determined in Lemma~\ref{LemmaOnTheProofOfBatsonsPinchingMoveFormula}.}\label{bandmove} 
\end{figure}
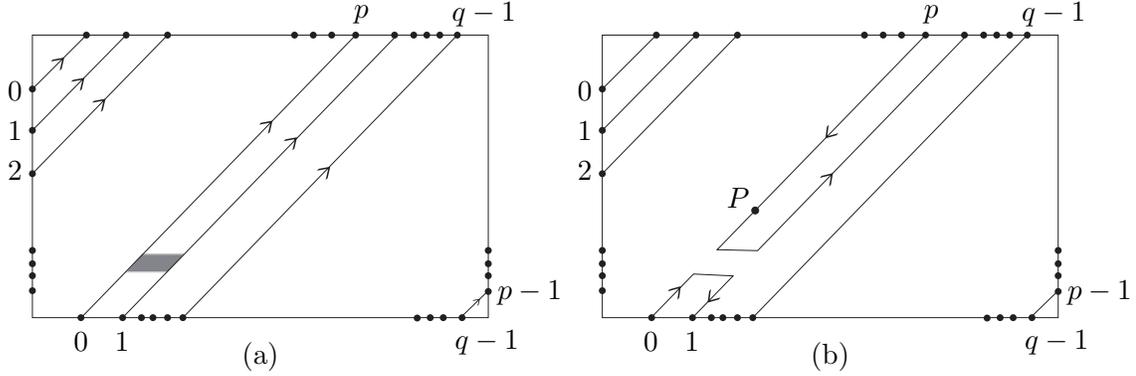

%%%
%%%
In the next lemma, we show that the signs of $p - 2t$ and $q - 2h$ never differ, which justifies our use of the terms {\em positive pinch move} and {\em negative pinch move}.  %%%
%%% 
\begin{lemma} \label{LemmaAboutTheSameSignaOfRAndS}
Let $p,q>1$ be relatively prime integers. Without loss of generality, let $q$ be odd and let $p>q$ if $p$ is also odd. Let $t, h$ be obtained from $p,q$ as in Lemma~\ref{LemmaOnTheProofOfBatsonsPinchingMoveFormula}. Then 
$$(p-2t)(q-2h)\ge0,$$ 
with equality occurring if and only if $T(p,q) = T(2, \ell)$ for some odd integer $\ell$.
%%%
\end{lemma}
%%%
%%%
\begin{proof}
Observe that the parities of $p-2t$ and $q - 2h$ coincide with the parities of $p$ and $q$, respectively. Therefore, if $(p-2t)(q-2h) = 0$, it must be that $pq$ is even, as well. And so, by our convention, $p$ must be even and $q$ must be odd. Therefore, $p-2t=0$, and so $t = \frac{p}{2}$. Using the equation in Lemma~\ref{AuxiliaryLemmaOnHowToFindTAndH}, substituting, and simplifying, we have
\begin{align*} 
%1 & = ph - qt  \cr
1& = ph - q\bigg(\frac{p}{2}\bigg) \cr
1& = \frac{p}{2} (2h - q) 
\end{align*}
It follows that $p = 2$ and $q - 2h = -1$. In conclusion, if $(p-2t)(q-2h) = 0$, then $T(p,q) = T(2, \ell)$ for some odd integer $\ell$.%, and a pinch move applied to this torus knot produces $T(0, -1)$, up to orientation. 

Conversely, if $T(p,q) = T(2, \ell)$ for some odd integer $\ell$, then a brief computation shows that $t=1$ and $h=\frac{\ell-1}{2} +1$, leading to $p-2t=2-2= 0$ and $q-2h = \ell - (\ell-1) -2 = -1$, and so $(p-2t)(q-2h) = 0$, as desired.

Next, suppose that $p-2t >0$. We want to show that $q - 2h >0$, as well. Starting with $p-2t >0$ and using the fact (from Lemma~\ref{AuxiliaryLemmaOnHowToFindTAndH}) that $ph = 1 + qt$ we have:
\begin{align*} 
%p - 2t&>0 \cr
ph - 2th &>0  \cr
1+qt - 2th &> 0 \cr
q - 2h &> \frac{-1}{t}
\end{align*}
Since $t \ge 1$ and since $q - 2h$ must be odd, it follows that $q-2h>0$, as desired.

In the other direction, suppose that $p-2t<0$. Mirroring the argument above, it follows that $q - 2h < 0$, as well, completing the proof of the lemma. 
\end{proof}

%%%
%Now, we find an alternate characterization for the values of $r$ and $s$ in \eqref{PinchMoveParameters}. 
\begin{lemma} \label{AuxiliaryLemmaOnTheRelationBetweenPQRAndS}
Suppose $T(r,s)$ is a torus knot obtained from the torus knot $T(p,q)$ by an $\eps$-move. Then $rq-sp = 2\eps$. 
\end{lemma}
%%%
%%%
\begin{proof}
 From Lemma \ref{AuxiliaryLemmaOnHowToFindTAndH}, we know that there exist $t$ and $h$ such that $r=\eps(p-2t)$ and $s=\eps(q-2h)$. Multiplying these two equations by $q$ and $p$ respectively, subtracting them, and applying Lemma~\ref{AuxiliaryLemmaOnHowToFindTAndH} leads to 
$$rq-sp = (p-2t)q\eps - (q-2h)p\eps =  2\eps(ph-qt) = 2\eps.$$
\end{proof}

\begin{lemma}\label{MagnitudeOfRandS}
Let $p,q>1$ be relatively prime integers and let the torus knot $T(r,s)$ have been obtained from the torus knot $T(p,q)$ by an $\eps$-move, with $r,s$ as in Lemma~\ref{LemmaOnTheProofOfBatsonsPinchingMoveFormula}. 
\begin{itemize}
\item[(a)] If $p>q$ and $\eps=1$ then $r>s$, while if $p>q$ and $\eps=-1$ then $r\ge s$.
\item[(b)] If $q>p$ and $\eps=1$ then $s \ge r$, while if $q>p$ and $\eps=-1$ then $ s> r$.
\end{itemize} 
\end{lemma}
%%%
%%%
\begin{proof}
Let $t,h\in \mathbb Z$ be the unique integers from Lemma \ref{AuxiliaryLemmaOnHowToFindTAndH} for which $ph-qt=1$ and with $h\in \{1,\dots, q-1\}$.% and $t\in \{1,\dots, p-1\}$. 

Consider the case of $p>q$ and assume firstly that $\eps = 1$. This means that $p-2t\geq 0$ and $q-2h\geq0$. Then
$$r=p-2t=p-2 \bigg(\frac{ph-1}{q}\bigg) = \frac{p}{q}(q-2h) + \frac{2}{q} \geq q-2h+\frac{2}{q} > q-2h =s. $$ 
Secondly, suppose $\eps = -1$. Hence $p-2t \leq 0$. A short calculation shows that if $q =2,$ then $t=\frac{p-1}{2} + 1$ and therefore $p - 2t = 1>0$, which is a contradiction. So therefore we conclude in this case that $q>2$. By Lemma~\ref{LemmaAboutTheSameSignaOfRAndS}, it follows that $q - 2h$ is nonzero. Since $\eps = -1$, it follows that $q-2h < 0$. So then
$$r=2t-p = 2\bigg(\frac{ph-1}{q}\bigg)-p = \frac{p}{q}(2h - q)- \frac{2}{q} > \frac{p}{q}(2h-q) -1 > (2h-q)-1 = s-1.$$
Since $r>s-1$ and both sides are integers, it follows that $r\ge s$. 

The case of $q>p$ follows analogously; we omit the details. 
%, and again start by assuming that $q-2a>0$, so that $p-2b>0$ also. Then 
%
%$$s= q-2a = q-2\frac{qb+1}{p} = \frac{q}{p} ( p-2b) - \frac{2}{p} \ge  \frac{q}{p} ( p-2b) - 1 >  ( p-2b) - 1 = r-1.$$
%
%Since $s>r-1$ it follows that $s\ge r$. Lastly, suppose that $q-2a<0$ so that $p-2b<0$ also. Then 
%
%$$s=2a-q=2\frac{qb+1}{p} -q = \frac{q}{p}(2b-p) + \frac{2}{p} >(2b-p) + \frac{2}{p}>2b-p=r.$$

%This proves the lemma for nonalternating torus knots. If $T(p,q)$ is alternating, say $(p,q) = (2n+1,2)$ for some $n\in \mathbb N$, then $(r,s) = (1,0)$, while if $(p,q) = (2,2n+1)$ then $(r,s) = (0,1)$, proving the lemma for all nontrivial torus knots.   
\end{proof}
%%%
%%%
\begin{remark}
The equality $r=s$ in the preceding lemma can only occur when $r=1=s$ (since $r$ and $s$ are still relatively prime integers) which, as we shall see shortly, %(paragraph before Theorem \ref{TheoremOnReversingPinchingMoves}), 
can only happen if $\{p,q\} = \{t, t\pm 2\}$ for some odd integer $t\ge 3$. 
\end{remark}

The above series of lemmas concentrated on the result of {\em one} pinch move. Now we turn to a bigger challenge of starting with a torus knot $T(p,q)$ and performing a series of pinch moves until the knot first becomes unknotted. We will encode the results of this process as follows: 
\begin{align}\label{ChainOfPinchMoves}
T(p,q) = T(p_n,q_n) \stackrel{\eps_n}{\longrightarrow} T(p_{n-1},q_{n-1}) \stackrel{\eps_{n-1}}{\longrightarrow} \dots \stackrel{\eps_2}{\longrightarrow} T(p_{1},q_{1}) \stackrel{\eps_1}{\longrightarrow} T(p_{0},q_0) = U.
\end{align}
%
%%%%%%%%%%%%%%%%
%%%%%%%%%%%%%%%%%
%%%%%%%%%%%%%%%%%

Our final goal in this section is to describe all torus knots $T(p,q)$ for which $n$ pinch moves will produce an unknot. As a warm-up exercise toward this goal, consider the case when $T(p,q) = T(p_1,q_1)$ reduces to the unknot $T(p_0,q_0)= U$ via a single $\eps = \eps_1$ pinch move. (Such knots must necessarily bound M\"obius bands in $B^4$.) 

We use the convention here that $p, q>1$ are relatively prime integers such that $q$ is odd and if $p$ is also odd, then $p>q$. Since $T(p_0,q_0)$ is an unknot, one of $p_0$ or $q_0$ must equal $1$. Because parity of parameters is preserved by a pinch move, $q_0$ must be odd, and if $p$ is also odd, then $p_0\geq q_0$ by Lemma~\ref{MagnitudeOfRandS}. Therefore $q_0=1$. Using this and Lemma \ref{AuxiliaryLemmaOnTheRelationBetweenPQRAndS}, we have
$$2\eps_1 = p_0q_1-q_0p_1 = p_0q_1-p_1.$$ 
Thus $p_1=p_0q_1-2\eps_1$, with $q_1\ge 3$, an odd integer. Observe that if $p_0 = 0$ or $1$, then it will follow that $\eps_1 = -1$ since $p_1$ is positive. 

In summary, if  some torus knot $T(p_1,q_1)$ reduces to an unknot $T(p_0,1)$ by an $\eps_1$-move, then $q_1$ is some odd integer with $q_1\ge 3$, and $p_1=p_0q_1-2\eps_1$. Moreover, these parameters satisfy the constraint that if $p_0 = 0$ or $1$, then it will follow that $\eps_1 = -1$.

Conversely, choose an odd integer $q_1\geq 3$, and choose any $p_0\geq 0$ and choose $\eps_1 \in \{\pm 1\}$ subject to the constraint that if $p_0 = 0$ or $1$, then $\eps_1 = -1$. Then the torus knot $T(p_1, q_1)$ where $p_1=p_0q_1-2\eps_1$ reduces to an unknot via an $\eps_1$-pinch move. Namely, set $h_1=\frac{1}{2}(q_1-\eps_1)$ and $t_1=\frac{1}{2}(p_1-\eps_1p_0)$, and note that $h_1, t_1\in \mathbb Z$. Then 
$$p_1h_1-q_1t_1 = \frac{1}{2} p_1 ( q_1-\eps_1) - \frac{1}{2}q_1(p_1-\eps_1p_0) = \frac{\eps_1}{2}(p_0q_1-p_1) =1.$$
Since $q_1\ge 3$, it follows that $1\le h_1 \le q_1-1$, and so Lemma \ref{AuxiliaryLemmaOnHowToFindTAndH} guarantees that $1\le t_1 \le p_1-1$, as well. This same lemma also shows that $T(p_1,q_1)$ reduces via an $\eps_1$ move to
$$T(\eps_1(p_1-2t_1), \eps_1(q_1-2h_1)) = T(\eps_1(p_1-(p_1-\eps_1p_0)),\eps_1(q_1-(q_1-\eps_1))) = T(p_0,1).$$ 

Summarizing this discussion, we proved that a nontrivial torus knot $T(p,q)$ (with $q$ odd and $p> q$ if $p$ is also odd) reduces to an unknot via a single pinch move if and only if $p=p_0q-2\eps_1$ for some odd $q\geq 3$ and for some $\eps_1\in\{\pm1\}$ and for some $p_0\geq 0$, where if $\eps_1 = 1$, then $p_0\geq 2$. Observe that this proves Proposition~\ref{TorusKnotsThatBoundMobiusBands} from the introduction. The next theorem generalizes this result. 
%%%
%%%
\begin{theorem}\label{TheoremOnReversingPinchingMoves}
Let $p,q>1$ be relatively prime integers with $q$ odd and with $p>q$ if $p$ is also odd. Suppose that $T(p,q)$ first reduces to the unknot after $n$ pinch moves. We keep track of the torus knots generated in this process and the sign of each pinch move as in \eqref{ChainOfPinchMoves} above:
$$T(p,q) = T(p_n,q_n) \stackrel{\eps_n}{\longrightarrow} T(p_{n-1},q_{n-1}) \stackrel{\eps_{n-1}}{\longrightarrow} \dots \stackrel{\eps_2}{\longrightarrow} T(p_{1},q_{1}) \stackrel{\eps_1}{\longrightarrow} T(p_{0},q_{0})=U$$
Then these integers generated in the process of pinch moves satisfy the following five identities and constraints: 
\begin{itemize}
\item[(a)] $q_0 = 1$. 
\item[(b)] $p_0\geq 0$, and if $p_0 = 0$ or $1$, then $\eps_1 = -1$. 
\item[(c)] The value of $q_1$ is odd and $q_1 \geq 3$. 
\item[(d)] $p_1 = p_0q_1 - 2\eps_1$. 
\item[(e)] There exist positive even integers $m_1, m_2, \ldots, m_{n-1}$ such that for all $k\in \{2, 3, \ldots, n\}$, 
\begin{align}\label{EquationSequenceOfPinchingMoves}
p_{k} & = m_{k-1}p_{k-1} - \eps_{k-1}\eps_k p_{k-2}, \cr
q_{k} & = m_{k-1}q_{k-1}-\eps_{k-1}\eps_k q_{k-2}.
\end{align}
\end{itemize}
%%%
Conversely, suppose that the pair $(p,q)$ is generated as follows. Choose values $p_0$ and $q_1$ and signs $\eps_1, \ldots, \eps_n$ satisfying constraints (b) and (c), and choose any set of positive even integers $m_1, \ldots, m_{n-1}$. These choices together with the identities (a), (d), and (e) determine pairs of relatively prime integers $(p_k,q_k)$ for $k\in \{0,1,\ldots, n\}$. Furthermore, a pinch move applied to $T(p_k,q_k)$ yields $T(p_{k-1},q_{k-1})$, and the pinch move has sign $\eps_{k}$. In particular, the torus knot $T(p,q) = T(p_n, q_n)$ will first reduce to the unknot after $n$ pinch moves.

In short, the correspondence between $(p,q)$ and the data set $\{ n, p_0, q_1, \{\eps_k\}_{k=1}^n, \{m_k\}_{k=1}^{n-1}\}$ subject to contraints (a)--(e) above, is a bijection. 
\end{theorem}
%%%
%%%
%
%\begin{remark}
%Observe that Theorem~\ref{TheoremOnReversingPinchingMoves} establishes a between such pairs $(p,q)$ of relatively prime integers as above, and collections of integers 
%
%$$n, p_0, q_1, \{m_k\}_{k=1}^{n-1}, \{\eps_k\}_{k=1}^n, $$
%
%with $n\in \mathbb N$, $p_0\ge0$, $q_1\ge 3$ odd, $m_k\ge 2$ even, and $\eps_k\in \{\pm 1\}$, subject to the additional constraint that $\eps_1=-1$ if $p_0=0$ or $p_0=1$. 
%\end{remark}
%
\begin{proof}
Let us start with a torus knot $T(p,q)$, where $p$ and $q$ satisfy the given conventions, and suppose that $T(p,q)$ first reduces to the unknot after $n$ pinch moves:
$$T(p,q) = T(p_n,q_n) \stackrel{\eps_n}{\longrightarrow} T(p_{n-1},q_{n-1}) \stackrel{\eps_{n-1}}{\longrightarrow} \dots \stackrel{\eps_2}{\longrightarrow} T(p_{1},q_{1}) \stackrel{\eps_1}{\longrightarrow} T(p_{0},q_{0})=U$$

The fact that $q_0=1$, $p_0\ge 0$, and $p_1 = p_0q_1-2\eps_1$ follows as in the \lq\lq warm-up exercise\rq\rq above. In addition, since parity is preserved by pinch moves, $q_1$ is odd, and since $T(p_1, q_1)$ is not the unknot, $q_1\geq 3$.

It only remains to show that if $n\ge 2$, the integers $p_k$ and $q_k$ satisfy the recursive relationships in~\eqref{EquationSequenceOfPinchingMoves}. We proceed by induction on $n$. 

For the base case ($n = 2$), suppose we have a pinch move sequence $T(p_2,q_2)\stackrel{\eps_2}{\longrightarrow}  T(p_1,q_1) \stackrel{\eps_1}{\longrightarrow} T(p_0, q_0) = U$. Substituting the known values for $q_0$ and $p_1$, we have:
$$T(p_2,q_2)\stackrel{\eps_2}{\longrightarrow}  T(p_0q_1 - 2\eps_1,q_1) \stackrel{\eps_1}{\longrightarrow} T(p_0, 1).$$
From  Lemma~\ref{AuxiliaryLemmaOnTheRelationBetweenPQRAndS}, we obtain 
$$(p_0q_1- 2\eps_1)q_2-q_1p_2 = 2 \eps_2$$
Modulo $q_1$ this yields $-2\eps_1q_2 \equiv 2\eps_2 (\text{mod }q_1)$, and since $q_1$ is odd, it follows that $q_2 \equiv - \eps_1\eps_2 (\text{mod }q_1)$. Write $q_2 = - \eps_1\eps_2+ m_1q_1$ for some integer $m_1$. Since both $q_1$ and $q_2$ are odd, $m_1$ is even, and since $q_2>1$, it follows that $m_1 \ge 2$.  Substituting this newfound equation for $q_2$ into the above expression we have
\begin{align*}
(p_0q_1-  2\eps_1)(-\eps_1\eps_2+m_1q_1)-q_1 p_2 & = 2\eps_2\cr
-\eps_1\eps_2p_0q_1+m_1p_0q_1^2+2\eps_2-  2\eps_1m_1q_1 -q_1p_2 & = 2\eps_2 \cr 
q_1(-\eps_1\eps_2p_0 +m_1p_0q_1 - 2\eps_1m_1-p_2) & = 0\cr
-\eps_1\eps_2p_0 +m_1p_0q_1 - 2\eps_1m_1 & = p_2.
\end{align*}
This formula for $p_2$ can be rewritten as follows:
\begin{align*} 
p_2 & =-\eps_1\eps_2p_0 +m_1p_0q_1 - 2\eps_1m_1  \cr
& = m_1(p_0q_1-2\eps_1) -\eps_1 \eps_2p_0 \cr
& =  {m_1 p_1 -\eps_1\eps_2p_0}.
\end{align*}
In conclusion, we find that there exists a positive even integer $m_1$ such that 
$$(p_2,q_2) = m_1\cdot (p_1,q_1) - \eps_1\eps_2(p_0,1),$$
which proves the recursive relationships in~\eqref{EquationSequenceOfPinchingMoves} for the base case $n = 2$. 

Now suppose that there is some integer $k$ with $2\leq k <n$ such that the recursive relationship~\eqref{EquationSequenceOfPinchingMoves} holds for all $j\leq k$. In particular, there exist positive even integers $m_1, m_2, \ldots, m_{k-1}$ such that for all $j \in \{2, 3, \ldots, k\}$, 
$$(p_{j},q_{j}) = m_{{j}-1}\cdot (p_{{j}-1},q_{j-1}) - \eps_{j-1}\eps_{j}(p_{j-2},q_{j-2}).$$
We want to show that such a relationship also holds for $j = k +1$. 

%From  Lemma~\ref{AuxiliaryLemmaOnTheRelationBetweenPQRAndS}, we obtain 
%
%$$p_{k_0}q_{k_0+1}-q_{k_0}p_{k_0+1} = 2 \eps_{k_0+1}$$ 

To make the inductive step run similarly to the base step, we introduce another object. For $j \in \{0, 1, 2, \ldots, k\}$, define $r_j\in \mathbb Z$ by 
\begin{equation} \label{DefinitionOfRn}
r_j =\frac{1}{2}\eps_1(p_0q_j-p_j). 
\end{equation}
Note that $r_j$ is an integer since $q_j$ is always odd, while $p_0$ and $p_j$ are of the same parity. Since $r_j$ is a linear combination of $p_j$ and $q_j$, and each of the latter two satisfies the same recursive relation \eqref{EquationSequenceOfPinchingMoves} for $j \in\{2, 3, \ldots, k\}$, it follows that $r_j$ also satisfies this relation, namely 
$$r_j = m_{j-1}r_{j-1}- \eps_{j-1}\eps_j r_{j-2}, \qquad \qquad j=2, \dots, k.$$
By way of example, $r_0 = 0$, $r_1 = 1$ and $r_2 = m_1$. For any $j=1,\dots, k$ we obtain    
\begin{align} \label{AuxiliaryRelationBetweenQnAndRn}
r_jq_{j-1} - r_{j-1}q_j  & = \textstyle \frac{\eps_1}{2}(p_0q_j-p_j)q_{j-1} - \frac{\eps_1}{2} (p_0q_{j-1} - p_{j-1})q_j  \cr 
& = \textstyle \frac{\eps_1}{2} ( p_{j-1}q_j - q_{j-1}p_j )  \cr
& = \eps_1\eps_j.
\end{align}
In the above, we relied on the result of Lemma \ref{AuxiliaryLemmaOnTheRelationBetweenPQRAndS} by which $p_{j-1}q_j-q_{j-1}p_j=2\eps_j$ provided $T(p_j,q_j)$ reduces to $T(p_{j-1},q_{j-1})$ via an $\eps_j$-move.  Relying on this same lemma once more, in the case that $j = k$, we have: 
$$ p_{k} q_{k+1} - q_{k} p_{k+1} = 2\eps_{k+1}. $$
Writing $p_{k} = p_0q_{k} - 2\eps_1 r_{k}$ gives
$$ (p_0q_{k}-2\eps_1 r_{k}) q_{k+1} - q_{k} p_{k+1} = 2\eps_{k+1}.$$
Reducing this equation modulo $q_{k}$ yields 
\begin{align*}
-2 \eps_1 r_{k}q_{k+1} & \equiv 2\eps_{k+1}\, (\text{mod }q_{k}) \qquad (\text{ multipy by } q_{k-1})\cr 
-2\eps_1 q_{k-1} r_{k}q_{k+1} & \equiv 2\eps_{k+1}q_{k-1} \, (\text{mod }q_{k}) \qquad (\text{ use \eqref{AuxiliaryRelationBetweenQnAndRn}}) \cr 
-2\eps_1(r_{k-1}q_{k}+\eps_1\eps_{k} ) q_{k+1} & \equiv 2\eps_{k+1}q_{k-1} \, (\text{mod }q_{k}) \cr
-2\eps_{k} q_{k+1} & \equiv 2\eps_{k+1}q_{k-1} \, (\text{mod }q_{k}) \qquad (q_{k} \text{ is odd, so divide by -2})\cr
q_{k+1} & \equiv -\eps_{k}\eps_{k+1}q_{k-1} \, (\text{mod }q_{k}) \cr
q_{k+1} & =m_{k}q_{k} -\eps_{k}\eps_{k+1}q_{k-1} \qquad \text{(for some integer $m_{k}$)}.  
\end{align*}
Since $q_{k-1}$, $q_{k}$ and $q_{k+1}$ all have the same parity, it follows that $m_{k}$ must be even. Moreover, the double inequality $q_{k-1}<q_{k}<q_{k+1}$ forces $m_{k}\ge 2$. Hence $q_{k+1}$ satisfies the recursive relationship \eqref{EquationSequenceOfPinchingMoves}, and it only remains to show that $p_{k+1}$ does, as well. We return to the equation $ p_{k} q_{k+1} - q_{k} p_{k+1} = 2\eps_{k+1}$ and substitute our newly found formula for $q_{k+1}$. (In the third line below, we rely on \eqref{AuxiliaryRelationBetweenQnAndRn}.)
\begin{align*}
(p_0q_k-2\eps_1r_k)(m_{k}q_k-\eps_{k}\eps_{k+1} q_{k-1})-q_kp_{k+1} & = 2\eps_{k+1}\cr
q_k(p_0m_{k}q_k -\eps_{k}\eps_{k+1}p_0q_{k-1} -2\eps_1r_k m_{k}  -p_{k+1}) & = 2\eps_{k+1}-2\eps_1\eps_{k}\eps_{k+1}r_kq_{k-1} \cr
q_k(p_0m_{k}q_k -\eps_{k}\eps_{k+1}p_0q_{k-1} -2\eps_1r_k m_{k}  -p_{k+1}) & = 2\eps_{k+1}-2\eps_1\eps_{k}\eps_{k+1}(r_{k-1}q_{k}+\eps_1\eps_{k}) \cr
q_k(p_0m_{k}q_k -\eps_{k}\eps_{k+1}p_0q_{k-1} -2\eps_1r_k m_{k}  -p_{k+1}) & = -2\eps_1\eps_{k}\eps_{k+1}r_{k-1}q_{k}. 
\end{align*}
%%%
%%%
Dividing both sides by $q_k$ leads to 
%%%
%%%
\begin{align*}
p_{k+1} & = p_0m_{k}q_k -\eps_{k}\eps_{k+1}p_0q_{k-1} -2\eps_1r_k m_{k}+2\eps_1\eps_{k}\eps_{k+1}r_{k-1}   \cr
& = m_k(p_0q_k -2\eps_1r_k) - \eps_k\eps_{k+1}(p_0q_{k-1} - 2\eps_1r_{k-1})\cr
& = m_k p_k - \eps_k \eps_{k+1} p_{k-1}
\end{align*}
completing the induction. \\
%%%%%%%%%%%%%%%%%%%%%%%%%%%%%%%%%%%%%

Now we turn to prove the second half of the theorem. Choose integers $p_0\ge 0$, $q_1\ge 3$ odd, and a choose signs $\eps_1, \ldots, \eps_n$. Moreover, choose positive even integers $m_k$, for $k=1, \dots, n-1$. Set $q_0=1$, $p_1=p_0q_1-2\eps_1$ and use \eqref{EquationSequenceOfPinchingMoves} to create the sequence of integer pairs $(p_k,q_k)$ for $k=0,\dots, n$.

We need to show that $T(p_{k-1},q_{k-1})$ arises from $T(p_k,q_k)$ via an $\eps_k$--move. For $n=1$ this was already shown in the \lq\lq warm-up exercise\rq\rq  preceding this proof. We can do something very similar in general. For $n\ge 2$ and $1<k\le n$, define the integers $h_k$ and $t_k$ as
$$h_k = \frac{1}{2} (q_k-\eps_kq_{k-1}) \qquad \text{ and } \qquad t_k =  \frac{1}{2} (p_k-\eps_kp_{k-1}).$$ 
Then
$$p_kh_k-q_kt_k = \frac{1}{2}p_k(q_k-\eps_kq_{k-1}) - \frac{1}{2}q_k(p_k-\eps_kp_{k-1}) = \frac{\eps_k}{2} (q_kp_{k-1} - p_kq_{k-1}) = 1.  $$
Since $q_k>q_{k-1}$ it follows that $h_k\in \{1,\dots, q_k-1\}$ and so Lemma \ref{AuxiliaryLemmaOnHowToFindTAndH} implies that an $\eps_k$--move turns the torus knot $T(p_k,q_k)$ into the torus knot 
$$T(\eps_k(p_k-2t_k), \eps_k(q_k-2h_k)) = T(p_{k-1}, q_{k-1}), $$
as desired. Observe that by our choice of $q_1\geq 3$, it follows that none of the torus knots are unknotted except for $T(p_0, q_0) = T(p_0, 1)$. Therefore $T(p,q) = T(p_n, q_n)$ is first unknotted after $n$ pinch moves. This completes the proof of the theorem. 
\end{proof}
%%%%%%%%%%%%%%
%%%%%%%%%%%%%%%
%%%%%%%%%%%%%%%%%
%%%%%%%%%%%%%%%%
%%%
%%%
%%%%%%%%%%%%%%%%%%%%%%%%%%%%%%%%%%%%%%%%%%%%%%%%%%%%%%%%%%%%%%%%%%%%%%%%%%%%%%%%%%%%%%%%%%%%%%%%%%%%%%%%%%%%%%%%%%%%%%%%%%%%%%%%%%%%%%%%%%%%%%%%%%%%
%%%%%%%%%%%%%%%%%%%%%%%%%%%%%%%%%%%%%%%%%%%%%%%%%%%%%%%%%%%%%%%%%%%%%%%%%%%%%%%%%%%%%%%%%%%%%%%%%%%%%%%%%%%%%%%%%%%%%%%%%%%%%%%%%%%%%%%%%%%%%%%%%%%%
%%%%%%%%%%%%%%%%%%%%%%%%%%%%%%%%%%%%%%%%%%%%%%%%%%%%%%%%%%%%%%%%%%%%%%%%%%%%%%%%%%%%%%%%%%%%%%%%%%%%%%%%%%%%%%%%%%%%%%%%%%%%%%%%%%%%%%%%%%%%%%%%%%%%
%%%%%%%%%%%%%%%%%%%%%%%%%%%%%%%%%%%%%%%%%%%%%%%%%%%%%%%%%%%%%%%%%%%%%%%%%%%%%%%%%%%%%%%%%%%%%%%%%%%%%%%%%%%%%%%%%%%%%%%%%%%%%%%%%%%%%%%%%%%%%%%%%%%%
%%%%%%%%%%%%%%%%%%%%%%%%%%%%%%%%%%%%%%%%%%%%%%%%%%%%%%%%%%%%%%%%%%%%%%%%%%%%%%%%%%%%%%%%%%%%%%%%%%%%%%%%%%%%%%%%%%%%%%%%%%%%%%%%%%%%%%%%%%%%%%%%%%%%
\section{Computing $\upsilon$ for torus knots} \label{SectionOnComputingTheUpsilonInvariant}
%%%%%%%%%%%%%%%%%%%%%%%%%%%%%%%%%%%%%%%%%%%%%%%%%%%%%%%%%%%%%%%%%%%%%%%%%%%%%%%%%%%%%%%%%%%%%%%%%%%%%%%%%%%%%%%%%%%%%%%%%%%%%%%%%%%%%%%%%%%%%%%%%%%%
%%%%%%%%%%%%%%%%%%%%%%%%%%%%%%%%%%%%%%%%%%%%%%%%%%%%%%%%%%%%%%%%%%%%%%%%%%%%%%%%%%%%%%%%%%%%%%%%%%%%%%%%%%%%%%%%%%%%%%%%%%%%%%%%%%%%%%%%%%%%%%%%%%%%
\subsection{Definition of $\mathbf{\upsilon}$, and a recursion formula}
%%%
%%%
Let $\mathcal C$ denote the {\em smooth concordance group} of oriented knots in $S^3$. Its elements are equivalence classes $[K]$ indexed by oriented knots $K$ in $S^3$, with two knots $K_0$ and $K_1$ being equivalent if there exists a smooth embedding $\varphi :S^1\times [0,1]\to S^3\times [0,1]$, such that $\varphi|_{S^1\times \{i\}}=K_i$ for $i=0,1$. The operation on $\mathcal C$ is given by connect summing knots, the identity element is the equivalence class of the unknot.  

Using a modified differential on the knot Floer chain complex, Ozsv\'ath, Stipsicz and Szab\'o \cite{OSSUpsilon:2017} defined a function $\Upsilon :\mathcal C\times [0,2]\to \mathbb R$, with the property that its restriction to a slice $\mathcal C\times \{t\}$ of the domain (for any $t\in [0,2]$), is a group homomorphism $\Upsilon (t):\mathcal C\to \mathbb R$. For a knot $K$, the value of $\Upsilon _K(t)$ is the analogue of what the Heegaard Floer knot invariant $\tau(K)$ \cite{OSTau:2003} is for the original differential on the knot Floer chain complex. 

The concordance homomorphism $\upsilon :\mathcal C\to \mathbb Z$ introduced by Ozsv\'ath, Stipsicz and Szab\'o \cite{OSS:2017}, is defined as 
$$\upsilon (K) = \Upsilon _K(1). $$
The fact that its image lies in $\mathbb Z$ follows easily from the general property of $\Upsilon$ by which $\Upsilon _K(\frac{m}{n}) \in \frac{1}{n}\mathbb Z$, see \cite{OSSUpsilon:2017}. The invariant $\upsilon$ is simply referred to as the {\em upsilon invariant of $K$}, and is distinguished from $\Upsilon _K(t)$ by referring to the latter as the {\em large upsilon invariant of $K$}. 

Feller and Krcatovich proved the following formulas for $\Upsilon$ for torus knots (Propositions 2.1 and 2.2 in \cite{FellerKrcatovich:2017}):
\begin{align} \label{FellerKrcatovichRecursiveFormulas}
\Upsilon_{T(n,n+1)}(t) & \textstyle = -i(i+1) -\frac{1}{2}n(n-1-2i)t,   \cr
\Upsilon _{T(a,b)}(t) & = \Upsilon_{T(a,b-a)}(t) + \Upsilon_{T(a,a+1)}(t).
\end{align}  
The first formula is valid for any $n\in \mathbb N$ and any choice of $t\in [\frac{2i}{n},\frac{2i+1}{n}]$, while the second formula holds for relatively prime integers $a, b$ with $0<a<b$. A combination of the two easily leads to our first result on $\upsilon$. 
%%%
%%%
\begin{lemma} \label{LemmaRecursionFormulaForUpsilon}
Let $0<a<b$ be relatively prime integers, then 
\begin{equation} \label{EquationRecursionFormulaForUpsilon}
\upsilon (T(a,b)) = \left\{
\begin{array}{ll}
\upsilon (T(a,b-a)) - \frac{1}{4}a^2 & \quad ; \quad a \text{ is even}, \cr 
& \cr
\upsilon (T(a,b-a)) - \frac{1}{4}(a^2-1) & \quad ; \quad a \text{ is odd}. \cr 
\end{array}
\right. 
\end{equation}
We shall refer to this formula as the \lq\lq Recursive Formula for $\upsilon$.\rq\rq
\end{lemma}
%%%
%%%
\begin{proof}
If $a$ is even, the first formula in \eqref{FellerKrcatovichRecursiveFormulas} with the choices of $n=a$, $i=\frac{a}{2}$ and $t=1$, yields
$$\textstyle \upsilon(T(a,a+1)) = \Upsilon_{T(a,a+1)}(1) = -\frac{a}{2}(\frac{a}{2}+1) -\frac{1}{2}a(a-1-a) = - \frac{a^2}{4}.$$
If $a$ is odd, the same formula but with the choices of $n=a$, $i=\frac{a-1}{2}$ and $t=1$, yields
$$\textstyle \upsilon(T(a,a+1)) = \Upsilon_{T(a,a+1)}(1) = -\frac{a-1}{2}(\frac{a-1}{2}+1) -\frac{1}{2}a(a-1-(a-1)) = - \frac{a^2-1}{4}.$$
Using the second formula in \eqref{FellerKrcatovichRecursiveFormulas} with $t=1$ leads to 
$$\upsilon (T(a,b)) = \Upsilon_{T(a,b)}(1) =  \Upsilon_{T(a,b-a)}(1) + \Upsilon_{T(a,a+1)}(1) = \upsilon(T(a,b-a)) + \upsilon (T(a,a+1)).$$
Together, these computations prove the claimed recursive formula. 
\end{proof}
%%%
%%%
The recursion formula for $\upsilon$ from the preceding lemma is the main tool for computing $\upsilon$ for torus knots $T(p,q)$. One may suspect that given a torus knot $T(p,q)$, and having formed the sequence of torus knots $T(p_k,q_k)$ using pinch moves as in \eqref{EquationSequenceOfPinchingMoves}, one may be able to compute $\upsilon( T(p_k,q_k))$ from knowing $\upsilon (T(p_{k-1},q_{k-1}))$. We have not been able to do this, as the recursion formula for $\upsilon $ from Lemma \ref{LemmaRecursionFormulaForUpsilon} is sufficiently  dissimilar from the relation \eqref{Definition Of (r,s) From (p,q)} (with $(p,q) = (p_k,q_k)$ and $(r,s) = (p_{k-1},q_{k-1})$) which relates $(p_k,q_k)$ to $(p_{k-1}, q_{k-1})$. Instead, we resort to defining a new, doubly indexed sequence of pairs of relatively prime integers, that in a way interpolates between \eqref{EquationRecursionFormulaForUpsilon} and \eqref{Definition Of (r,s) From (p,q)}.  
%%%%%%%%%%%%%%%%%%%%%%%%%%%%%%%%%%%%%%%%%%%%%%%%%%%%%%%%%%%%%%%%%%%%%%%%%%%%%%%%%%%%%%%%%%%%%%%%%%%%%%%%%%%%%%%%%%%%%%%%%%%%%%%%%%%%%%%%%%%%%%%%%%%%
%%%%%%%%%%%%%%%%%%%%%%%%%%%%%%%%%%%%%%%%%%%%%%%%%%%%%%%%%%%%%%%%%%%%%%%%%%%%%%%%%%%%%%%%%%%%%%%%%%%%%%%%%%%%%%%%%%%%%%%%%%%%%%%%%%%%%%%%%%%%%%%%%%%%
\subsection{A doubly indexed sequence associated to $\mathbf{(p,q)}$}
%%%
%%%
Let $p,q>1$ be relatively prime integers, and let $n$, $p_0$, $q_1$, $\{m_k\}_{k=1}^{n-1}$ and $\{\eps_k\}_{k=1}^n$ be obtained from $(p,q)$ as in Theorem \ref{TheoremOnReversingPinchingMoves}. 
We then define a doubly indexed sequence $\rho_{k,n}$ of integers for $k\in \mathbb N$ and $n\ge k-1$, as follows:
\begin{align} \label{DefinitionOfRhoKCommaN}
\rho_{1,n} &  = \frac{1}{2}\eps_1 (p_0q_n-p_n), & \quad &  n\ge 0,\cr & \cr
\rho _{2,n} & = \eps_1\eps_2( q_1\rho_{1,n} -q_n ), & \quad & n\ge 1,  \cr & \cr
\rho_{k,n}  & = \eps_{k-1}\eps_k (m_{k-2}\rho_{k-1,n} - \rho_{k-2,n}) &\quad  &  k \ge 3 \text{ and } n\ge k-1.  
\end{align}
%%%
%%%
For simplicity of notation we shall use $r_n$ and  $s_n$ to denote $\rho_{k,n}$ for $k=1,2$ respectively, that is 
\begin{equation} \label{EquationDefinitionOfRnAndSn}
r_n= \rho_{1,n} = \frac{1}{2}\eps_1(p_0q_n-p_n), \qquad \qquad s_n = \rho_{2,n} = \eps_1\eps_2(q_1r_n-q_n).
\end{equation}
Our definition of $r_n$ here agrees with the use of $r_n$ in the proof of Theorem \ref{TheoremOnReversingPinchingMoves}. %The integers $\rho_k$ appearing in Theorem \ref{TheoremOnAlterantiveWayOfComputingP0Q1Etc} shall be revealed to equal $\rho_{k,n}$ (see Section \ref{SectionOnTheProofOfAlternativelyComputingP0Etc}).  
%%%
%%%
\begin{theorem} \label{TheoremSomePropertiesOfRhoKN}
For all $k\ge 1$ we obtain 
\begin{equation} \label{EquationInitialValuesOfRhoK}
\rho _{k,k-1} = 0, \qquad \rho _{k,k} = 1, \qquad \rho _{k,k+1} =  m_{k}
\end{equation}
Moreover, the inequalities $\rho_{k,n-1} < \rho_{k,n}$ and $\rho_{k,n} < \rho_{k-1,n}$ hold for all $k$ and $n$ for which the terms appearing in them are defined.   
\end{theorem}
%%%
%%%
\begin{proof}
Since $p_n$ and $q_n$ satisfy the recursive relations 
$$p_n = m_{n-1}p_{n-1}-\eps_{n-1}\eps_n p_{n-2} \qquad \text{ and } \qquad q_n=m_{n-1}q_{n-1} - \eps_{n-1}\eps_n q_{n-2},$$ 
and since $\rho_{1,n}$ is a linear combination of $p_n$ and $q_n$, then $\rho_{1,n}$ satisfies the same recursive relation, that is $\rho_{1,n} = m_{n-1}\rho_{1,n-1}-\eps_{n-1}\eps_n\rho_{1,n-2}$ for all $n\ge 2$. Next, since $\rho_{2,n}$ is a linear combination of $q_n$ and $\rho_{1,n}$, then $\rho_{2,n}$ satisfies the same recursive relation also, that is $\rho_{2,n} = m_{n-1}\rho_{2,n-1}-\eps_{n-1}\eps_n\rho_{2,n-2}$ for all $n\ge 3$. Continuing this argument inductively shows that 
\begin{equation} \label{EquationTheSecondRecursionForRho}
\rho_{k,n} = m_{n-1}\rho_{k,n-1}-\eps_{n-1}\eps_n\rho_{k,n-2}, \qquad \text{ for all } n \ge k+1.
\end{equation}
The values for $r_n = \rho_{1,n}$ and $s_n= \rho_{2,n}$ for $n$ small can be explicitly computed, and are given by 
$$\begin{array}{lclcl}
r_0 =0          &  \qquad &               \cr
r_1 =1          &  \qquad & s_1 = 0       \cr
r_2 =m_1        &  \qquad & s_2 = 1     \cr
r_3 = m_1m_2-\eps_2\eps_3   &  \qquad & s_3= m_2  \cr
\end{array}$$
We shall prove the relations \eqref{EquationInitialValuesOfRhoK} by induction on $k$, to which the preceding formulas represent the base. Suppose that \eqref{EquationInitialValuesOfRhoK} has been shown to hold for all $k\le k_0$ for some $k_0\ge 2$. From \eqref{EquationTheSecondRecursionForRho} we obtain
\begin{align} \label{EquationOtherValuesOfRho}
\rho_{k_0-1, k_0+1} & = m_{k_0} \rho_{k_0-1,k_0} -\eps_{k_0}\eps_{k_0+1} \rho_{k_0-1,k_0-1} = m_{k_0} m_{k_0-1} - \eps_{k_0}\eps_{k_0+1},\cr 
\rho_{k_0-1, k_0+2} & = m_{k_0+1} \rho_{k_0-1,k_0+1} - \eps_{k_0+1}\eps_{k_0+2} \rho_{k_0-1,k_0} \cr
& = m_{k_0+1} m_{k_0} m_{k_0-1} - \eps_{k_0}\eps_{k_0+1}m_{k_0+1} - \eps_{k_0+1}\eps_{k_0+2}  m_{k_0-1}.\cr 
\end{align}
From these we find 
\begin{align*} 
\rho_{k_0+1,k_0} & = \eps_{k_0}\eps_{k_0+1}( m_{k_0-1}\rho_{k_0,k_0} - \rho_{k_0-1,k_0})\cr
& = \eps_{k_0}\eps_{k_0+1}(m_{k_0-1} - m_{k_0-1}) \cr
& = 0, \cr & \cr  
\rho_{k_0+1,k_0+1} & = \eps_{k_0}\eps_{k_0+1}(m_{k_0-1}\rho_{k_0,k_0+1} - \rho_{k_0-1,k_0+1}) \cr
& = \eps_{k_0}\eps_{k_0+1}(  m_{k_0-1}m_{k_0}  -(m_{k_0} m_{k_0-1} - \eps_{k_0}\eps_{k_0+1}) )  \cr
& = 1, \cr & \cr
\rho_{k_0+1,k_0+2} & = \eps_{k_0}\eps_{k_0+1}(m_{k_0-1}\rho_{k_0,k_0+2} - \rho_{k_0-1,k_0+2}) \cr 
& = \eps_{k_0}\eps_{k_0+1} [  m_{k_0-1}(m_{k_0+1}m_{k_0} -\eps_{k_0+1}\eps_{k_0+2})   \cr
& \qquad \qquad -(m_{k_0+1} m_{k_0} m_{k_0-1} - \eps_{k_0}\eps_{k_0+1} m_{k_0+1} -\eps_{k_0+1}\eps_{k_0+2}  m_{k_0-1}) ]\cr
& = m_{k_0+1}, 
\end{align*}
completing the proof of \eqref{EquationInitialValuesOfRhoK}. 

Next we turn to the inequality $\rho_{k,n-1} < \rho_{k,n}$ which we seek to prove for all $k\ge 1$ and all $n\ge k-1$. It is easy to see that the inequality holds true for $n=k, k+1$ given the values of $\rho$ from \eqref{EquationInitialValuesOfRhoK} and the fact that $m_k\ge 2$. We take this as the basis of induction on $n$, with $k$ fixed but arbitrary. To continue the inductive reasoning, assume that $\rho_{k,n-1} < \rho_{k,n}$ for all $k\le n \le n_0$ for some $n_0>k$. Then by relying on \eqref{EquationTheSecondRecursionForRho} we obtain $\rho_{k,n_0+1}  = m_{n_0} \rho_{k,n_0} - \eps_{n_0-1}\eps_{n_0}\rho_{k,n_0-1}$. If $\eps_{n_0-1}\eps_{n_0}=-1$ if follows that $\rho_{k,n_0+1} > m_{n_0} \rho_{k,n_0} >\rho_{k,n_0}$ proving the claim. If $\eps_{n_0-1}\eps_{n_0}=1$ then 
\begin{align*}
\rho_{k,n_0+1} & = m_{n_0} \rho_{k,n_0} - \rho_{k,n_0-1} \cr
& = (m_{n_0}-1)\rho_{k,n_0} + (\rho_{k,n_0}-\rho_{k,n_0-1}) \cr
& >  (m_{n_0}-1)\rho_{k,n_0} \cr
& \ge \rho_{k,n_0}
\end{align*} 
completing the inductive proof of $\rho_{k,n-1}<\rho_{k,n}$ for all $k\ge1$ and all $n\ge k-1$ . 

We now turn to the inequality $\rho_{k,n}< \rho_{k-1,n}$ valid for $k\ge 2$ and $n\ge k-1$. We shall prove this inequality by proving, inductively on $n$, the inequality 
$$\rho_{k-1,n} - \rho_{k,n}\ge \rho_{k-1,n-1} - \rho_{k,n-1}.$$
This inequality is easily seen to be valid for $n=k, k+1$ given the values of $\rho$ from \eqref{EquationInitialValuesOfRhoK} and \eqref{EquationOtherValuesOfRho}, thus proving the basis of induction. For the step of the  induction, assume that $\rho_{k-1,n} - \rho_{k,n}\ge \rho_{k-1,n-1} - \rho_{k,n-1}$ for all $n$ with  $k\le n\le n_0$ for some $n_0>k-1$. In particular then $\rho_{k-1,n} - \rho_{k,n}\ge \rho_{k-1,k-1} - \rho_{k,k-1} = 1 >0$ for all $n\le n_0$. The following argument relies on the recursive relation \eqref{EquationTheSecondRecursionForRho}.
\begin{align*}
\rho_{k-1,n_0+1} - \rho_{k,n_0+1}  & = [m_{n_0}\rho_{k-1,n_0} -\eps_{n_0}\eps_{n_0+1}  \rho_{k-1,n_0-1}] - [m_{n_0}\rho_{k,n_0} - \eps_{n_0}\eps_{n_0+1}  \rho_{k,n_0-1}]  \cr
& =m_{n_0}(\rho_{k-1,n_0} - \rho_{k,n_0})-\eps_{n_0}\eps_{n_0+1} ( \rho_{k-1,n_0-1}- \rho_{k,n_0-1}).
\end{align*}
If $\eps_{n_0}\eps_{n_0+1}=-1$ then the above show that 
$$\rho_{k-1,n_0+1} - \rho_{k,n_0+1} > m_{n_0}(\rho_{k-1,n_0} - \rho_{k,n_0}) > \rho_{k-1,n_0} - \rho_{k,n_0}.$$
If $\eps_{n_0}\eps_{n_0+1}=1$ then
\begin{align*}
\rho_{k-1,n_0+1} - \rho_{k,n_0+1} & =m_{n_0}(\rho_{k-1,n_0} - \rho_{k,n_0})-\eps_{n_0}\eps_{n_0+1} ( \rho_{k-1,n_0-1}- \rho_{k,n_0-1}) \cr
& = (m_{n_0}-1) (\rho_{k-1,n_0} - \rho_{k,n_0})+ [ (\rho_{k-1,n_0} - \rho_{k,n_0}  ) -(\rho_{k-1,n_0-1}  - \rho_{k,n_0-1}  ) ]  \cr
& \ge (m_{n_0}-1) (\rho_{k-1,n_0} - \rho_{k,n_0}) \cr
& \ge \rho_{k-1,n_0} - \rho_{k,n_0}.
\end{align*} 
If follows that $\rho_{k-1,n} - \rho_{k,n} \ge \rho_{k-1,k-1} - \rho_{k,k-1} = 1>0$ for all $k\ge 2$ and all $n\ge k-1$, and thus that $\rho_{k-1,n}>\rho_{k,n}$, as claimed.  
\end{proof}
%%%
%%%
\begin{corollary} \label{CorollaryRelativeSizeOfQnAndRn}
For all $n\ge 1$, $2r_n<q_n$. 
\end{corollary}
%%%
%%%
\begin{proof}
Theorem \ref{TheoremSomePropertiesOfRhoKN} tells us that $\rho_{k,n} < \rho_{k-1,n}$, which for $k=2$ gives $s_n<r_n$. Recall that $q_n = q_1r_n-\eps s_n$ where $\eps = \eps_1\eps_2$, that $q_1\ge 3$ is odd, and that $r_n>0$ for all $n\ge 1$. If $\eps = -1$ then 
$$ q_n  = q_1r_n + s_n  \ge q_1r_n > 2r_n.$$
If $\eps=1$ then 
\begin{align*}
q_n & = q_1r_n -  s_n  \ge (q_1-1) r_n +(r_n-s_n) > (q_1-1)r_n \ge 2r_n.
\end{align*} 
\end{proof}
\begin{remark} \label{RemarkAboutTheParitiesAndRelativeSizesOfRhoKN}
The use of the recursive formula for $\upsilon$ found in Lemma~\ref{LemmaRecursionFormulaForUpsilon} for computing $\upsilon(T(a,b))$ is sensitive to the relative size of $a$ and $b$. In Section \ref{SectionForSignatureComputations}, we shall use a similar recursive relation to compute the $\sigma(T(a,b))$ (Equation \eqref{EquationRecursiveFormulaForSignature}) which, aside from depending on the relative size of $a$ and $b$, is also sensitive to their parities. These two recursive relations shall be applied to calculate $\upsilon$ and $\sigma$ of the following torus knots:  
\begin{align*}
&T(aq_n-2\eps_1r_n, q_n) \quad (\text{with } a=0, 1, 2), \cr
&T(\rho_{k,n}+\rho_{k+1,n}, 2\rho_{k+1,n}), \cr
&T(\rho_{k,n}+\rho_{k+1,n}, \rho_{k,n} - \rho_{k+1,n}).
\end{align*}  
The relative sizes of these entries are controlled by the inequalities from Theorem \ref{TheoremSomePropertiesOfRhoKN} and from Corollary \ref{CorollaryRelativeSizeOfQnAndRn}. The parity of $\rho_{k,n}$ with $k$ fixed, only depends on the parity of $n$ according to \eqref{EquationTheSecondRecursionForRho} (since $m_{n-1}$ is always even). Since $\rho_{k,k-1}=0$ and $\rho_{k,k}=1$, it follows that the parity of $\rho_{k,n}$ agrees with that of $n-k$. In particular, $\rho_{k,n} \pm \rho_{k+1,n}$ is always odd. We shall rely on these properties for the remainder of this section and again in the next section, without further explicit mention.     
\end{remark}

%%%%%%%%%%%%%%%%%%%%%%%%%
%%%%%%%%%%%%%%%%%%%%%%%%%
%%%%%%%%%%%%%%%%%%%%%%%%%
\subsection{Computing upsilon for torus knots} \label{SubSectionComputingUpsilonForAllTorusKnots}
%%%
%%%
Before delving into detail rich computations, we pause for a bit to outline our strategy of computing $\upsilon(T(p,q))$. Our calculations run through three stages of successive simplification of the knots involved in the computations, starting with $T(p,q)$ and ending up with the unknot. At every stage our chief tool is the recursive formula \eqref{EquationRecursionFormulaForUpsilon} for $\upsilon$. To outline the three stages, let the doubly indexed sequence $\rho_{k,n}$ be associated to a pair of relatively prime integers $p,q>1$ as in the preceding section. 

In stage 1 we use \eqref{EquationRecursionFormulaForUpsilon} to compute the $\upsilon(T(p,q)) = \upsilon(T(p_n,q_n))$ in terms of
\begin{align} \label{EquationCasesForTheFirstReductionStepInTheIntroduction}
\upsilon(T(2q_n-2\eps_1r_n,q_n)) & \quad \text{ if $p_0\ge 2$ is even}, \cr
\upsilon(T(q_n-2\eps_1r_n,q_n)) & \quad \text{ if $p_0\ge 1$ is odd}, \cr
\upsilon(T(2r_n,q_n)) & \quad \text{ if $p_0=0$}. \cr
\end{align}  
This is accomplished in Proposition \ref{PropositionReducingGeneralPNCaseToMultipleOfP0} below. 

In stage 2 we reduce the computation of $\upsilon (T(aq_n-2\eps_1r_n,q_n))$ (with $a=0$, 1 or 2, depending on which case occurred in \eqref{EquationCasesForTheFirstReductionStepInTheIntroduction}) to that of 
\begin{align}
\upsilon(T(r_n+s_n,2s_n)) & \quad \text{ if $p$ is even}, \cr
\upsilon(T(r_n+s_n,r_n-s_n)) & \quad \text{ if $p$ is odd.}
\end{align}
In the computation of $\upsilon$ alone one may choose to reduce $\upsilon (T(aq_n-2\eps_1r_n,q_n))$ to either $\upsilon (T(r_n+s_n,2s_n))$ or $\upsilon (T(r_n+s_n,r_n-s_n))$, and be able to continue to stage 3. The noted parity cases in the formula above however need to be followed later when we get to signature computations. Stage 2 is taken in Proposition \ref{PropositionPropertiesOfUpsilonInTheFirstStep}. 

Lastly, stage 3, which can be found in Proposition \ref{PropositionAuxRecursiveRelationsForUpsilonForCombinationsOfRhos}, explains how to compute $\upsilon (T(\rho_{k,n}+\rho_{k+1,n}, 2\rho_{k+1,n}))$ in terms of $\upsilon (T(\rho_{k+1,n}+\rho_{k+2,n}, 2\rho_{k+2,n}))$ if $p$ is even, or how to compute $\upsilon (T(\rho_{k,n}+\rho_{k+1,n}, \rho_{k,n}-\rho_{k+1,n}))$ in terms of $\upsilon (T(\rho_{k+1,n}+\rho_{k+2,n}, \rho_{k+1,n}-\rho_{k+2,n}))$ if $p$ is odd. Since $r_n = \rho_{1,n}$ and $s_n = \rho_{2,n}$, one can bring the results from stage 2 to bear on stage 3. The recursive formula from Proposition \ref{PropositionAuxRecursiveRelationsForUpsilonForCombinationsOfRhos} terminates after finitely many steps since $\rho_{k,k} =1$ and $\rho_{k,k-1} = 0$ by Theorem \ref{TheoremSomePropertiesOfRhoKN}, and since $\upsilon$ of the unknot equals zero. 

These three steps can be assembled to find a closed formulas for $\upsilon(T(\rho_{k,n}+\rho_{k+1,n}, 2\rho_{k+1,n}))$ and $\upsilon (T(\rho_{k,n}+\rho_{k+1,n}, \rho_{k,n}-\rho_{k+1,n}))$ for any $k, n$ (Proposition \ref{PropositionComputationForGeneralRhosWithAnySign}), which in turn provide the basis for the closed formula for $\upsilon (T(p,q))$ stated in Theorem \ref{TheoremAboutTheValueOfUpsilonForAllTorusKnots}.

\begin{proposition} \label{PropositionReducingGeneralPNCaseToMultipleOfP0}
Let $p,q>1$ be relatively prime integers and assume that $T(p,q)$ becomes unknotted after $n$ pinch moves. Let $(p_n,q_n) = (p,q)$, let $p_0$, $\eps_1$ be defined as in Theorem \ref{TheoremOnReversingPinchingMoves}, and let $r_n$ be as in \eqref{EquationDefinitionOfRnAndSn}. Recall that $p_n = p_0q_n-2\eps_1r_n$ and that $\eps_1=-1$ if $p_0=0$. Then 
\begin{align*}
\upsilon(T(p_n,q_n)) =  \left\{
\begin{array}{ll}
\upsilon(T(2q_n-2\eps_1r_n,q_n)) - \frac{p_0-2}{4}[q_n^2-1] & ; \quad \text{ if } p_0\ge 2, \cr &  \cr 
\upsilon(T(q_n-2\eps_1r_n,q_n)) - \frac{p_0-1}{4}[q_n^2-1] & ;\quad \text{ if } p_0\ge 1, \cr &  \cr 
\upsilon(T(2r_n,q_n)) - \frac{p_0}{4}[q_n^2-1]& ; \quad \text{ if } p_0\ge 0 \text{ and } \eps_1=-1.
\end{array}\right.
\end{align*}
Said differently, one may compute $\upsilon (T(p_0q_n-2\eps_1r_n,q_n))$ for any $p_0\ge 0$  by only considering the 3 cases $p_0=0, 1, 2$. 
\end{proposition}
%%%%
%%%%
\begin{proof}
Using \eqref{EquationRecursionFormulaForUpsilon} repeatedly leads to the desired result:
%%%
\begin{align*}
\upsilon (T(p_n, q_n )) & = \upsilon (T(p_0q_n-2\eps_1r_n, q_n )) \cr
& = \textstyle \upsilon (T((p_0-1)q_n-2\eps_1r_n, q_n )) - \frac{1}{4}[q_n^2-1] \cr
& = \textstyle \upsilon (T((p_0-2)q_n-2\eps_1r_n, q_n )) - \frac{2}{4}[q_n^2-1] \cr
& \qquad \vdots \cr
& = \textstyle \upsilon (T(2q_n-2\eps_1r_n, q_n )) - \frac{p_0-2}{4}[q_n^2-1] \cr
& = \textstyle \upsilon (T(q_n-2\eps_1r_n, q_n )) - \frac{p_0-1}{4}[q_n^2-1].
\end{align*}
%%%
If $\eps_1=-1$, then we additionally obtain 
%%%
\begin{align*}
\upsilon (T(p_n, q_n )) = \textstyle \upsilon (T(q_n+2r_n, q_n )) - \frac{p_0-1}{4}[q_n^2-1] = \upsilon (T(2r_n, q_n )) - \frac{p_0}{4}[q_n^2-1].
\end{align*}
%%%
\end{proof}
%%%
%%%
\begin{proposition} \label{PropositionPropertiesOfUpsilonInTheFirstStep}
Let $p,q>1$ be relatively prime integers and assume that $T(p,q)$ becomes unknotted after $n$ pinch moves. Let $(p_n,q_n) = (p,q)$, let $p_0$, $\eps_1$, $\eps_2$ be defined as in Theorem \ref{TheoremOnReversingPinchingMoves}, and let $r_n$, $s_n$ be as in \eqref{EquationDefinitionOfRnAndSn}. If $p_0\ge 2$ then  	 
%%%
\begin{align*}
\upsilon & (T(2q_n-2\eps_1r_n, q_n)) =  \cr
& = \left\{\begin{array}{l}
\textstyle \upsilon(T(r_n + s_n , r_n - s_n )) -\frac{1}{4}r_n^2(2q_1^2-2\eps_1 q_1-1) +\frac{1}{2}\eps_1 \eps_2 r_n s_n (2q_1-\eps_1) +\frac{3}{4}(1-s_n^2)\cr
%%%
\upsilon (T(r_n-s_n,2s_n)) - \textstyle \frac{1}{2}r_n^2q_1(q_1-\eps_1) +\frac{1}{2}\eps_1\eps_2 r_n s_n (2q_1-\eps_1+\eps_1\eps_2) +(1-s_n^2)  \cr 
\upsilon (T(r_n+s_n,2s_n))- \textstyle \frac{1}{2}r_n^2q_1(q_1-\eps_1) +\frac{1}{2}\eps_1\eps_2 r_n s_n (2q_1-\eps_1+\eps_1\eps_2) +1,
%%%
\end{array}
\right. 
\end{align*}
%%%%%%%%
if $p_0\ge 1$ then  
%%%%%%%%
\begin{align*}
\upsilon & (T(q_n-2\eps_1r_n, q_n)) = \cr
&= \left\{ \begin{array}{l}
\textstyle \upsilon(T(r_n + s_n , r_n-s_n))-\frac{1}{4}r_n^2(q_1^2-\eps_12q_1-1) + \frac{1}{2}\eps_1\eps_2 r_ns_n(q_1-\eps_1) +\frac{1}{2} (1-s_n^2), \cr
\textstyle \upsilon (T(r_n-s_n,2s_n)) \textstyle  -\frac{1}{4}r_n^2(q_1^2-2\eps_1 q_1) + \frac{1}{2}\eps_1\eps_2 r_ns_n(q_1-\eps_1+\eps_1\eps_2) +\frac{3}{4} (1-s_n^2), \cr
\textstyle \upsilon (T(r_n+s_n,2s_n)) \textstyle  -\frac{1}{4}r_n^2(q_1^2-2\eps_1 q_1) + \frac{1}{2}\eps_1\eps_2 r_ns_n(q_1-\eps_1+\eps_1\eps_2) +\frac{3}{4} + \frac{1}{4}s_n^2,
\end{array}\right. 
\end{align*}
%%%%%%%%
and if $p_0\ge 0$ and $\eps_1=-1$ then 
%%%%%%%%
\begin{align*}
\upsilon  (T(q_n , 2r_n)) & = \left\{ \begin{array}{l}
\textstyle \upsilon(T(r_n+s_n, r_n-s_n ))-\frac{1}{4}r_n^2(2q_1-1) - \frac{1}{2}\eps_2r_ns_n+\frac{1}{4}(1-s_n^2),\cr
\textstyle  \upsilon(T(r_n-s_n,2s_n)) -\frac{1}{2}q_1 r_n^2  -\frac{1}{2}r_ns_n(\eps_2-1) +\frac{1}{2}(1-s_n^2), \cr
\textstyle  \upsilon(T(r_n+s_n,2s_n)) -\frac{1}{2}q_1 r_n^2  -\frac{1}{2}r_ns_n(\eps_2-1) +\frac{1}{2}(1+s_n^2).
\end{array}\right.
\end{align*}
%%%
\end{proposition}
%%%%%%
%%%%%%
\begin{proof}
These formulas follow by direct computation relying on the recursive relation \eqref{EquationRecursionFormulaForUpsilon}. We start with the computation of $\upsilon (T(2q_n-2\eps_1r_n, q_n))$, breaking into cases, and the other two computations will follow quickly. 

Let us suppose that $\eps_1=1$. Recall that $q_n=q_1r_n-\eps_2s_n$ is odd.  
%%%
\begin{align*}
\upsilon (T(2q_n-&2r_n  ,q_n))  = \textstyle \upsilon(T(q_n-2r_n, q_n)) - \frac{1}{4} [q_n^2-1] \cr
& =  \textstyle \upsilon(T(q_n-2r_n, 2r_n)) - \frac{1}{4} [ (q_n-2r_n)^2-1] - \frac{1}{4} [q_n^2-1] \cr
& =  \textstyle \upsilon(T(r_n(q_1-2)) -\eps_2s_n , 2r_n) - \frac{1}{4} [ (q_n-2r_n)^2-1] - \frac{1}{4} [q_n^2-1] \cr
& =  \textstyle \upsilon(T(r_n(q_1-4)) -\eps_2s_n , 2r_n)  - r_n^2 - \frac{1}{4} [ (q_n-2r_n)^2-1] - \frac{1}{4} [q_n^2-1] \cr
& \qquad \qquad \vdots \cr
& = \textstyle \upsilon(T(r_n - \eps_2 s_n , 2r_n)) -  \frac{q_1-3}{2} r_n^2 - \frac{1}{4} [ (q_n-2r_n)^2-1] - \frac{1}{4} [q_n^2-1] \cr
& = \textstyle \upsilon(T(r_n - \eps_2 s_n , r_n + \eps_2 s_n )) - \frac{1}{4}[(r_n-\eps_2 s_n)^2-1] -  \frac{q_1-3}{2} r_n^2  \cr
& \qquad \qquad \qquad \textstyle - \frac{1}{4} [ (q_n-2r_n)^2-1]  -  \frac{1}{4} [q_n^2-1] \cr
& =  \textstyle \upsilon(T(r_n + s_n , r_n - s_n )) -\frac{1}{4}r_n^2(2q_1^2-2q_1-1) +\frac{1}{2} \eps_2 r_n s_n (2q_1-1) +\frac{3}{4}(1-s_n^2).
\end{align*}
%%%
%%%
From this formula and from 
\begin{align} \label{EquationAuxiliaryFormulasForRnAndSn}
\upsilon (T(r_n+s_n,r_n-s_n)) & = \upsilon (T(r_n-s_n,2s_n)) - \textstyle \frac{1}{4}[(r_n-s_n)^2-1], \text{ and }   \cr
\upsilon (T(r_n+s_n,2s_n)) & = \upsilon (T(r_n-s_n,2s_n)) - s_n^2,
\end{align}
we additionally obtain 
\begin{align*}
\upsilon (T(2q_n-2r_n  ,q_n)) & = \upsilon (T(r_n-s_n,2s_n)) - \textstyle \frac{1}{2}r_n^2q_1(q_1-1) +\frac{1}{2}\eps_2 r_n s_n (2q_1-1+\eps_2) +(1-s_n^2) \cr 
& = \upsilon (T(r_n+s_n,2s_n))- \textstyle \frac{1}{2}r_n^2q_1(q_1-1) +\frac{1}{2}\eps_2 r_n s_n (2q_1-1+\eps_2) +1.
\end{align*}
%%%
When $\eps_1=-1$, the computation of $\upsilon (T(2q_n+2r_n,q_n))$ proceeds similarly. 
%%%
\begin{align*}
\upsilon (T(2q_n+2r_n & ,q_n))  = \textstyle \upsilon(T(q_n+2r_n, q_n)) - \frac{1}{4} [q_n^2-1] \cr
& =  \textstyle \upsilon(T(q_n, 2r_n)) - \frac{1}{2} [ q_n^2-1] \cr
& =  \textstyle \upsilon(T(r_n q_1 +\eps_2s_n , 2r_n)) - \frac{1}{2} [ q_n^2-1] \cr
& =  \textstyle \upsilon(T(r_n (q_1-2) +\eps_2s_n , 2r_n)) - r_n^2- \frac{1}{2} [ q_n^2-1] \cr
& \qquad \qquad \vdots \cr
& = \textstyle \upsilon(T(r_n + \eps_2 s_n , 2r_n)) -  \frac{q_1-1}{2} r_n^2- \frac{1}{2} [ q_n^2-1] \cr
& = \textstyle \upsilon(T(r_n + \eps_2 s_n , r_n - \eps_2 s_n )) - \frac{1}{4}[(r_n+\eps_2 s_n)^2-1] -  \frac{q_1-1}{2} r_n^2- \frac{1}{2} [ q_n^2-1] \cr
& =  \textstyle \upsilon(T(r_n + s_n , r_n - s_n )) -\frac{1}{4}r_n^2(2q_1^2+2q_1-1) -\frac{1}{2} \eps_2 r_n s_n (2q_1+1) +\frac{3}{4}(1-s_n^2).
\end{align*}
%%%
%%%
Using again \eqref{EquationAuxiliaryFormulasForRnAndSn} we also obtain 
\begin{align*}
\upsilon (T(2q_n+2r_n  ,q_n)) & =\upsilon (T(r_n-s_n,2s_n)) \textstyle  -\frac{1}{4}r_n^2(2q_1^2+2q_1) -\frac{1}{2} \eps_2 r_n s_n (2q_1+1-\eps_2) +(1-s_n^2), \cr
& = \upsilon (T(r_n+s_n,2s_n)) \textstyle  -\frac{1}{4}r_n^2(2q_1^2+2q_1) -\frac{1}{2} \eps_2 r_n s_n (2q_1+1-\eps_2) +1.
\end{align*}
%%%
It is now easy to see that the computations for $\upsilon (T(2q_n-2\eps_1r_n,q_n))$ for the two cases of $\eps_1=\pm 1$ can be combined into the single formula stated in the proposition. 

Next, turn to the calculation of $\upsilon (T(q_n-2\eps_1r_n,q_n))$. We can use our previous computation of $\upsilon(T(2q_n-2\eps_1r_n, q_n))$ to easily compute $\upsilon(T(q_n-2\eps_1r_n, q_n))$. The first lines of our previous computation for both the case $\eps_1 = 1$ and $\eps_1 = -1$ stated that: 
$$ \upsilon (T(2q_n-\eps_12r_n,q_n))  = \textstyle \upsilon(T(q_n-2\eps_1r_n, q_n)) - \frac{1}{4} [q_n^2-1] $$

Hence the value of $\upsilon(T(q_n-2\eps_1r_n, q_n))$ is simply given by adding $\frac{1}{4}(q_n^2 - 1)$ to  the expressions we found to represent $\upsilon(T(2q_n-2\eps_1r_n, q_n))$ and by using the fact that $q_n=q_1r_n-\eps_2s_n$ to simplify.

In addition, we can use our previous computation of $\upsilon(T(2q_n-2\eps_1r_n, q_n))$ to easily compute $\upsilon(T(q_n , 2r_n))$, as well. Observe that the second line of our previous computation in the case that $\eps_1 = -1$ stated that: 
$$\upsilon (T(2q_n+2r_n,q_n))  =  \textstyle \upsilon(T(q_n, 2r_n)) - \frac{1}{2} [ q_n^2-1]$$

Hence the value of $\upsilon(T(q_n , 2r_n))$ is given by adding $\frac{1}{2} [ q_n^2-1]$ to the expressions we found to represent $\upsilon(T(2q_n-2\eps_1r_n, q_n))$ and by setting $\eps_1 = -1$ and simplifying. We leave these details of substitution and simplification to the reader. This completes the proof of the proposition. 
\end{proof}
%%%%%%%%%%%%%%%%%%%%%%%%%%%%%%%%
%%%%%%%%%%%%%%%%%%%%%%%%%%%%%%%%
%%%%%%%%%%%%%%%%%%%%%%%%%%%%%%%%
\begin{proposition}   \label{PropositionAuxRecursiveRelationsForUpsilonForCombinationsOfRhos}
Let $p,q>1$ be relatively prime integers and assume that $T(p,q)$ becomes unknotted after $n$ pinch moves. Let $\rho_{k,n}$ be defined in terms of $p,q$ as in  \eqref{DefinitionOfRhoKCommaN}. Then
%%%
%%%
\begin{align*}
\upsilon(T(\rho_{k,n}+\rho_{k+1,n},& \rho_{k,n}-\rho_{k+1,n})) =  \textstyle \upsilon(T(\rho_{k+1,n}+ \rho_{k+2,n} , \rho_{k+1,n}-\rho_{k+2,n} ))+ \cr
& \qquad \qquad \qquad \qquad\qquad \textstyle +\frac{1}{2} +\frac{1}{4}[(\rho_{k+1,n}^2-\rho_{k,n}^2)-(\rho_{k+2,n}^2-\rho_{k+1,n}^2)].
\end{align*}
%%%
From this we additionally obtain
%%%
\begin{align*}
\upsilon(&T(\rho_{k,n}-\rho_{k+1,n},2\rho_{k+1,n}))  =  \cr
& = \left\{
\begin{array}{l}
\upsilon(T(\rho_{k+1,n}-\rho_{k+2,n},2\rho_{k+2,n})) -\textstyle \frac{1}{2}[ \rho_{k,n}\rho_{k+1,n} - \rho_{k+1,n}^2-\rho_{k+1,n}\rho_{k+2,n} + \rho_{k+2,n}^2-1] \cr
\upsilon(T(\rho_{k+1,n}+\rho_{k+2,n},2\rho_{k+2,n})) -\textstyle \frac{1}{2}[ \rho_{k,n}\rho_{k+1,n} - \rho_{k+1,n}^2-\rho_{k+1,n}\rho_{k+2,n} - \rho_{k+2,n}^2-1], \cr
\end{array}
\right.
\end{align*}
%%%
%%%
\begin{align*}
\upsilon(&T(\rho_{k,n}+\rho_{k+1,n},2\rho_{k+1,n}))  =  \cr
& = \left\{
\begin{array}{l}
\upsilon(T(\rho_{k+1,n}-\rho_{k+2,n},2\rho_{k+2,n})) -\textstyle \frac{1}{2}[ \rho_{k,n}\rho_{k+1,n} + \rho_{k+1,n}^2-\rho_{k+1,n}\rho_{k+2,n} + \rho_{k+2,n}^2-1] \cr
\upsilon(T(\rho_{k+1,n}+\rho_{k+2,n},2\rho_{k+2,n})) -\textstyle \frac{1}{2}[ \rho_{k,n}\rho_{k+1,n} + \rho_{k+1,n}^2-\rho_{k+1,n}\rho_{k+2,n} - \rho_{k+2,n}^2-1]. \cr
\end{array}
\right.
\end{align*}
\end{proposition}
%%%
%%%
\begin{proof}
Recall the relation \eqref{DefinitionOfRhoKCommaN} by which  $\rho_{k,n}  = m_k\rho_{k+1,n} -\eps_{k+1}\eps_{k+2} \rho_{k+2,n}$ for all $k\in \mathbb N$. Assume firstly that $\eps_{k+1}\eps_{k+2}=1$, then the recursive relation \eqref{EquationRecursionFormulaForUpsilon} leads to   
%%%
\begin{align*}
\upsilon  (&T(\rho_{k,n}   +\rho_{k+1,n}  ,\rho_{k,n} - \rho_{k+1,n}))  = \upsilon  (T(\rho_{k,n}   -\rho_{k+1,n}  ,2 \rho_{k+1,n})) -\textstyle \frac{1}{4} [(\rho_{k,n} - \rho_{k+1,n})^2-1]   \cr
& = \upsilon (T(\rho_{k+1,n}(m_k-1) - \rho_{k+2,n}, 2\rho_{k+1,n})) -\textstyle \frac{1}{4} [(\rho_{k,n} - \rho_{k+1,n})^2-1] \cr
& = \textstyle \upsilon(T(\rho_{k+1,n} (m_k-3) - \rho_{k+2,n} , 2\rho_{k+1,n})) - \rho_{k+1,n}^2 -\textstyle \frac{1}{4} [(\rho_{k,n} - \rho_{k+1,n})^2-1]\cr
& \qquad \qquad \vdots \cr
& = \textstyle \upsilon(T(\rho_{k+1,n}- \rho_{k+2,n} , 2\rho_{k+1,n})) - \frac{m_k-2}{2} \rho_{k+1,n}^2 -\textstyle \frac{1}{4} [(\rho_{k,n} - \rho_{k+1,n})^2-1] \cr
& = \textstyle \upsilon(T(\rho_{k+1,n}- \rho_{k+2,n} , \rho_{k+1,n}+\rho_{k+2,n} )) - \frac{1}{4}[(\rho_{k+1,n}-\rho_{k+2,n})^2-1]  - \frac{m_k-2}{2} \rho_{k+1,n}^2 \cr
& \qquad \qquad \qquad -\textstyle \frac{1}{4} [(\rho_{k,n} - \rho_{k+1,n})^2-1] \cr
%
%& =  \textstyle \upsilon(T(\rho_{k+1,n}+ \rho_{k+2,n} , \rho_{k+1,n}-\rho_{k+2,n} ) -\frac{1}{4}  \rho_{k+1,n}^2(m_k^2-2)  +\frac{1}{2} m_k \rho_{k+1,n} \rho_{k+2,m} +\frac{1}{2} (1-\rho_{k+2,n}^2).
%
& =  \textstyle \upsilon(T(\rho_{k+1,n}+ \rho_{k+2,n} , \rho_{k+1,n}-\rho_{k+2,n} )) +\frac{1}{2} +\frac{1}{4}[(\rho_{k+1,n}^2-\rho_{k,n}^2)-(\rho_{k+2,n}^2-\rho_{k+1,n}^2)].
\end{align*}
%%%
The case of $\eps_{k+1}\eps_{k+2}=-1$ proceeds similarly:
%%%
\begin{align*}
\upsilon  (&T(\rho_{k,n}   +\rho_{k+1,n}  ,\rho_{k,n} - \rho_{k+1,n}))  = \upsilon  (T(\rho_{k,n}   -\rho_{k+1,n}  ,2 \rho_{k+1,n})) -\textstyle \frac{1}{4} [(\rho_{k,n} - \rho_{k+1,n})^2-1]   \cr
& = \upsilon (T(\rho_{k+1,n}(m_k-1) + \rho_{k+2,n}, 2\rho_{k+1,n})) -\textstyle \frac{1}{4} [(\rho_{k,n} - \rho_{k+1,n})^2-1] \cr
& = \textstyle \upsilon(T(\rho_{k+1,n} (m_k-3) + \rho_{k+2,n} , 2\rho_{k+1,n})) - \rho_{k+1,n}^2 -\textstyle \frac{1}{4} [(\rho_{k,n} - \rho_{k+1,n})^2-1]\cr
& \qquad \qquad \vdots \cr
& = \textstyle \upsilon(T(\rho_{k+1,n}+ \rho_{k+2,n} , 2\rho_{k+1,n})) - \frac{m_k-2}{2} \rho_{k+1,n}^2 -\textstyle \frac{1}{4} [(\rho_{k,n} - \rho_{k+1,n})^2-1] \cr
& = \textstyle \upsilon(T(\rho_{k+1,n}- \rho_{k+2,n} , \rho_{k+1,n}+\rho_{k+2,n} )) - \frac{1}{4}[(\rho_{k+1,n}+\rho_{k+2,n})^2-1]  - \frac{m_k-2}{2} \rho_{k+1,n}^2 \cr
& \qquad \qquad \qquad -\textstyle \frac{1}{4} [(\rho_{k,n} - \rho_{k+1,n})^2-1] \cr
& =  \textstyle \upsilon(T(\rho_{k+1,n}+ \rho_{k+2,n} , \rho_{k+1,n}-\rho_{k+2,n} )) +\frac{1}{2} +\frac{1}{4}[(\rho_{k+1,n}^2-\rho_{k,n}^2)-(\rho_{k+2,n}^2-\rho_{k+1,n}^2)].
%
%& =  \textstyle \upsilon(T(\rho_{k+1,n}+ \rho_{k+2,n} , \rho_{k+1,n}-\rho_{k+2,n} ).
%
\end{align*}
%%%
This proves the first formula of the proposition. The remaining formulas follow from this one and the obvious relations
\begin{align} \label{EquationObivousRecursiveRelationsAmongTheRhos}
\upsilon(T(\rho_{k,n}+\rho_{k+1,n}, \rho_{k,n}-\rho_{k+1,n})) & = \textstyle \upsilon(T(\rho_{k,n}-\rho_{k+1,n}, 2\rho_{k+1,n})) - \frac{1}{4}[ (\rho_{k,n}-\rho_{k+1,n})^2-1] , \cr  
\upsilon(T(\rho_{k,n}+\rho_{k+1,n}, 2\rho_{k+1,n})) & = \upsilon(T(\rho_{k,n}-\rho_{k+1,n}, 2\rho_{k+1,n})) - \rho_{k+1,n}^2. \cr  
\end{align}
\end{proof}
%%%
%%%
\begin{proposition} \label{PropositionComputationForGeneralRhosWithAnySign}
Let $p,q>1$ be relatively prime integers and assume that $T(p,q)$ becomes unknotted after $n$ pinch moves. Let $\rho_{k,n}$ be defined in terms of $p,q$ as in  \eqref{DefinitionOfRhoKCommaN}. Then
	
\begin{align*}
\upsilon(T(\rho_{k,n}+\rho_{k+1,n},\rho_{k,n}-\rho_{k+1,n})) & =  \textstyle \frac{1}{2}(n-k) +\frac{1}{4}(\rho_{k+1,n}^2-\rho_{k,n}^2+1), \cr
\upsilon(T(\rho_{k,n}-\rho_{k+1,n},2\rho_{k+1,n})) & =  \textstyle \frac{1}{2}(n-k +\rho_{k+1,n}^2-\rho_{k,n}\rho_{k+1,n} ), \cr
\upsilon(T(\rho_{k,n}+\rho_{k+1,n},2\rho_{k+1,n})) & =  \textstyle \frac{1}{2}(n-k -\rho_{k+1,n}^2-\rho_{k,n}\rho_{k+1,n} ).
\end{align*}
%%%
%%%
\end{proposition}
%%%
%%%
\begin{proof}
This is a direct computation using the formulas from Proposition \ref{PropositionAuxRecursiveRelationsForUpsilonForCombinationsOfRhos}. 
\begin{align*}
\upsilon(T(\rho_{k,n}+\rho_{k+1,n}&,\rho_{k,n}-\rho_{k+1,n}))  =\upsilon(T(\rho_{k+1,n}+\rho_{k+2,n},\rho_{k+1,n}-\rho_{k+2,n})) + \cr
& \qquad \qquad \qquad \textstyle +\frac{1}{2} +\frac{1}{4}[(\rho_{k+1,n}^2-\rho_{k,n}^2)-(\rho_{k+2,n}^2-\rho_{k+1,n}^2)] \cr
& =\upsilon(T(\rho_{k+2,n}+\rho_{k+3,n},\rho_{k+2,n}-\rho_{k+3,n})) + \cr
& \qquad \qquad \qquad \textstyle +\frac{1}{2} +\frac{1}{4}[(\rho_{k+1,n}^2-\rho_{k,n}^2)-(\rho_{k+2,n}^2-\rho_{k+1,n}^2)] \cr
& \qquad \qquad \qquad \textstyle +\frac{1}{2} +\frac{1}{4}[(\rho_{k+2,n}^2-\rho_{k+1,n}^2)-(\rho_{k+3,n}^2-\rho_{k+2,n}^2)] \cr
& =\upsilon(T(\rho_{k+3,n}+\rho_{k+4,n},\rho_{k+3,n}-\rho_{k+4,n})) + \cr
& \qquad \qquad \qquad \textstyle +\frac{1}{2} +\frac{1}{4}[(\rho_{k+1,n}^2-\rho_{k,n}^2)-(\rho_{k+2,n}^2-\rho_{k+1,n}^2)] \cr
& \qquad \qquad \qquad \textstyle +\frac{1}{2} +\frac{1}{4}[(\rho_{k+2,n}^2-\rho_{k+1,n}^2)-(\rho_{k+3,n}^2-\rho_{k+2,n}^2)] \cr
& \qquad \qquad \qquad \textstyle +\frac{1}{2} +\frac{1}{4}[(\rho_{k+3,n}^2-\rho_{k+2,n}^2)-(\rho_{k+4,n}^2-\rho_{k+3,n}^2)] \cr
& \qquad \qquad \vdots \cr
& =\upsilon(T(\rho_{n,n}+\rho_{n+1,n},\rho_{n,n}-\rho_{n+1,n})) + \cr
& \qquad \qquad \qquad \textstyle +\frac{1}{2} +\frac{1}{4}[(\rho_{k+1,n}^2-\rho_{k,n}^2)-(\rho_{k+2,n}^2-\rho_{k+1,n}^2)] \cr
& \qquad \qquad \qquad \textstyle +\frac{1}{2} +\frac{1}{4}[(\rho_{k+2,n}^2-\rho_{k+1,n}^2)-(\rho_{k+3,n}^2-\rho_{k+2,n}^2)] \cr
& \qquad \qquad \qquad \textstyle +\frac{1}{2} +\frac{1}{4}[(\rho_{k+3,n}^2-\rho_{k+2,n}^2)-(\rho_{k+4,n}^2-\rho_{k+3,n}^2)] \cr
& \qquad \qquad \qquad \qquad \qquad \qquad \qquad \qquad\vdots \cr
& \qquad \qquad \qquad \textstyle +\frac{1}{2} +\frac{1}{4}[(\rho_{n,n}^2-\rho_{n-1,n}^2)-(\rho_{n+1,n}^2-\rho_{n,n}^2)] \cr
\end{align*}
Recall from Theorem \ref{TheoremSomePropertiesOfRhoKN} that $\upsilon(T(\rho_{n,n}+\rho_{n+1,n},\rho_{n,n}-\rho_{n+1,n})) = \upsilon(T(1,1)) = 0$. Of the remaining summands on the right-hand side above, most cancel with opposite summands in the preceding or proceeding row. The exceptions to this are the terms 
$$\frac{1}{2} + \dots + \frac{1}{2} +\frac{1}{4}((\rho_{k+1,n}^2-\rho_{k,n}^2)-( \rho_{n+1,n}^2-\rho_{n,n}^2))= \frac{1}{2}(n-k) +\frac{1}{4}(\rho_{k+1,n}^2-\rho_{k,n}^2+1).$$ 
This proves the claimed formula for $\upsilon(T(\rho_{k,n}+\rho_{k+1,n}, \rho_{k,n}-\rho_{k+1,n}))$, the remaining two follow from this relation and the relations \eqref{EquationObivousRecursiveRelationsAmongTheRhos}.
\end{proof}
%%%
%%%
\begin{theorem} \label{TheoremWithTheUltimateUpsilonCalcuationsProved}
Let $p,q>1$ be relatively prime integers and assume that $T(p,q)$ becomes unknotted after $n$ pinch moves. Let $p_0$ be defined as in Theorem \ref{TheoremOnReversingPinchingMoves}, and let $r_n$, $s_n$ be as in \eqref{EquationDefinitionOfRnAndSn}. Then\\

\begin{align*}
\upsilon(T(p,q)) & = \frac{n}{2} + \frac{1}{4}(p_0-pq)  ,\cr  & \cr 
\upsilon(T(r_n+s_n,r_n-s_n)) & =  \textstyle \frac{1}{2}(n-1) +\frac{1}{4}(s_n^2-r_n^2+1), \cr
\upsilon(T(r_n+s_n,2s_n)) & =   \textstyle \frac{1}{2}(n-1-s_n^2-r_ns_n ),\cr
\upsilon(T(r_n-s_n,2s_n)) & =  \textstyle \frac{1}{2}(n-1+s_n^2-r_ns_n ). 
\end{align*}
\end{theorem}
%%%
%%%
\begin{proof}
If $p_0\ge 1$, then by Proposition \ref{PropositionReducingGeneralPNCaseToMultipleOfP0} we obtain 
$$ \upsilon(T(p_n,q_n)) = \upsilon(T(q_n-2\eps_1r_n,q_n) - \frac{p_0-1}{4}[q_n^2-1].$$
Proposition \ref{PropositionPropertiesOfUpsilonInTheFirstStep} gives the equality 
$$\upsilon (T(q_n-2\eps_1r_n, q_n)) = \textstyle \upsilon(T(r_n + s_n , r_n-s_n))-\frac{1}{4}r_n^2(q_1^2-\eps_12q_1-1) + \frac{1}{2}\eps_1\eps_2 r_ns_n(q_1-\eps_1) +\frac{1}{2} (1-s_n^2)$$
and since $\upsilon(T(r_n + s_n , r_n-s_n)) = \upsilon(T(\rho_{1,n} + \rho_{2,n} , \rho_{1,n} - \rho_{2,n}))$, Proposition \ref{PropositionComputationForGeneralRhosWithAnySign} gives us 
$$\textstyle \upsilon(T(r_n + s_n , r_n-s_n)) = \frac{1}{2}(n-1) + \frac{1}{4} ( s_n^2-r_n^2+1).$$
Combining these 3 relations, and keeping in mind that $r_n = \frac{\eps_1}{2}(p_0q_n - p_n)$ and $s_n=\eps_1\eps_2(q_1r_n- q_n)$, leads to 
%%%
%%%
\begin{align*}
\upsilon(&T(p_n,q_n))  = \textstyle  \frac{1}{2}(n-1) + \frac{1}{4} ( s_n^2-r_n^2+1) \cr
& \textstyle \qquad \qquad\qquad - \frac{1}{4}r_n^2(q_1^2-\eps_12q_1-1) + \frac{1}{2}\eps_1\eps_2 r_ns_n(q_1-\eps_1) +\frac{1}{2} (1-s_n^2) \cr
& \textstyle \qquad \qquad\qquad- \frac{p_0-1}{4}[q_n^2-1] \cr
& = \textstyle \frac{1}{2}n -\frac{1}{4}s_n^2 -\frac{1}{4}q_1 r_n^2 (q_1-2\eps_1)+ \frac{1}{2}\eps_1\eps_2 r_ns_n(q_1-\eps_1) - \frac{p_0-1}{4}q_n^2  +\frac{1}{4}p_0 \cr
& = \textstyle \frac{1}{2}n -\frac{1}{4}(q_1r_n-q_n)^2 -\frac{1}{4}q_1 r_n^2 (q_1-2\eps_1)+ \frac{1}{2} r_n(q_1r_n-q_n)(q_1-\eps_1)- \frac{p_0-1}{4}q_n^2  +\frac{1}{4}p_0 \cr
& = \textstyle \frac{1}{2}n+\frac{1}{2}\eps_1r_nq_n - \frac{1}{4}p_0q_n^2+\frac{1}{4}p_0 \cr
& = \textstyle \frac{1}{2}n+ \frac{1}{4} q_n(p_0q_n-p_n) - \frac{1}{4}p_0q_n^2+\frac{1}{4}p_0 \cr
& = \textstyle \frac{1}{2}n+\frac{1}{4}(p_0-p_nq_n). 
\end{align*}
%%%
%%%
The case of $p_0=0$ follows very similarly, with the differences from the case of $p_0\ge 1$ above being that $\upsilon (T(p_n,q_n)) = \upsilon(T(2r_n,q_n))$
and so Proposition \ref{PropositionPropertiesOfUpsilonInTheFirstStep} implies that    
$$ \upsilon(T(2r_n,q_n)) = \upsilon(T(r_n+s_n, r_n-s_n ))-\frac{1}{4}r_n^2(2q_1-1) - \frac{1}{2}\eps_2r_ns_n+\frac{1}{4}(1-s_n^2),$$
with $\upsilon(T(r_n+s_n, r_n-s_n ))$ computed in the same way as in the case of $p_0\ge 1$. Since $\eps_1=-1$ when $p_0=0$, then $r_n = \frac{1}{2}p_n$ and $s_n = -\eps_2(q_1r_n-q_n)$. Using these leads to $\upsilon (T(p_n,q_n)) = \frac{n}{2} -\frac{pq}{4}$. 

Lastly, the formula for $\upsilon(T(r_n + s_n , r_n-s_n))$ above is now easily used in conjunction with the recursive formula for $\upsilon$ \eqref{EquationRecursionFormulaForUpsilon} to derived the claimed formulas for  $\upsilon(T(r_n \pm s_n , 2s_n))$. Alternatively, one can also read these off directly from Proposition \ref{PropositionComputationForGeneralRhosWithAnySign} by utilizing the equalities $r_n=\rho_{1,n}$ and $s_n=\rho_{2,n}$. 
\end{proof}
%%%
%%%
Theorem \ref{TheoremAboutTheValueOfUpsilonForAllTorusKnots} is the first equality from 
Theorem \ref{TheoremWithTheUltimateUpsilonCalcuationsProved}, and is thereby proved. 
%%%%%%%%%%%%%%%%%%%%%%%%%%%%%%%%%%%%%%%%%%%%%%%%%%%%%%%%%%%%%%%%%%%%%%%%%%%%%%%%%%%%%%%%%%%%%%%%%%%%%%%%%%%%%%%%%%%%%%%%%%%%%%%%%%%%%%%%%%%%
%%%%%%%%%%%%%%%%%%%%%%%%%%%%%%%%%%%%%%%%%%%%%%%%%%%%%%%%%%%%%%%%%%%%%%%%%%%%%%%%%%%%%%%%%%%%%%%%%%%%%%%%%%%%%%%%%%%%%%%%%%%%%%%%%%%%%%%%%%%%
%%%%%%%%%%%%%%%%%%%%%%%%%%%%%%%%%%%%%%%%%%%%%%%%%%%%%%%%%%%%%%%%%%%%%%%%%%%%%%%%%%%%%%%%%%%%%%%%%%%%%%%%%%%%%%%%%%%%%%%%%%%%%%%%%%%%%%%%%%%%
%%%%%%%%%%%%%%%%%%%%%%%%%%%%%%%%%%%%%%%%%%%%%%%%%%%%%%%%%%%%%%%%%%%%%%%%%%%%%%%%%%%%%%%%%%%%%%%%%%%%%%%%%%%%%%%%%%%%%%%%%%%%%%%%%%%%%%%%%%%%
%%%%%%%%%%%%%%%%%%%%%%%%%%%%%%%%%%%%%%%%%%%%%%%%%%%%%%%%%%%%%%%%%%%%%%%%%%%%%%%%%%%%%%%%%%%%%%%%%%%%%%%%%%%%%%%%%%%%%%%%%%%%%%%%%%%%%%%%%%%%
%%%%%%%%%%%%%%%%%%%%%%%%%%%%%%%%%%%%%%%%%%%%%%%%%%%%%%%%%%%%%%%%%%%%%%%%%%%%%%%%%%%%%%%%%%%%%%%%%%%%%%%%%%%%%%%%%%%%%%%%%%%%%%%%%%%%%%%%%%%%
%%%%%%%%%%%%%%%%%%%%%%%%%%%%%%%%%%%%%%%%%%%%%%%%%%%%%%%%%%%%%%%%%%%%%%%%%%%%%%%%%%%%%%%%%%%%%%%%%%%%%%%%%%%%%%%%%%%%%%%%%%%%%%%%%%%%%%%%%%%%
%%%%%%%%%%%%%%%%%%%%%%%%%%%%%%%%%%%%%%%%%%%%%%%%%%%%%%%%%%%%%%%%%%%%%%%%%%%%%%%%%%%%%%%%%%%%%%%%%%%%%%%%%%%%%%%%%%%%%%%%%%%%%%%%%%%%%%%%%%%%
\section{Computing $\upsilon -\frac{1}{2}\sigma$ for torus knots} \label{SectionForSignatureComputations}
%%%%%%%%%%%%%%%%%%%%%%%%%%%%%%%%%%%%%%%%%%%%%%%%%%%%%%%%%%%%%%%%%%%%%%%%%%%%%%%%%%%%%%%%%%%%%%%%%%%%%%%%%%%%%%%%%%%%%%%%%%%%%%%%%%%%%%%%%%%%
%%%%%%%%%%%%%%%%%%%%%%%%%%%%%%%%%%%%%%%%%%%%%%%%%%%%%%%%%%%%%%%%%%%%%%%%%%%%%%%%%%%%%%%%%%%%%%%%%%%%%%%%%%%%%%%%%%%%%%%%%%%%%%%%%%%%%%%%%%%%
\subsection{A recursion formula for the signature of torus knots}
%%%
%%%
The computation of the signature of torus knots uses the recursive formulas proved by Gordon, Litherland, and Murasugi \cite{GordonLitherlandMurasugi:1981}. The formula we use is changed from its original version by an overall sign so as to be in agreement with the signature function used in \cite{Batson:2014} and \cite{OSS:2017}:
\begin{equation} \label{EquationRecursiveFormulaForSignature}
\sigma(T(a,b)) = 
\left\{
\begin{array}{ll}
\phantom{-}\sigma(T(b-2a,a)) -(a^2-1)    & \quad ; \quad 2a<b \text{ and $a$ is odd, } \cr & \cr   
\phantom{-}\sigma(T(b-2a,a)) -a^2    & \quad ; \quad 2a<b \text{ and $a$ is even, } \cr & \cr   
-\sigma(T(2a-b,a)) -(a^2-1)    & \quad ; \quad a\le b<2a \text{ and $a$ is odd, } \cr & \cr   
-\sigma(T(2a-b,a)) -(a^2-2)    & \quad ; \quad a\le b<2a \text{ and $a$ is even. }
\end{array}
\right. 
\end{equation}
The above relations are valid for relatively prime natural numbers $a, b$, and are accompanied by the \lq\lq boundary conditions\rq\rq \, 
$$\sigma(T(2,b)) = -(b-1)\qquad \text{ and } \qquad \sigma (T(1,b)) =0.$$ 
%%%%%
%%%%%
Our strategy for computing $\upsilon -\frac{1}{2}\sigma$ for a torus knot $T(p,q)$ follows the same 3 step strategy used in the computation of $\upsilon (T(p,q))$, a strategy we outlined at the beginning of Secion \ref{SubSectionComputingUpsilonForAllTorusKnots}. Step 1 for the computations of $\upsilon -\frac{1}{2}\sigma$ is accomplished in Proposition \ref{PropositionForInitialReductionOfUpsilonMinusHalfOfTheSignature} (analogue of Proposition \ref{PropositionReducingGeneralPNCaseToMultipleOfP0}). The second step is taken in Propositions \ref{PropositionSecondReductionStepForTheCaseOfP0BeingEven}, \ref{PropositionSecondReductionStepForTheCaseOfP0BeingOdd} and \ref{PropositionSecondReductionStepForTheCaseOfP0BeingZero} (analogues to Proposition \ref{PropositionPropertiesOfUpsilonInTheFirstStep}, broken into 3 propositions because of the lengths of their proofs), while step 3 is given by Propositions \ref{PropositionThirdReductionStepForTheCaseOfP0BeingEven} and \ref{PropositionThirdReductionStepForTheCaseOfP0BeingOdd} (analogues of Proposition \ref{PropositionAuxRecursiveRelationsForUpsilonForCombinationsOfRhos}).  
%%%
%%%
\begin{proposition} \label{PropositionForInitialReductionOfUpsilonMinusHalfOfTheSignature}
Let $p,q>1$ be relatively prime integers and assume that $T(p,q)$ becomes unknotted after $n$ pinch moves. Let $(p_n,q_n) = (p,q)$, let $p_0$, $\eps_1$ be defined as in Theorem \ref{TheoremOnReversingPinchingMoves}, and let $r_n$ be as in  \eqref{EquationDefinitionOfRnAndSn}. Recall that $p_n = p_0q_n-2\eps_1r_n$. Then 
\begin{align*}
\textstyle  (\upsilon-\frac{1}{2}\sigma)(T(p_n ,q_n)) =  (\upsilon-\frac{1}{2}\sigma)(T((p_0-2)q_n-2\eps_1r_n ,q_n)). 
\end{align*}
%%%
In particular 
$$\textstyle  (\upsilon-\frac{1}{2}\sigma)(T(p_0q_n-2\eps_1r_n ,q_n)) = \left\{
\begin{array}{ll}
\textstyle  (\upsilon-\frac{1}{2}\sigma)(T(2q_n-2\eps_1r_n,q_n)) & \quad ; \quad p_0\ge 2 \text{ is even}, \cr & \cr
\textstyle  (\upsilon-\frac{1}{2}\sigma)(T(q_n-2\eps_1r_n,q_n)) & \quad ; \quad p_0\ge 1 \text{ is odd}, \cr & \cr
\textstyle  (\upsilon-\frac{1}{2}\sigma)(T(2r_n,q_n)) & \quad ; \quad p_0 = 0. \cr
\end{array}
\right.$$
\end{proposition}
%%%
%%%
\begin{proof}
If $p_0\ge 3$ then the recursive signature formula \eqref{EquationRecursiveFormulaForSignature} gives 
%%%
\begin{align*}
\sigma(T(p_0q_n-2\eps_1r_n,q_n)) & = \sigma(T((p_0-2)q_n-2\eps_1r_n,q_n)) - [q_n^2-1]. 
\end{align*}
On the other hand, the recursive formula \eqref{EquationRecursionFormulaForUpsilon} for $\upsilon$ gives
\begin{align*}
\upsilon(T(p_0q_n-2\eps_1r_n,q_n)) & = \upsilon(T((p_0-1)q_n-2\eps_1r_n,q_n)) - \frac{1}{4} [q_n^2-1]\cr
& = \upsilon(T((p_0-2)q_n-2\eps_1r_n,q_n)) - \frac{1}{2} [q_n^2-1]
\end{align*}
The combination of the two proves the proposition. 
\end{proof}
%%%%
%%%%
The preceding proposition shows that in order to compute $(\upsilon -\frac{1}{2} \sigma)(T(p,q))$, it suffices to do so only for the case of $p_0=0, 1, 2$, three cases which we shall address separately. 
%%%
%%%
\begin{proposition} \label{PropositionSecondReductionStepForTheCaseOfP0BeingEven}
Let $p,q>1$ be relatively prime integers and assume that $T(p,q)$ becomes unknotted after $n$ pinch moves. Let $(p_n,q_n) = (p,q)$, let $p_0$, $\eps_1$, $\eps_2$ be defined as in Theorem \ref{TheoremOnReversingPinchingMoves}, and let $r_n$, $s_n$ be as in \eqref{EquationDefinitionOfRnAndSn}. Set $\eta=1$ if $q_1\equiv 1\,(\text{mod }4)$, and set $\eta =-1$ if $q_1\equiv 3\,(\text{mod }4)$. Then 
\begin{align*}
 (\upsilon  - \textstyle \frac{1}{2}\sigma)&(T(2q_n-2\eps_1r_n,q_n))  =  \cr
& =  \left\{ 
\begin{array}{llll}
\textstyle -(\upsilon -\frac{1}{2}\sigma)(T(r_n-\eta s_n,2s_n)) +n & \qquad  \eps_1=1, &\eps_2=1, \cr
\textstyle \phantom{-}(\upsilon -\frac{1}{2}\sigma)(T(r_n+\eta s_n,2s_n)) +1 & \qquad  \eps_1=1, &\eps_2=-1,\cr
\textstyle -(\upsilon -\frac{1}{2}\sigma)(T(r_n +  \eta s_n , 2s_n))+(n-1)  & \qquad  \eps_1=-1, &\eps_2=1,\cr
\textstyle \phantom{-}(\upsilon -\frac{1}{2}\sigma)(T(r_n -  \eta s_n , 2s_n)) & \qquad  \eps_1=-1, &\eps_2=-1.
\end{array}
\right.
\end{align*}
\end{proposition}
%%%
%%%
%%%
\begin{proof}
The signature computations proceed in slightly different ways depending on the value of $\eps_1$, $\eps_2$ and on the modulus of $q_1$ with respect to 4, leading us to consider 8 cases in all. We begin with the case of $\eps_1=1$, and consider firstly the event that $q_1\equiv 1\, (\text{mod }4)$. After several steps this calculation then splits itself in two cases corresponding to $\eps_2=\pm 1$. Recall that $q_n = q_1r_n-\eps_1\eps_2s_n$. 
\begin{align*}
\sigma (T(2q_n-2r_n,q_n)) & = \textstyle -\sigma (T(q_n,2r_n)) - [q_n^2-1] \cr
& = \textstyle -\sigma (T(q_1r_n-\eps_2 s_n ,    2r_n)) - [q_n^2-1] \cr
& =  \textstyle -\sigma(T(r_n(q_1-4)-\eps_2 s_n, 2r_n)) + 4r_n^2 - [q_n^2-1]  \cr
& \qquad \qquad \vdots \cr
& =  \textstyle -\sigma(T(r_n - \eps_2  s_n , 2r_n)) + (q_1-1) r_n^2 - [q_n^2-1]
\end{align*}
The computation proceeds in slightly different ways depending on the value of $\eps_2$. We first address the case of $\eps_2=1$. 
%%%
\begin{align*}
\sigma (T(2q_n-2r_n,q_n)) & =   \textstyle -\sigma(T(r_n -   s_n , 2r_n)) + (q_1-1) r_n^2 - [q_n^2-1] \cr
& =  \textstyle -\sigma(T(r_n - s_n , 2s_n)) + [(r_n-s_n)^2-1]  + (q_1-1) r_n^2 - [q_n^2-1]  \cr
& =  \textstyle -\sigma(T(r_n - s_n , 2s_n)) -r_n^2q_1(q_1-1) +2r_ns_n (q_1-1). 
\end{align*}  
%%%
If $\eps_2=-1$, then 
%%%
\begin{align*}
\sigma (T(2q_n-2r_n,q_n)) & =  \textstyle -\sigma(T(r_n +   s_n , 2r_n)) + (q_1-1) r_n^2 - [q_n^2-1] \cr
& =  \textstyle \sigma(T(r_n + s_n , 2s_n )) + [(r_n+s_n)^2-1]  + (q_1-1) r_n^2 - [q_n^2-1]  \cr
& =  \textstyle \sigma(T(r_n + s_n , 2s_n)) -r_n^2q_1(q_1-1) -2r_ns_n (q_1-1). 
\end{align*}
%%%
The preceding two calculations can be summarized in a single relation (still only valid for $q_1\equiv 1\,(4)$ and $\eps_1=1$)
\begin{equation} \label{EquationFirstStepRecuctionForSigmaCase1}\sigma (T(2q_n-2r_n,q_n)) = -\eps_2 \sigma(T(r_n - \eps_2 s_n , 2s_n)) -r_n^2q_1(q_1-1) +2 \eps_2 r_ns_n (q_1-1). 
\end{equation}
%%%
If $q_1\equiv 3\, (\text{mod }4)$, then
%%%
\begin{align*}
\sigma (T(2q_n-2r_n,q_n)) & = \textstyle -\sigma (T(q_n,2r_n) - [q_n^2-1] \cr
& =  \textstyle -\sigma (T(q_1r_n-\eps_2s_n ,2r_n)) - [q_n^2-1] \cr
& =  \textstyle -\sigma(T(r_n(q_1-4) -\eps_2s_n , 2r_n)) + 4r_n^2 - [q_n^2-1]  \cr
& \qquad \qquad \vdots \cr
& =  \textstyle -\sigma(T(3r_n - \eps_2 s_n , 2r_n)) + (q_1-3) r_n^2 - [q_n^2-1]  \cr
& =  \textstyle \sigma(T(r_n + \eps_2 s_n , 2r_n)) + (4r_n^2-2) +  (q_1-3) r_n^2 - [q_n^2-1]  
\end{align*}
%%%
We continue by differentiating between $\eps_2=1$ and $\eps_2=-1$, starting with the former. 
%%%
\begin{align*} 
\sigma (T(2q_n-2r_n,q_n)) & = \textstyle \sigma(T(r_n +  s_n , 2r_n)) + (4r_n^2-2) +  (q_1-3) r_n^2 - [q_n^2-1]  \cr
& =  \textstyle -\sigma(T(r_n + s_n , 2s_n))  - [(r_n+s_n)^2-1] + (4r_n^2-2) +  (q_1-3) r_n^2 - [q_n^2-1]  \cr
& =  \textstyle -\sigma(T(r_n + s_n , 2s_n))  -r_n^2q_1(q_1-1) +2r_ns_n (q_1-1)-2s_n^2. 
\end{align*}
%%%
Here is the calculation for the case of $\eps_2=-1$:
%%%
\begin{align*} 
\sigma (T(2q_n-2r_n,q_n)) & = \textstyle \sigma(T(r_n -  s_n , 2r_n)) + (4r_n^2-2) +  (q_1-3) r_n^2 - [q_n^2-1]  \cr
& =  \textstyle \sigma(T(r_n - s_n , 2s_n))  - [(r_n-s_n)^2-1] + (4r_n^2-2) +  (q_1-3) r_n^2 - [q_n^2-1]  \cr
& =  \textstyle \sigma(T(r_n - s_n , 2s_n))  -r_n^2q_1(q_1-1) -2r_ns_n (q_1-1)-2s_n^2.    
\end{align*}
%%%
These last two calculation can again be summarized in a single relation (valid if $\eps_1=1$ and $q_1\equiv 3\,(4)$)
\begin{equation} \label{EquationFirstStepRecuctionForSigmaCase2}
\sigma (T(2q_n-2r_n,q_n)) = -\eps_2 \sigma(T(r_n +\eps_2 s_n , 2s_n)  -r_n^2q_1(q_1-1)+  2\eps_2 r_ns_n (q_1-1)-2s_n^2.
\end{equation}
Next we turn to the case of $\eps_1=-1$ and calculate the signature of $T(2q_n+2r_n,q_n)$. Here too we first consider the case of $q_1\equiv 1\,(4)$.
\begin{align*}
\sigma (T(2q_n+2r_n,q_n)) & = \textstyle \sigma (T(q_n,2r_n)) - [q_n^2-1] \cr
& = \textstyle \sigma (T(q_1r_n+\eps_2 s_n ,    2r_n)) - [q_n^2-1] \cr
& =  \textstyle \sigma(T(r_n(q_1-4)+\eps_2 s_n, 2r_n)) - 4r_n^2 - [q_n^2-1]  \cr
& \qquad \qquad \vdots \cr
& =  \textstyle \sigma(T(r_n + \eps_2  s_n , 2r_n)) - (q_1-1) r_n^2 - [q_n^2-1]
\end{align*}
If $\eps_2=1$, we proceed as follows:
\begin{align*}
\sigma (T(2q_n+2r_n,q_n)) & =  \textstyle \sigma(T(r_n +  s_n , 2r_n)) - (q_1-1) r_n^2 - [q_n^2-1] \cr
 & =  \textstyle -\sigma(T(r_n +  s_n , 2s_n)) - [(r_n+s_n)^2-1] - (q_1-1) r_n^2 - [q_n^2-1] \cr
& =  \textstyle -\sigma(T(r_n +  s_n , 2s_n)) -r_n^2q_1(q_1+1) - 2r_ns_n(q_1+1) +2(1-s_n^2).
\end{align*}
If $\eps_2=-1$, then 
\begin{align*}
\sigma (T(2q_n+2r_n,q_n)) & =  \textstyle \sigma(T(r_n -  s_n , 2r_n)) - (q_1-1) r_n^2 - [q_n^2-1] \cr
 & =  \textstyle \sigma(T(r_n -  s_n , 2s_n)) - [(r_n-s_n)^2-1] - (q_1-1) r_n^2 - [q_n^2-1] \cr
& =  \textstyle \sigma(T(r_n -  s_n , 2s_n)) -r_n^2q_1(q_1+1) + 2r_ns_n(q_1+1) +2(1-s_n^2).
\end{align*}
Both cases can be summarized with (still with $\eps_1=-1$ and $q_1\equiv1\,(4)$)
\begin{equation} \label{EquationFirstStepRecuctionForSigmaCase3}
\sigma (T(2q_n+2r_n,q_n))  =  \textstyle -\eps_2\sigma(T(r_n +\eps_2  s_n , 2s_n)) -r_n^2q_1(q_1+1) - 2\eps_2 r_ns_n(q_1+1) +2(1-s_n^2).
\end{equation}
Turning to the case of $q_1\equiv 3\,(4)$, we find
\begin{align*}
\sigma (T(2q_n+2r_n,q_n)) & = \textstyle \sigma (T(q_n,2r_n)) - [q_n^2-1] \cr
& = \textstyle \sigma (T(q_1r_n+\eps_2 s_n ,    2r_n)) - [q_n^2-1] \cr
& =  \textstyle \sigma(T(r_n(q_1-4)+\eps_2 s_n, 2r_n)) - 4r_n^2 - [q_n^2-1]  \cr
& \qquad \qquad \vdots \cr
& =  \textstyle \sigma(T(3r_n + \eps_2  s_n , 2r_n)) - (q_1-3) r_n^2 - [q_n^2-1] \cr
& =  \textstyle -\sigma(T(r_n - \eps_2  s_n , 2r_n)) - (4r_n^2-2) - (q_1-3) r_n^2 - [q_n^2-1]
\end{align*}
If $\eps_2=1$, then 
\begin{align*}
\sigma (T(2q_n+2r_n,& q_n))  = \textstyle -\sigma(T(r_n -  s_n , 2r_n)) - (4r_n^2-2)- (q_1-3) r_n^2 - [q_n^2-1] \cr
& =  \textstyle -\sigma(T(r_n -  s_n , 2s_n )) +[(r_n-s_n)^2-1]    - (4r_n^2-2) - (q_1-3) r_n^2 - [q_n^2-1] \cr
& =  \textstyle -\sigma(T(r_n -  s_n , 2s_n ))-r_n^2q_1(q_1+1) -2r_ns_n(q_1+1) +2,
\end{align*}
while if $\eps_2=-1$, then 
\begin{align*}
\sigma (T(2q_n+2r_n,& q_n))  = \textstyle -\sigma(T(r_n +  s_n , 2r_n)) - (4r_n^2-2)- (q_1-3) r_n^2 - [q_n^2-1] \cr
& =  \textstyle \sigma(T(r_n +  s_n , 2s_n )) +[(r_n+s_n)^2-1]    - (4r_n^2-2) - (q_1-3) r_n^2 - [q_n^2-1] \cr
& =  \textstyle \sigma(T(r_n +  s_n , 2s_n ))-r_n^2q_1(q_1+1) +2r_ns_n(q_1+1) +2.
\end{align*}
The two can be summarized (with $\eps_1=-1$ and $q_1\equiv 3\,(4)$) as 
\begin{equation} \label{EquationFirstStepRecuctionForSigmaCase4}
\sigma (T(2q_n+2r_n, q_n))  =  -\eps_2  \sigma(T(r_n -\eps_2 s_n , 2s_n ))-r_n^2q_1(q_1+1) -2\eps_2 r_ns_n(q_1+1) +2.
\end{equation}

We now take the signature formulas \eqref{EquationFirstStepRecuctionForSigmaCase1} -- \eqref{EquationFirstStepRecuctionForSigmaCase4} and combine them with similar formulas for $\upsilon$ calculated in Proposition \ref{PropositionPropertiesOfUpsilonInTheFirstStep}. There are eight cases, those where the value of $\eps_1$ and $\eps_2$ vary between $-1$ and $1$, the congruence class of $q_1$ modulo 4 varies between 1 and 3. Each case demands the exact same pattern of substitution, cancelation, and simplification. We do the first two cases, which show all the needed techniques, and leave the remaining cases to the reader. 

If $\eps_1=1$, $\eps_2=1$, and $q_1\equiv 1\,(\text{mod }4)$, then 
%%%
\begin{align*}
(\upsilon & \textstyle-\frac{1}{2}\sigma)(T(2q_n-2r_n,q_n))  =\textstyle \upsilon(T(2q_n-2r_n,q_n)) -\frac{1}{2}\sigma(T(2q_n-2r_n,q_n)) \cr
& = \upsilon (T(r_n-s_n,2s_n)) - \textstyle \frac{1}{2}r_n^2q_1(q_1-1) + r_n s_n q_1 +(1-s_n^2) + \cr
&\qquad \qquad \qquad \textstyle + \frac{1}{2}\sigma(T(r_n-s_n,2s_n))  +\frac{1}{2} r_n^2q_1(q_1-1) -r_ns_n (q_1-1) \cr
& = \textstyle -(\upsilon -\frac{1}{2}\sigma)(T(r_n-s_n,2s_n)) +2  \upsilon (T(r_n-s_n,2s_n)) +(1-s_n^2) +r_ns_n \cr
& = \textstyle -(\upsilon -\frac{1}{2}\sigma)(T(r_n-s_n,2s_n)) +(n -1-  s_nr_n+s_n^2)    +(1-s_n^2) +r_ns_n \cr
& = \textstyle -(\upsilon -\frac{1}{2}\sigma)(T(r_n-s_n,2s_n)) +n.
\end{align*}
%%%
If $\eps_1=1$, $\eps_2=-1$, and $q_1\equiv 1\,(4)$, then  
%%%
\begin{align*}
(\upsilon & \textstyle-\frac{1}{2}\sigma)(T(2q_n-2r_n,q_n))  =\textstyle \upsilon(T(2q_n-2r_n,q_n)) -\frac{1}{2}\sigma(T(2q_n-2r_n,q_n)) \cr
& = \upsilon (T(r_n+s_n,2s_n))- \textstyle \frac{1}{2}r_n^2q_1(q_1-1) -r_n s_n (q_1-1) +1 \cr
&\qquad \qquad \qquad \textstyle - \frac{1}{2}\sigma(T(r_n+s_n,2s_n)) +\frac{1}{2} r_n^2q_1(q_1-1) + r_ns_n (q_1-1) \cr
& = \textstyle (\upsilon -\frac{1}{2}\sigma)(T(r_n+s_n,2s_n)) +1.
\end{align*}
%%%
Continuing in this manner, the proof is complete.
\end{proof}
%%%
%%%
\begin{proposition}  \label{PropositionSecondReductionStepForTheCaseOfP0BeingOdd}
Let $p,q>1$ be relatively prime integers and assume that $T(p,q)$ becomes unknotted after $n$ pinch moves. Let $(p_n,q_n) = (p,q)$, let $p_0$, $\eps_1$, $\eps_2$ be defined as in Theorem \ref{TheoremOnReversingPinchingMoves}, and let $r_n$, $s_n$ be as in \eqref{EquationDefinitionOfRnAndSn}. Set $\eta=1$ if $q_1\equiv 1\,(\text{mod }4)$, and set $\eta =-1$ if $q_1\equiv 3\,(\text{mod }4)$. Then 
\begin{align*}
\textstyle (\upsilon -\frac{1}{2}\sigma)(T(q_n-2\eps_1 r_n,q_n)) & = \left\{
\begin{array}{lll}
\textstyle -(\upsilon -\frac{1}{2}\sigma)(T(r_n+s_n,r_n-s_n)) +n & \quad \eta \eps_1  = 1, \cr &&\cr
\textstyle \phantom{-} (\upsilon -\frac{1}{2}\sigma)(T(r_n+s_n,r_n-s_n)) & \quad \eta \eps_1 = -1.
\end{array}
\right.
\end{align*}
%%%
\end{proposition}
%%%
%%%
\begin{proof}
The proof of this proposition consists of two parts. In the first part, we use the recursive signature formula \eqref{EquationRecursiveFormulaForSignature} to find relations between the signature of $T(q_n-2\eps_1r_n,q_n)$ and $T(r_n+s_n,r_n-s_n)$, and in the second part, we combine these relations with the analogous upsilon computations from Section \ref{SectionOnComputingTheUpsilonInvariant} to obtain the formulas claimed. 

We begin with signature computations, where we consider several subcases, starting with $\eps_1=1$. Recall that $q_n = q_1r_n-\eps_1\eps_2s_n$.  
%%%
\begin{align*}
\sigma (T&(q_n-2r_n,q_n))  = \sigma(T(r_n(q_1-2) - \eps_2 s_n, q_1r_n-\eps_2 s_n)) \cr  
& =  -\sigma (T(r_n(q_1-4) - \eps_2 s_n, r_n(q_1-2) -\eps_2 s_n)) - [(r_n(q_1-2) -\eps_2 s_n)^2-1]  \cr
& = \sigma (T( r_n(q_1 -6) -\eps_2 s_n ,r_n(q_1-4) -\eps_2 s_n ))   + \sum_{k=1}^2 (-1)^k   [(r_n(q_1-2k) -\eps_2 s_n)^2-1]  \cr  
& \qquad \qquad \vdots \cr
& = (-1)^\frac{q_1-3}{2}\sigma(T(r_n -\eps_2 s_n ,3r_n -\eps_2 s_n ))   + \sum_{k=1}^{(q_1-3)/2} (-1)^k   [(r_n(q_1-2k) -\eps_2 s_n)^2-1]  \cr  
& = (-1)^\frac{q_1-3}{2}\sigma(T(r_n - \eps_2 s_n ,r_n +\eps_2  s_n )) -  (-1)^\frac{q_1-3}{2} [(r_n-\eps_2 s_n)^2-1]  + \cr
& \qquad \qquad \qquad \qquad + \sum_{k=1}^{(q_1-3)/2} (-1)^k   [(r_n(q_1-2k) -\eps_2 s_n)^2-1]  \cr  
\end{align*}
%%%
The above alternating sum is easily seen to evaluate to 
\begin{align*} 
\sum_{k=1}^{(q_1-3)/2} (-1)^k  & [(r_n(q_1-2k) -\eps_2 s_n)^2-1]  = \cr
& = \left\{
\begin{array}{ll}
 -\frac{1}{2}r_n^2(q_1^2-2q_1+3) + \eps_2 r_ns_n(q_1+1) +(1-s_n^2) & \quad ; \quad q_1\equiv 1\,(4), \cr  & \cr
-\frac{1}{2}r_n^2(q_1^2-2q_1-3) + \eps_2 r_ns_n(q_1-3)  & \quad ; \quad q_1\equiv 3\,(4),
\end{array}
\right.
\end{align*}
implying that when $\eps_1 = 1$, we have
\begin{align} \label{EquationIntermediateSignatureFormulaForP0EqualToOneCase1} 
&\sigma(T( q_n-2r_n,q_n))  = \cr
& = \left\{
\begin{array}{ll}
-\sigma(T(r_n+s_n, r_n-s_n)) -\frac{1}{2}r_n^2(q_1^2-2q_1+1) +\eps_2 r_ns_n(q_1-1)  & ; q_1\equiv 1\,(4), \cr & \cr
\phantom{-}\sigma(T(r_n+s_n, r_n-s_n))  -\frac{1}{2}r_n^2(q_1^2-2q_1-1) +\eps_2 r_ns_n(q_1-1) + (1-s_n^2) & ; q_1\equiv 3\,(4).
\end{array}
\right.
\end{align}
%%%
Next we turn to the case of $\eps_1=-1$. 
%%%
\begin{align*}
\sigma (T&(q_n+2r_n,q_n))  = -\sigma (T(q_n-2r_n,q_n)) -(q_n^2-1) \cr
&= -\sigma (T(r_n(q_1 -2) + \eps_2s_n ,r_nq_1+\eps_2s_n )) -(q_n^2-1) \cr
& =\sigma (T(r_n(q_1 -4) + \eps_2s_n ,r_n(q_1-2) +\eps_2s_n )) + [(r_n(q_1-2)+\eps_2s_n)^2-1]  -(q_n^2-1) \cr
& =-\sigma (T(r_n(q_1 -6) + \eps_2s_n ,r_n(q_1-4) +\eps_2s_n )) + \sum_{k=1}^2 (-1)^{k-1} [(r_n(q_1-2k)+\eps_2s_n)^2-1]   \cr
& \qquad \qquad \qquad  \qquad \qquad \qquad -(q_n^2-1) \cr
& \qquad \qquad \vdots \cr
& =(-1)^{\frac{q_1-1}{2}}\sigma (T(r_n + \eps_2s_n ,3r_n +\eps_2s_n )) + \sum_{k=1}^{(q_1-3)/2} (-1)^{k-1} [(r_n(q_1-2k)+\eps_2s_n)^2-1]  \cr
& \qquad \qquad \qquad  \qquad \qquad \qquad -(q_n^2-1) \cr
& =(-1)^{\frac{q_1-1}{2}}\sigma (T(r_n + \eps_2s_n ,r_n -\eps_2s_n )) -(-1)^{\frac{q_1-1}{2}}[(r_n+\eps_2s_n)^2-1] + \cr 
& \qquad \qquad \qquad  \qquad \qquad \qquad + \sum_{k=1}^{(q_1-3)/2} (-1)^{k-1} [(r_n(q_1-2k)+\eps_2s_n)^2-1]  -(q_n^2-1).
\end{align*}
%%%
The above alternating sum is again easily evaluated, and becomes 
\begin{align*} 
\sum_{k=1}^{(q_1-3)/2} (-1)^{k-1} &  [(r_n(q_1-2k) +\eps_2 s_n)^2-1]  = \cr
& = \left\{
\begin{array}{ll}
 \frac{1}{2}r_n^2(q_1^2-2q_1+3) + \eps_2 r_ns_n(q_1+1) -(1-s_n^2) & \quad ; \quad q_1\equiv 1\,(4), \cr  & \cr
\frac{1}{2}r_n^2(q_1^2-2q_1-3) + \eps_2 r_ns_n(q_1-3)  & \quad ; \quad q_1\equiv 3\,(4),
\end{array}
\right.
\end{align*}
leading to 
\begin{align}  \label{EquationIntermediateSignatureFormulaForP0EqualToOneCase2} 
&\sigma(T( q_n+2r_n,q_n))  = \cr
& = \left\{
\begin{array}{ll}
\phantom{-}\sigma(T(r_n+s_n, r_n-s_n)) -\frac{1}{2}r_n^2(q_1^2+2q_1-1) -\eps_2 r_ns_n(q_1+1)+(1-s_n^2)  & ; q_1\equiv 1\,(4), \cr & \cr
-\sigma(T(r_n+s_n, r_n-s_n))  -\frac{1}{2}r_n^2(q_1^2+2q_1+1) -\eps_2 r_ns_n(q_1+1)  & ; q_1\equiv 3\,(4).
\end{array}
\right.
\end{align}
%%%

Next we combine formulas \eqref{EquationIntermediateSignatureFormulaForP0EqualToOneCase1} and \eqref{EquationIntermediateSignatureFormulaForP0EqualToOneCase2} with the result on upsilon from Proposition \ref{PropositionPropertiesOfUpsilonInTheFirstStep} and Theorem \ref{TheoremWithTheUltimateUpsilonCalcuationsProved}. There are 4 cases to consider: $\eps_1=\pm 1$ and $q_1\equiv 1,3 \, (\text{mod }4)$. As before, each of the 4 cases demands the exact same pattern of substitution, canceling, and simplification. We do the first case, and leave the remaining cases to the reader.

If $\eps_1=1$ and $q_1\equiv 1\,(\text{mod }4)$ then 
%%%
\begin{align*}
\textstyle (\upsilon  -\frac{1}{2}\sigma) & (T(q_n-2r_n,q_n))  = \textstyle \upsilon (T(q_n-2r_n,q_n)) - \frac{1}{2}\sigma(T(q_n-2r_n,q_n)) \cr
& = \textstyle  \upsilon(T(r_n + s_n , r_n-s_n))-\frac{1}{4}r_n^2(q_1^2-2q_1-1) + \frac{1}{2}\eps_2 r_ns_n(q_1-1) +\frac{1}{2} (1-s_n^2)  \cr
& \qquad \textstyle +\frac{1}{2} \sigma(T(r_n+s_n, r_n-s_n)) +\frac{1}{4}r_n^2(q_1^2-2q_1+1) -\frac{1}{2} \eps_2 r_ns_n(q_1-1) \cr
& = \textstyle -(\upsilon - \frac{1}{2}\sigma)(T(r_n+s_n,r_n-s_n)) + 2\upsilon (T(r_n+s_n,r_n-s_n))+\frac{1}{2}r_n^2   +\frac{1}{2} (1-s_n^2)\cr
& = \textstyle -(\upsilon - \frac{1}{2}\sigma)(T(r_n+s_n,r_n-s_n)) + (n-1) +\frac{1}{2}(s_n^2-r_n^2+1) +\frac{1}{2}r_n^2   +\frac{1}{2} (1-s_n^2) \cr
& =  \textstyle -(\upsilon - \frac{1}{2}\sigma)(T(r_n+s_n,r_n-s_n)) + n.
\end{align*}
%%%
Continuing in this way, the proposition is proved. 
\end{proof}
%%%
%%%
\begin{proposition}  \label{PropositionSecondReductionStepForTheCaseOfP0BeingZero}
Let $p,q>1$ be relatively prime integers and assume that $T(p,q)$ becomes unknotted after $n$ pinch moves. Let $(p_n,q_n) = (p,q)$, let $p_0$, $\eps_2$ be defined as in Theorem \ref{TheoremOnReversingPinchingMoves}, and let $r_n$, $s_n$ be as in \eqref{EquationDefinitionOfRnAndSn} and let $\eps_1=-1$. %Observe that in this case, it will be the case that $\eps_1 = -1$. 
%
%The value of $(\upsilon-\frac{1}{2}\sigma)(T(q_n,2r_n))$ is identical to that of  $(\upsilon -\frac{1}{2}\sigma)(T(2q_n+2r_n,q_n))$ found in Proposition~\ref{PropositionSecondReductionStepForTheCaseOfP0BeingEven} (by setting $\eps_1 = -1$). In particular, let $\eta=1$ if $q_1\equiv 1\,(\text{mod }4)$, and let $\eta =-1$ if $q_1\equiv 3\,(\text{mod }4)$. Then we have
Let $\eta=1$ if $q_1\equiv 1\,(\text{mod }4)$, and let $\eta =-1$ if $q_1\equiv 3\,(\text{mod }4)$, then
 \begin{align*}
\textstyle (\upsilon-\frac{1}{2}\sigma)(T(q_n, 2r_n)) & = \left\{
\begin{array}{ll}
\textstyle -(\upsilon -\frac{1}{2}\sigma)(T(r_n+\eta s_n,2s_n)) +(n-1) & ;  \eps_2=1, \cr & \cr
\textstyle \phantom{-}(\upsilon -\frac{1}{2}\sigma)(T(r_n-\eta s_n,2s_n))  & ; \eps_2=-1, \cr 
\end{array}
\right.
\end{align*}
Note that $(\upsilon-\frac{1}{2}\sigma)(T(q_n, 2r_n)) = (\upsilon-\frac{1}{2}\sigma)(T(2q_n+2r_n , q_n))$, the latter of which is as calculated in Proposition \ref{PropositionSecondReductionStepForTheCaseOfP0BeingEven}. \end{proposition}
%%%
%%%
\begin{proof}
Observe that in the proof of Proposition~\ref{PropositionSecondReductionStepForTheCaseOfP0BeingEven} in the case that $\eps_1 = -1$, we come upon the following equation, regardless of the values $\eps_2$ or $q_1$:

$$\sigma (T(2q_n+2r_n,q_n)) = \textstyle \sigma (T(q_n,2r_n)) - [q_n^2-1] $$
Therefore, we know the values of $\sigma$ for these torus knots $T(q_n,2r_n)$ are given by:
$$ \sigma (T(q_n,2r_n))  = \sigma (T(2q_n+2r_n,q_n)) - [q_n^2-1] $$
Moreover, in the proof of Proposition~\ref{PropositionPropertiesOfUpsilonInTheFirstStep}, we noted that the values of $\upsilon$ for torus knots $T(q_n,2r_n)$ are  given by:
$$ \upsilon (T(q_n,2r_n))  = \upsilon (T(2q_n+2r_n,q_n)) - \frac{1}{2} [q_n^2-1] $$
Hence, we have the desired result:
$$(\upsilon - \textstyle \frac{1}{2}\sigma)(T(q_n,2r_n)) = (\upsilon - \frac{1}{2}\sigma)(T(2q_n+2r_n,q_n)).$$
\end{proof}

\begin{proposition}  \label{PropositionThirdReductionStepForTheCaseOfP0BeingEven}
Let $p,q>1$ be relatively prime integers and assume that $T(p,q)$ becomes unknotted after $n$ pinch moves. Let $\rho_{k,n}$ be defined in terms of $p,q$ as in  \eqref{DefinitionOfRhoKCommaN}. For $k\in \mathbb N$, set $\eps=\eps_{k+1}\eps_{k+2}$. Then  
%%%
\begin{align*}
 (\upsilon & -\textstyle \frac{1}{2}\sigma)(T(\rho_{k,n} \pm \rho_{k+1,n},2\rho_{k+1,n}))  = \cr 
& = \left\{
\begin{array}{ll}
\textstyle \phantom{-}(\upsilon -\frac{1}{2}\sigma)(T(\rho_{k+1,n} \mp  \rho_{k+2,n}, 2\rho_{k+2,n}))   & \quad ; \quad  m_k\equiv 0\,(4), \eps=1 \cr  & \cr
\textstyle \phantom{-}(\upsilon -\frac{1}{2}\sigma)(T(\rho_{k+1,n}\pm  \rho_{k+2,n}, 2\rho_{k+2,n}))   & \quad ; \quad  m_k\equiv 2\,(4), \eps=1, \cr & \cr
\textstyle - (\upsilon -\frac{1}{2}\sigma) (T(\rho_{k+1,n} \pm   \rho_{k+2,n}, 2\rho_{k+2,n})) + (n-k-1)   & \quad ; \quad  m_k\equiv 0\,(4), \eps=-1 \cr  & \cr
\textstyle - (\upsilon -\frac{1}{2}\sigma) (T(\rho_{k+1,n} \mp   \rho_{k+2,n}, 2\rho_{k+2,n})) + (n-k-1)   & \quad ; \quad  m_k\equiv 2\,(4), \eps=-1.
\end{array}
\right.
\end{align*}
%%%
%%%
\end{proposition}
%%%
%%%
\begin{proof}
The signature computations break into several distinct cases. The first 4 computations regard $\sigma(T(\rho_{k,n}-\rho_{k+1,n}, 2\rho_{k+1,n}))$, the next 4 consider $\sigma(T(\rho_{k,n}+\rho_{k+1,n}, 2\rho_{k+1,n}))$. Each step has 4 substeps the distinguish the sign of $\eps$ and the congruence class of $m_k$ modulo 4. 

We start by assuming that  $\eps = 1$ and $m_k\equiv 0\, (\text{mod }4)$, then    
%%%
\begin{align*}  %\label{EquationSignatureOfRhoKNMinusRhoKPlus1NToTwoRhoKPlus1MKCongruentToZero}
\sigma (T(\rho_{k,n} & -\rho_{k+1,n},2\rho_{k+1,n}))   = \textstyle \sigma (T(\rho_{k+1,n} (m_k -1) -\rho_{k+2,n} ,2\rho_{k+1,n})) \cr
& =  \textstyle \sigma (T(\rho_{k+1,n}(m_k-5) - \rho_{k+2,n} , 2\rho_{k+1,n})) - 4\rho_{k+1,n}^2  \cr
& \qquad \qquad \vdots \cr
& =   \textstyle \sigma (T(3\rho_{k+1,n}-\rho_{k+2,n},2\rho_{k+1,n})) - (m_k-4) \rho_{k+1,n}^2  \cr
& =   \textstyle -\sigma (T(\rho_{k+1,n}+\rho_{k+2,n},2\rho_{k+1,n})) -(4\rho_{k+1,n}^2-2) - (m_k-4) \rho_{k+1,n}^2  \cr
& =   \textstyle \sigma (T(\rho_{k+1,n}+\rho_{k+2,n},2\rho_{k+2,n})) + [(\rho_{k+1,n}+\rho_{k+2,n})^2-1] -(4\rho_{k+1,n}^2-2)\cr
& \qquad \qquad  - (m_k-4) \rho_{k+1,n}^2  \cr
%
%& =   \textstyle \sigma (T(\rho_{k+1,n}+\rho_{k+2,n},2\rho_{k+2,n}) -[(m_k-1)\rho_{k+1,n}^2 - 2\rho_{k+1,n} \rho_{k+2,n} -\rho_{k+2,n}^2 -1]\cr
%
& =  \textstyle \sigma (T(\rho_{k+1,n}+\rho_{k+2,n},2\rho_{k+2,n})) -[\rho_{k,n}\rho_{k+1,n} -\rho_{k+1,n}^2-\rho_{k+1,n}\rho_{k+2,n} -\rho_{k+2,n}^2-1].
\end{align*}
%%%
If $\eps=-1$ and $m_k\equiv 0\, (\text{mod }4)$, the calculation differs from the previous one only in the last 3 rows, and we only indicate those differences. 
%%%
\begin{align*} %\label{SignatureOfRhoKNMinusRhoKPlus1NToTwoRhoKPlus1MKCongruentToZero}
\sigma (&T(\rho_{k,n}  -\rho_{k+1,n},2\rho_{k+1,n}))   = \textstyle \sigma (T(\rho_{k+1,n} (m_k -1) +\rho_{k+2,n} ,2\rho_{k+1,n})) \cr
& =   \textstyle \sigma (T(3\rho_{k+1,n}+\rho_{k+2,n},2\rho_{k+1,n})) - (m_k-4) \rho_{k+1,n}^2  \cr
& =   \textstyle -\sigma (T(\rho_{k+1,n}-\rho_{k+2,n},2\rho_{k+1,n})) -(4\rho_{k+1,n}^2-2) - (m_k-4) \rho_{k+1,n}^2  \cr
& =\textstyle -\sigma (T(\rho_{k+1,n}-\rho_{k+2,n},2\rho_{k+2,n})) + [(\rho_{k+1,n}-\rho_{k+2,n})^2-1] - (4\rho_{k+1,n}^2-2)\cr
& \qquad \qquad  - (m_k-4) \rho_{k+1,n}^2  \cr
& = \textstyle -\sigma (T(\rho_{k+1,n}-\rho_{k+2,n},2\rho_{k+2,n})) -[\rho_{k,n}\rho_{k+1,n} -\rho_{k+1,n}^2+\rho_{k+1,n}\rho_{k+2,n} -\rho_{k+2,n}^2-1].
\end{align*}
%%%
If $\eps=1$ and $m_k\equiv 2\, (\text{mod }4)$ then 
%%%
\begin{align*} %\label{SignatureOfRhoKNMinusRhoKPlus1NToTwoRhoKPlus1MKCongruentToZero}
\sigma (T(\rho_{k,n} & -\rho_{k+1,n},2\rho_{k+1,n}))   = \textstyle \sigma (T(\rho_{k+1,n} (m_k -1) -\rho_{k+2,n} ,2\rho_{k+1,n})) \cr
& =  \textstyle \sigma (T(\rho_{k+1,n}(m_k-5) - \rho_{k+2,n} , 2\rho_{k+1,n})) - 4\rho_{k+1,n}^2  \cr
& \qquad \qquad \vdots \cr
& =   \textstyle \sigma (T(\rho_{k+1,n}-\rho_{k+2,n},2\rho_{k+1,n})) - (m_k-2) \rho_{k+1,n}^2  \cr
& =    \textstyle \sigma (T(\rho_{k+1,n}-\rho_{k+2,n},2\rho_{k+2,n})) -[(\rho_{k+1,n}-\rho_{k+2,n})^2-1]  - (m_k-2) \rho_{k+1,n}^2  \cr
& =    \textstyle \sigma (T(\rho_{k+1,n}-\rho_{k+2,n},2\rho_{k+2,n})) -[\rho_{k,n}\rho_{k+1,n} -\rho_{k+1,n}^2-\rho_{k+1,n}\rho_{k+2,n} +\rho_{k+2,n}^2-1].
\end{align*}
%%%
If $\eps=-1$ and $m_k\equiv 2\, (\text{mod }4)$, the calculation differs from the previous one in the last 3 lines only.
%%%
\begin{align*} %\label{SignatureOfRhoKNMinusRhoKPlus1NToTwoRhoKPlus1MKCongruentToZero}
\sigma (T(\rho_{k,n} & -\rho_{k+1,n},2\rho_{k+1,n}))   = \textstyle \sigma (T(\rho_{k+1,n} (m_k -1) +\rho_{k+2,n} ,2\rho_{k+1,n})) \cr
& =   \textstyle \sigma (T(\rho_{k+1,n}+\rho_{k+2,n},2\rho_{k+1,n})) - (m_k-2) \rho_{k+1,n}^2  \cr
& =   -\textstyle \sigma (T(\rho_{k+1,n}+\rho_{k+2,n},2\rho_{k+2,n})) -[(\rho_{k+1,n}+\rho_{k+2,n})^2-1] - (m_k-2) \rho_{k+1,n}^2  \cr
& =   -\textstyle \sigma (T(\rho_{k+1,n}+\rho_{k+2,n},2\rho_{k+2,n})) -[\rho_{k,n}\rho_{k+1,n} -\rho_{k+1,n}^2+\rho_{k+1,n}\rho_{k+2,n} +\rho_{k+2,n}^2-1].
\end{align*}
%%%

We now turn to computations of $\sigma(T(\rho_{k,n}+\rho_{k+1,n},2\rho_{k+1,n}))$, starting with the case of $\eps=1$ and $m_k\equiv 0\, (\text{mod }4)$.  
%%%
\begin{align*} %\label{SignatureOfRhoKNMinusRhoKPlus1NToTwoRhoKPlus1MKCongruentToZero}
\sigma (T(\rho_{k,n} & +\rho_{k+1,n},2\rho_{k+1,n}))   = \textstyle \sigma (T(\rho_{k+1,n} (m_k +1) -\rho_{k+2,n} ,2\rho_{k+1,n})) \cr
& =  \textstyle \sigma (T(\rho_{k+1,n}(m_k-3) - \rho_{k+2,n} , 2\rho_{k+1,n})) - 4\rho_{k+1,n}^2  \cr
& \qquad \qquad \vdots \cr
& = \textstyle \sigma (T(\rho_{k+1,n}-\rho_{k+2,n},2\rho_{k+1,n})) - m_k\rho_{k+1,n}^2  \cr
& =  \textstyle \sigma (T(\rho_{k+1,n}-\rho_{k+2,n},2\rho_{k+2,n})) -[(\rho_{k+1,n}-\rho_{k+2,n})^2-1] - m_k\rho_{k+1,n}^2  \cr
& =  \textstyle \sigma (T(\rho_{k+1,n}-\rho_{k+2,n},2\rho_{k+2,n})) -[\rho_{k,n}\rho_{k+1,n} +\rho_{k+1,n}^2-\rho_{k+1,n}\rho_{k+2,n} +\rho_{k+2,n}^2-1].
\end{align*}
%%%
The case of $\eps=-1$ and $m_k\equiv 0\, (\text{mod }4)$ differs from the previous case only in the last 3 lines.  
%%%
\begin{align*} %\label{SignatureOfRhoKNMinusRhoKPlus1NToTwoRhoKPlus1MKCongruentToZero}
\sigma (T(\rho_{k,n} & +\rho_{k+1,n},2\rho_{k+1,n}))   = \textstyle \sigma (T(\rho_{k+1,n} (m_k +1) +\rho_{k+2,n} ,2\rho_{k+1,n})) \cr
& = \textstyle \sigma (T(\rho_{k+1,n}+\rho_{k+2,n},2\rho_{k+1,n})) - m_k\rho_{k+1,n}^2  \cr
& =  -\textstyle \sigma (T(\rho_{k+1,n}+\rho_{k+2,n},2\rho_{k+2,n})) -[(\rho_{k+1,n}+\rho_{k+2,n})^2-1] - m_k\rho_{k+1,n}^2  \cr
& =  \textstyle -\sigma (T(\rho_{k+1,n}+\rho_{k+2,n},2\rho_{k+2,n})) -[\rho_{k,n}\rho_{k+1,n} +\rho_{k+1,n}^2+\rho_{k+1,n}\rho_{k+2,n} +\rho_{k+2,n}^2-1].
\end{align*}
%%%
Next up consider the case of $\eps=1$ and $m_k\equiv 2\, (\text{mod }4)$.  
%%%
\begin{align*} %\label{SignatureOfRhoKNMinusRhoKPlus1NToTwoRhoKPlus1MKCongruentToZero}
\sigma (T(\rho_{k,n} & +\rho_{k+1,n},2\rho_{k+1,n}))   = \textstyle \sigma (T(\rho_{k+1,n} (m_k +1) -\rho_{k+2,n} ,2\rho_{k+1,n})) \cr
& =  \textstyle \sigma (T(\rho_{k+1,n}(m_k-3) - \rho_{k+2,n} , 2\rho_{k+1,n})) - 4\rho_{k+1,n}^2  \cr
& \qquad \qquad \vdots \cr
& = \textstyle \sigma (T(3\rho_{k+1,n}-\rho_{k+2,n},2\rho_{k+1,n})) - (m_k-2)\rho_{k+1,n}^2  \cr
& =  \textstyle -\sigma (T(\rho_{k+1,n}+\rho_{k+2,n},2\rho_{k+1,n})) -(4\rho_{k+1,n}^2-2)- (m_k-2)\rho_{k+1,n}^2  \cr
& =  \textstyle \sigma (T(\rho_{k+1,n}+\rho_{k+2,n},2\rho_{k+2,n})) +[(\rho_{k+1,n}+\rho_{k+2,n})^2-1]  -(4\rho_{k+1,n}^2-2) \cr
& \qquad \qquad - (m_k-2)\rho_{k+1,n}^2  \cr
& =  \textstyle \sigma (T(\rho_{k+1,n}+\rho_{k+2,n},2\rho_{k+2,n})) -[\rho_{k,n}\rho_{k+1,n} +\rho_{k+1,n}^2-\rho_{k+1,n}\rho_{k+2,n} -\rho_{k+2,n}^2-1].
\end{align*}
%%%
The last case, that of  $\eps=-1$ and $m_k\equiv 2\, (\text{mod }4)$, differs from the previous case in the last 3 lines only. 
%%%
\begin{align*} %\label{SignatureOfRhoKNMinusRhoKPlus1NToTwoRhoKPlus1MKCongruentToZero}
\sigma (T(\rho_{k,n} & +\rho_{k+1,n},2\rho_{k+1,n}))   = \textstyle \sigma (T(\rho_{k+1,n} (m_k +1) +\rho_{k+2,n} ,2\rho_{k+1,n})) \cr
& = \textstyle \sigma (T(3\rho_{k+1,n}+\rho_{k+2,n},2\rho_{k+1,n})) - (m_k-2)\rho_{k+1,n}^2  \cr
& =  \textstyle -\sigma (T(\rho_{k+1,n}-\rho_{k+2,n},2\rho_{k+1,n})) -(4\rho_{k+1,n}^2-2)- (m_k-2)\rho_{k+1,n}^2  \cr
& = \textstyle -\sigma (T(\rho_{k+1,n}-\rho_{k+2,n},2\rho_{k+2,n})) + [(\rho_{k+1,n}-\rho_{k+2,n})^2-1]  -(4\rho_{k+1,n}^2-2) \cr
& \qquad \qquad - (m_k-2)\rho_{k+1,n}^2  \cr
& =  \textstyle-\sigma (T(\rho_{k+1,n}-\rho_{k+2,n},2\rho_{k+2,n})) -[\rho_{k,n}\rho_{k+1,n} +\rho_{k+1,n}^2+\rho_{k+1,n}\rho_{k+2,n} -\rho_{k+2,n}^2-1].
\end{align*}
%%%

Having found formulas for $\sigma(T(\rho_{k,n}\pm \rho_{k+1,n}, 2\rho_{k+1,n}))$, we now can use similar formulas for upsilon from Propositions \ref{PropositionAuxRecursiveRelationsForUpsilonForCombinationsOfRhos} and  \ref{PropositionComputationForGeneralRhosWithAnySign}, to prove the validity of the expressions for $(\upsilon-\frac{1}{2}\sigma)(T(\rho_{k,n}\pm \rho_{k+1,n}, 2\rho_{k+1,n}))$ claimed in the proposition. As in the first half of the proof, this second half also has 8 subcases to consider. Each case demands the exact same pattern of substitution, cancelation, and simplification. We do the first two cases, which show all the needed techniques, and leave the remaining cases to the reader.

If $\eps=1$ and $m_k\equiv 0\,(\text{mod }4)$, then 
%%%
\begin{align*}
(\upsilon & -\textstyle \frac{1}{2}\sigma)(T(\rho_{k,n}-\rho_{k+1,n},2\rho_{k+1,n})) = \cr
& = \upsilon(T(\rho_{k,n}-\rho_{k+1,n},2\rho_{k+1,n}))  -\textstyle \frac{1}{2}\sigma(T(\rho_{k,n}-\rho_{k+1,n},2\rho_{k+1,n})) \cr
& = \textstyle \upsilon(T(\rho_{k+1,n}+\rho_{k+2,n},2\rho_{k+2,n})) -\textstyle \frac{1}{2}[ \rho_{k,n}\rho_{k+1,n} - \rho_{k+1,n}^2-\rho_{k+1,n}\rho_{k+2,n} - \rho_{k+2,n}^2-1]  \cr
& \qquad  \textstyle  -\frac{1}{2} \sigma (T(\rho_{k+1,n}+\rho_{k+2,n},2\rho_{k+2,n})) +\frac{1}{2} [\rho_{k,n}\rho_{k+1,n} -\rho_{k+1,n}^2-\rho_{k+1,n}\rho_{k+2,n} -\rho_{k+2,n}^2-1] \cr
& = \textstyle (\upsilon -\frac{1}{2}\sigma)(T(\rho_{k+1,n}+\rho_{k+2,n},2\rho_{k+2,n})).
\end{align*}
%%%
If $\eps=-1$ and $m_k\equiv 0\,(\text{mod }4)$, then 
%%%
\begin{align*}
(\upsilon & -\textstyle \frac{1}{2}\sigma)(T(\rho_{k,n}-\rho_{k+1,n},2\rho_{k+1,n})) = \cr
& = \upsilon(T(\rho_{k,n}-\rho_{k+1,n},2\rho_{k+1,n}))  -\textstyle \frac{1}{2}\sigma(T(\rho_{k,n}-\rho_{k+1,n},2\rho_{k+1,n})) \cr
& = \textstyle \upsilon(T(\rho_{k+1,n}-\rho_{k+2,n},2\rho_{k+2,n})) -\textstyle \frac{1}{2}[ \rho_{k,n}\rho_{k+1,n} - \rho_{k+1,n}^2-\rho_{k+1,n}\rho_{k+2,n} + \rho_{k+2,n}^2-1]  \cr
& \qquad  \textstyle  +\frac{1}{2} \sigma (T(\rho_{k+1,n}-\rho_{k+2,n},2\rho_{k+2,n})) +\frac{1}{2} [\rho_{k,n}\rho_{k+1,n} -\rho_{k+1,n}^2+\rho_{k+1,n}\rho_{k+2,n} -\rho_{k+2,n}^2-1]  \cr
& = \textstyle -(\upsilon -\frac{1}{2}\sigma)(T(\rho_{k+1,n}-\rho_{k+2,n},2\rho_{k+2,n})) + 2\upsilon(T(\rho_{k+1,n}-\rho_{k+2,n},2\rho_{k+2,n})) \cr
&\qquad \qquad   + \rho_{k+1,n}\rho_{k+2,n} - \rho_{k+2,n}^2 \cr
& = \textstyle -(\upsilon -\frac{1}{2}\sigma)(T(\rho_{k+1,n}-\rho_{k+2,n},2\rho_{k+2,n})) + (n-k-1 +\rho_{k+2,n}^2-\rho_{k+1,n}\rho_{k+2,n} ) \cr
& \qquad \qquad  + \rho_{k+1,n}\rho_{k+2,n} - \rho_{k+2,n}^2 \cr 
& = \textstyle -(\upsilon -\frac{1}{2}\sigma)(T(\rho_{k+1,n}-\rho_{k+2,n},2\rho_{k+2,n})) +n-k-1.
\end{align*}
%%%
Continuing in this way, the proof is complete. 
\end{proof}
%%%
%%%
\begin{proposition} \label{PropositionThirdReductionStepForTheCaseOfP0BeingOdd}
Let $p,q>1$ be relatively prime integers and assume that $T(p,q)$ becomes unknotted after $n$ pinch moves. Let $\rho_{k,n}$ be defined in terms of $p,q$ as in  \eqref{DefinitionOfRhoKCommaN}. Then
\begin{align*}
\textstyle (\upsilon -\frac{1}{2}\sigma)(T&(\rho_{k,n}+\rho_{k+1,n} , \rho_{k,n}-\rho_{k+1,n}))  = \cr
& = \left\{
\begin{array}{ll}
 - (\upsilon -\frac{1}{2}\sigma)(T(\rho_{k+1,n}+\rho_{k+2,n} , \rho_{k+1,n}-\rho_{k+2,n}))+(n-k)  & ; \quad m_k\equiv 0\,(4), \cr & \cr
 \phantom{-} (\upsilon -\frac{1}{2}\sigma)(T(\rho_{k+1,n}+\rho_{k+2,n} , \rho_{k+1,n}-\rho_{k+2,n}))  & ; \quad m_k\equiv 2\,(4).
\end{array}
\right.
\end{align*}
\end{proposition}
%%%
%%%
\begin{proof}
Recall the relation $\rho_{k,n} = m_k\rho_{k+1,n}-\eps_{k+1}\eps_{k+2}\rho_{k+2,n}$. 
%%%
%%%
\begin{align*}
\sigma & (T(\rho_{k,n}+\rho_{k+1,n} , \rho_{k,n}-\rho_{k+1,n}))  = \cr
& =  \textstyle \sigma(T(\rho_{k+1,n}(m_k+1)-\eps_{k+1}\eps_{k+2}\rho_{k+2,n} ,\rho_{k+1,n}(m_k-1)-\eps_{k+1}\eps_{k+2}\rho_{k+2,n} ))  \cr
& = -\sigma(T(\rho_{k+1,n}(m_k-1)-\eps_{k+1}\eps_{k+2}\rho_{k+2,n} , \rho_{k+1,n}(m_k-3)-\eps_{k+1}\eps_{k+2}\rho_{k+2,n}))  \cr
&\qquad \qquad \qquad \qquad  \textstyle  -[(\rho_{k+1,n}(m_k-1)-\eps_{k+1}\eps_{k+2}\rho_{k+2,n})^2-1]  \cr
& =  \sigma(T(\rho_{k+1,n}(m_k-3)-\eps_{k+1}\eps_{k+2}\rho_{k+2,n} , \rho_{k+1,n}(m_k-5)-\eps_{k+1}\eps_{k+2}\rho_{k+2,n} ))   \cr
& \qquad \qquad \qquad \qquad  \textstyle+ \sum_{\ell=1}^2 (-1)^\ell  [(\rho_{k+1,n}(m_k-2\ell+1)-\eps_{k+1}\eps_{k+2}\rho_{k+2,n})^2-1]  \cr
& \qquad \qquad \vdots \cr
& = (-1)^\frac{m_k-2}{2}  \sigma(T(3\rho_{k+1,n}-\eps_{k+1}\eps_{k+2}\rho_{k+2,n} , \rho_{k+1,n}-\eps_{k+1}\eps_{k+2}\rho_{k+2,n} )) \cr
& \qquad \qquad \qquad \qquad \textstyle  + \sum_{\ell=1}^{(m_k-2)/2}  (-1)^\ell  [(\rho_{k+1,n}(m_k-2\ell+1)-\eps_{k+1}\eps_{k+2}\rho_{k+2,n})^2-1]  \cr
& = (-1)^\frac{m_k-2}{2}  \sigma(T(\rho_{k+1,n}+\eps_{k+1}\eps_{k+2}\rho_{k+2,n} , \rho_{k+1,n}-\eps_{k+1}\eps_{k+2}\rho_{k+2,n} ))  \cr
& \qquad \qquad \qquad \qquad  \textstyle -  (-1)^\frac{m_k-2}{2}  [(\rho_{k+1,n}-\eps_{k+1}\eps_{k+2}\rho_{k+2,n})^2-1]  \cr
& \qquad \qquad \qquad \qquad \textstyle + \sum_{\ell=1}^{(m_k-2)/2}  (-1)^\ell  [(\rho_{k+1,n}(m_k-2\ell+1)-\eps_{k+1}\eps_{k+2}\rho_{k+2,n})^2-1]  \cr
& = (-1)^\frac{m_k-2}{2}  \sigma(T(\rho_{k+1,n}+\rho_{k+2,n} , \rho_{k+1,n}-\rho_{k+2,n} )) \cr
& \qquad \qquad \qquad \qquad  \textstyle -  (-1)^\frac{m_k-2}{2}  [(\rho_{k+1,n}-\eps_{k+1}\eps_{k+2}\rho_{k+2,n})^2-1]  \cr
& \qquad \qquad \qquad \qquad \textstyle + \sum_{\ell=1}^{(m_k-2)/2}  (-1)^\ell  [(\rho_{k+1,n}(m_k-2\ell+1)-\eps_{k+1}\eps_{k+2}\rho_{k+2,n})^2-1].
\end{align*}
%%%%
The alternating sum becomes 
\begin{align*} 
&\sum_{\ell=1}^{(m_k-2)/2}  (-1)^\ell   [(\rho_{k+1,n} (m_k-2\ell+1) -\eps_{k+1}\eps_{k+2} \rho_{k+2,n})^2-1] = \cr
& \quad = \left\{
\begin{array}{ll}
-\frac{1}{2}\rho_{k+1,n}^2(m_k^2+2) +\eps_{k+1}\eps_{k+2} \rho_{k+1,n} \rho_{k+2,n}(m_k+2) +(1-\rho_{k+2,n}^2) ,  & \quad m_k\equiv 0\,(4), \cr & \cr
-\frac{1}{2}\rho_{k+1,n}^2(m_k^2-4) +\eps_{k+1}\eps_{k+2}\rho_{k+1,n} \rho_{k+2,n}(m_k-2)  ,  & \quad m_k\equiv 2\,(4),
\end{array}
\right.
\end{align*}
%%%
leading to 
%%%
\begin{align*} 
&\sigma(T(\rho_{k,n}+  \rho_{k+1,n},\rho_{k,n}-\rho_{k+1,n}) = \cr
&\quad  = \left\{
\begin{array}{ll}
-\sigma (T(\rho_{k+1,n}+\rho_{k+2,n},\rho_{k+1,n}-\rho_{k+2,n})) -\frac{1}{2}(\rho_{k,n}^2-\rho_{k+2,n}^2)  &  \quad m_k\equiv 0\,(4), \cr & \cr  
\phantom{-} \sigma (T(\rho_{k+1,n}+\rho_{k+1,n},\rho_{k+1,n}-\rho_{k+1,n}))+ 1-\frac{1}{2}(\rho_{k,n}^2-2\rho_{k+1,n}^2+\rho_{k+2,n}^2)  &\quad m_k\equiv 2\,(4).
\end{array}
\right.
\end{align*}
%%%
Equipped with these signature calculations, we join them with the upsilon calculations from Propositions \ref{PropositionAuxRecursiveRelationsForUpsilonForCombinationsOfRhos} and  \ref{PropositionComputationForGeneralRhosWithAnySign}. If $m_k\equiv 0\,(\text{mod }4)$ then 
%%%
\begin{align*}
(\upsilon & -\textstyle\frac{1}{2}\sigma)(T(\rho_{k,n}+\rho_{k+1,n},\rho_{k,n}-\rho_{k+1,n})) = \cr
& = \upsilon(T(\rho_{k,n}+\rho_{k+1,n},\rho_{k,n}-\rho_{k+1,n})) - \textstyle \frac{1}{2}\sigma (T(\rho_{k,n}+\rho_{k+1,n},\rho_{k,n}-\rho_{k+1,n})) \cr
& = \textstyle  \upsilon(T(\rho_{k+1,n}+ \rho_{k+2,n} , \rho_{k+1,n}-\rho_{k+2,n} ))+ \frac{1}{2} +\frac{1}{4}[(\rho_{k+1,n}^2-\rho_{k,n}^2)-(\rho_{k+2,n}^2-\rho_{k+1,n}^2)]\cr
& \textstyle \qquad +\frac{1}{2} \sigma (T(\rho_{k+1,n}+\rho_{k+2,n},\rho_{k+1,n}-\rho_{k+2,n})) +\frac{1}{4}(\rho_{k,n}^2-\rho_{k+2,n}^2)\cr
& = \textstyle -(\upsilon-\frac{1}{2}\sigma)(T(\rho_{k+1,n}+\rho_{k+2,n},\rho_{k+1,n}-\rho_{k+2,n}))+2\upsilon(T(\rho_{k+1,n}+\rho_{k+2,n},\rho_{k+1,n}-\rho_{k+2,n})) \cr
& \textstyle \qquad \qquad +\frac{1}{2}(1+\rho_{k+1,n}^2 - \rho_{k+2,n}^2)\cr
& =    \textstyle -(\upsilon-\frac{1}{2}\sigma)(T(\rho_{k+1,n}+\rho_{k+2,n},\rho_{k+1,n}-\rho_{k+2,n}))+(n-k-1  +\frac{1}{2}(\rho_{k+2,n}^2-\rho_{k+1,n}^2+1))\cr
& \textstyle \qquad \qquad +\frac{1}{2}(1+\rho_{k+1,n}^2 - \rho_{k+2,n}^2)\cr
& =    \textstyle -(\upsilon-\frac{1}{2}\sigma)(T(\rho_{k+1,n}+\rho_{k+2,n},\rho_{k+1,n}-\rho_{k+2,n}))+n-k.
\end{align*}
%%%
 If $m_k\equiv 2\,(\text{mod }4)$ then 
%%%
\begin{align*}
(\upsilon & -\textstyle\frac{1}{2}\sigma)(T(\rho_{k,n}+\rho_{k+1,n},\rho_{k,n}-\rho_{k+1,n})) = \cr
& = \upsilon(T(\rho_{k,n}+\rho_{k+1,n},\rho_{k,n}-\rho_{k+1,n})) - \textstyle \frac{1}{2}\sigma (T(\rho_{k,n}+\rho_{k+1,n},\rho_{k,n}-\rho_{k+1,n})) \cr
& = \textstyle  \upsilon(T(\rho_{k+1,n}+ \rho_{k+2,n} , \rho_{k+1,n}-\rho_{k+2,n} ))+ \frac{1}{2} +\frac{1}{4}[(\rho_{k+1,n}^2-\rho_{k,n}^2)-(\rho_{k+2,n}^2-\rho_{k+1,n}^2)] \cr
& \textstyle \qquad -\frac{1}{2}\sigma (T(\rho_{k+1,n}+\rho_{k+1,n},\rho_{k+1,n}-\rho_{k+1,n}))-\frac{1}{2} +\frac{1}{4}(\rho_{k,n}^2-2\rho_{k+1,n}^2+\rho_{k+2,n}^2) \cr
& = \textstyle (\upsilon-\frac{1}{2}\sigma)(T(\rho_{k+1,n}+\rho_{k+2,n},\rho_{k+1,n}-\rho_{k+2,n})).
\end{align*}
%%%
This completes the proof. 
\end{proof}
%%%%%%%%%%%%%%%%%%%%%%%%%%%%%%%%%%%%%%%%%%%%%%%%%%%%%%%%%%%%%%%%%%%%%%%%%%%%%%%%%%%%%%%%%%%%%%%%%%%%%%%%%%%%%%%%%%
%%%%%%%%%%%%%%%%%%%%%%%%%%%%%%%%%%%%%%%%%%%%%%%%%%%%%%%%%%%%%%%%%%%%%%%%%%%%%%%%%%%%%%%%%%%%%%%%%%%%%%%%%%%%%%%%%%
%%%%%%%%%%%%%%%%%%%%%%%%%%%%%%%%%%%%%%%%%%%%%%%%%%%%%%%%%%%%%%%%%%%%%%%%%%%%%%%%%%%%%%%%%%%%%%%%%%%%%%%%%%%%%%%%%%
%%%%%%%%%%%%%%%%%%%%%%%%%%%%%%%%%%%%%%%%%%%%%%%%%%%%%%%%%%%%%%%%%%%%%%%%%%%%%%%%%%%%%%%%%%%%%%%%%%%%%%%%%%%%%%%%%%
%%%%%%%%%%%%%%%%%%%%%%%%%%%%%%%%%%%%%%%%%%%%%%%%%%%%%%%%%%%%%%%%%%%%%%%%%%%%%%%%%%%%%%%%%%%%%%%%%%%%%%%%%%%%%%%%%%
\section{The Main Proofs} \label{SectionWithProofsOfMainTheorems}
%%%
%%%
\subsection{Proof of Theorem \ref{MainResultForP0Odd}} \label{ProofOfMainResultForP0Odd}
%%%
%%%
Assume that $p>1$ odd, then $p_0\ge 1$ also odd. Recall that we set $\eta =1$ if $q_1\equiv 1\,(\text{mod }4)$, and we set $\eta =-1$ if $q_1\equiv 3\,(\text{mod }4)$. By Proposition \ref{PropositionForInitialReductionOfUpsilonMinusHalfOfTheSignature} we know that 
\begin{equation} \label{EquationAuxiliaryFormulaForP0BeingOddNumber1}
\textstyle (\upsilon-\frac{1}{2}\sigma)(T(p,q)) = (\upsilon-\frac{1}{2}\sigma)(T(q_n-2\eps_1r_n,q_n)).
\end{equation}
By Proposition \ref{PropositionSecondReductionStepForTheCaseOfP0BeingOdd} we further know that 
%%%
\begin{align} \label{EquationAuxiliaryFormulaForP0BeingOddNumber2}
\textstyle (\upsilon -\frac{1}{2}\sigma)(T(q_n-2\eps_1 r_n,q_n)) & = \left\{
\begin{array}{lll}
\textstyle -(\upsilon -\frac{1}{2}\sigma)(T(r_n+s_n,r_n-s_n)) +n & \quad \eta \eps_1  = 1, \cr &&\cr
\textstyle \phantom{-} (\upsilon -\frac{1}{2}\sigma)(T(r_n+s_n,r_n-s_n)) & \quad \eta \eps_1 = -1.
\end{array}
\right.
\end{align}
%%%
%%% 
\begin{remark}
The condition $\eta \eps_1=1$ is equivalent to the congruence $q_1-\eps_1\equiv 0\,(\text{mod }4)$, while the condition $\eta \eps_1=-1$ is equivalent to the congruence $q_1-\eps_1\equiv 2\,(\text{mod }4)$. We use these two congruences in the statement of Theorem \ref{MainResultForP0Odd} as do they not require the introduction of the extra variable $\eta$. However, we rely on $\eta$ in the present chapter to  allow us to state some results in simpler format, for example Propositions \ref{PropositionSecondReductionStepForTheCaseOfP0BeingEven} and \ref{PropositionSecondReductionStepForTheCaseOfP0BeingZero}.
\end{remark}
%%%
%%%
These suffice to prove Theorem \ref{MainResultForP0Odd} for $n=1$, in which case $r_n=r_1=1$ and $s_n=s_1=0$, leading to $T(r_n+s_n,r_n-s_n) = T(1,1)$ and therefore $(\upsilon -\frac{1}{2}\sigma)(T(r_n+s_n,r_n-s_n)) = 0$. From \eqref{EquationAuxiliaryFormulaForP0BeingOddNumber1} and \eqref{EquationAuxiliaryFormulaForP0BeingOddNumber2} it follows that 
$$\textstyle (\upsilon -\frac{1}{2}\sigma)(T(p,q))  = (\upsilon -\frac{1}{2}\sigma)(T(q_1-2\eps_1r_1,q_1)) =  \left\{
\begin{array}{lll}
1 & \quad ; \quad  \eta \eps_1  = 1, \cr &&\cr
0 & \quad ; \quad \eta \eps_1 = -1,
\end{array}
\right.$$
%%%
which is as claimed in Theorem \ref{MainResultForP0Odd}. 

If $n\ge 2$ we additionally invoke Proposition \ref{PropositionThirdReductionStepForTheCaseOfP0BeingOdd} by which 
%%%
\begin{align} \label{EquationAuxFormulaForStepThreeInTheProofOfTheCaseOfP0BeingOdd}
\textstyle (\upsilon -\frac{1}{2}&\sigma)(T(\rho_{k,n}+\rho_{k+1,n} , \rho_{k,n}-\rho_{k+1,n}))  = \cr
& = \left\{
\begin{array}{ll}
 - (\upsilon -\frac{1}{2}\sigma)(T(\rho_{k+1,n}+\rho_{k+2,n} , \rho_{k+1,n}-\rho_{k+2,n}))+(n-k)  & ; \quad m_k\equiv 0\,(4), \cr & \cr
 \phantom{-} (\upsilon -\frac{1}{2}\sigma)(T(\rho_{k+1,n}+\rho_{k+2,n} , \rho_{k+1,n}-\rho_{k+2,n}))  & ; \quad m_k\equiv 2\,(4).
\end{array}
\right.
\end{align}
%%%
We see that in passing from $(\upsilon -\frac{1}{2}\sigma)(T(\rho_{k,n}+\rho_{k+1,n} , \rho_{k,n}-\rho_{k+1,n}))$ to $(\upsilon -\frac{1}{2}\sigma)(T(\rho_{k+1,n}+\rho_{k+2,n} , \rho_{k+1,n}-\rho_{k+2,n}))$ using formula \eqref{EquationAuxFormulaForStepThreeInTheProofOfTheCaseOfP0BeingOdd}, we either get identical values (if $m_k\equiv 2\,(\text{mod }4)$) or the quantity changes sign and picks up the summand $n-k$ (if $m_k\equiv 0\,(\text{mod }4)$). This motivates the definition of the set
$$\mathcal I = \{k\in \{1,\dots, n-1\}\,|\, m_k\equiv 0\,(\text{mod }4)\},$$ 
and we write $\mathcal I=\{k_{1}, k_{2}, \dots, k_{\ell}\}$ with $k_{1}<k_{2}<\dots<k_{\ell}$. If $k\notin \mathcal I$ then 
$$\textstyle   (\upsilon -\frac{1}{2}\sigma)(T(\rho_{k,n}+\rho_{k+1,n} , \rho_{k,n}-\rho_{k+1,n})) = (\upsilon -\frac{1}{2}\sigma)(T(\rho_{k+1,n}+\rho_{k+2,n} , \rho_{k+1,n}-\rho_{k+2,n}))$$
while if $k\in \mathcal I$ then 
$$\textstyle   (\upsilon -\frac{1}{2}\sigma)(T(\rho_{k,n}+\rho_{k+1,n} , \rho_{k,n}-\rho_{k+1,n})) = -(\upsilon -\frac{1}{2}\sigma)(T(\rho_{k+1,n}+\rho_{k+2,n} , \rho_{k+1,n}-\rho_{k+2,n})) + (n-k).$$
From these it follows that 
\begin{align} \label{EquationAuxiliaryFormulaForP0BeingOddNumber3}
(\upsilon & -\textstyle\frac{1}{2}\sigma)(T((r_n +s_n , r_n-s_n))  = \textstyle (\upsilon -\frac{1}{2}\sigma)(T((\rho_{1,n}+\rho_{2,n} , \rho_{1,n}-\rho_{2,n}))\cr 
& = \left( \sum _{k_i\in \mathcal I} (-1)^{i-1} (n-k_i) \right) +   \textstyle (\upsilon -\frac{1}{2}\sigma)(T((\rho_{n,n}+\rho_{n+1,n} , \rho_{n,n}-\rho_{n+1,n})) \cr
& = \left( \sum _{k_i\in \mathcal I} (-1)^{i-1} (n-k_i) \right) +   \textstyle (\upsilon -\frac{1}{2}\sigma)(T(1,1) \cr
& =  \sum _{k_i\in \mathcal I} (-1)^{i-1} (n-k_i).
\end{align}
Note that if $\mathcal I=\emptyset$ then 
$$(\upsilon  -\textstyle\frac{1}{2}\sigma)(T((r_n +s_n , r_n-s_n))  = \textstyle (\upsilon -\frac{1}{2}\sigma)(T(1,1) = 0,$$
showing that \eqref{EquationAuxiliaryFormulaForP0BeingOddNumber3} remains valid in this case as well, with the convention that the sum $\sum _{k_i\in \mathcal I} (-1)^{i-1} (n-k_i)$ equals zero if $\mathcal I=\emptyset$.

Combining equations \eqref{EquationAuxiliaryFormulaForP0BeingOddNumber1}, \eqref{EquationAuxiliaryFormulaForP0BeingOddNumber2} and \eqref{EquationAuxiliaryFormulaForP0BeingOddNumber3} yields Theorem \ref{MainResultForP0Odd}. 
%%%
%%%
\begin{remark} \label{FirstRemarkAboutTheValueOfTheSumFromMainTheorem}
Observe that the value of $n-k_i$ is a decreasing function of $i\in\{1,\dots, n-1\}$ and therefore the sum $\sum_{k_i\in\mathcal I}(-1)^{i-1}(n-k_i)$ is nonnegative. For nonempty $\mathcal I$, the sum is maxmimized precisely when $\mathcal I=\{1\}$, in which case its value becomes $n-1$.   
\end{remark}
%%%%%%%%%%%%%%%%%%%%%%%%%%%%%%%%%%%%%%%%%%%%%%%%%%%%%%%%%%%%%%%%%%%%%%%%%%%%%%%%%%%%%%%%%%%%%%%%%%%%%%%%%%%%%%%%%%
%%%%%%%%%%%%%%%%%%%%%%%%%%%%%%%%%%%%%%%%%%%%%%%%%%%%%%%%%%%%%%%%%%%%%%%%%%%%%%%%%%%%%%%%%%%%%%%%%%%%%%%%%%%%%%%%%%
%%%%%%%%%%%%%%%%%%%%%%%%%%%%%%%%%%%%%%%%%%%%%%%%%%%%%%%%%%%%%%%%%%%%%%%%%%%%%%%%%%%%%%%%%%%%%%%%%%%%%%%%%%%%%%%%%%
%%%%%%%%%%%%%%%%%%%%%%%%%%%%%%%%%%%%%%%%%%%%%%%%%%%%%%%%%%%%%%%%%%%%%%%%%%%%%%%%%%%%%%%%%%%%%%%%%%%%%%%%%%%%%%%%%%
%%%%%%%%%%%%%%%%%%%%%%%%%%%%%%%%%%%%%%%%%%%%%%%%%%%%%%%%%%%%%%%%%%%%%%%%%%%%%%%%%%%%%%%%%%%%%%%%%%%%%%%%%%%%%%%%%%
\subsection{Proof of Theorem \ref{MainResultForP0Even}} \label{ProofOfMainResultForP0Even}
In this section, we assume that $p> 0$ is even which forces $p_0$ to be even as well. %We distinguish the two cases of $p_0\ge 2$ and $p_0=0$. Recall that in the latter case we are forced to use $\eps_1=-1$. We begin by addressing the case of $p_0\ge 2$. 
Propositions \ref{PropositionForInitialReductionOfUpsilonMinusHalfOfTheSignature} and~\ref{PropositionSecondReductionStepForTheCaseOfP0BeingZero} guarantee that  
\begin{equation} \label{AuxEquation1InTheProveOfTheMainTheoremWhenP0IsEven}
\textstyle (\upsilon-\frac{1}{2}\sigma)(T(p,q)) = (\upsilon-\frac{1}{2}\sigma)(T(2q_n-2\eps_1r_n,q_n)).
\end{equation}
Proposition \ref{PropositionSecondReductionStepForTheCaseOfP0BeingEven} computes the value of $ (\upsilon-\frac{1}{2}\sigma)(T(2q_n-2\eps_1r_n,q_n))$ in terms of  $(\upsilon-\frac{1}{2}\sigma)(T(r_n\pm s_n,2s_n))$ as 
%%%
\begin{align} \label{AuxEquation2InTheProveOfTheMainTheoremWhenP0IsEven}
 (\upsilon  - \textstyle \frac{1}{2}\sigma)&(T(2q_n-2\eps_1r_n,q_n))  =  \cr
& =  \left\{ 
\begin{array}{llll}
\textstyle -(\upsilon -\frac{1}{2}\sigma)(T(r_n-\eta s_n,2s_n)) +n & \qquad  \eps_1=1, &\eps_2=1, \cr
\textstyle \phantom{-}(\upsilon -\frac{1}{2}\sigma)(T(r_n+\eta s_n,2s_n)) +1 & \qquad  \eps_1=1, &\eps_2=-1,\cr
\textstyle -(\upsilon -\frac{1}{2}\sigma)(T(r_n +  \eta s_n , 2s_n))+(n-1)  & \qquad  \eps_1=-1, &\eps_2=1,\cr
\textstyle \phantom{-}(\upsilon -\frac{1}{2}\sigma)(T(r_n -  \eta s_n , 2s_n)) & \qquad  \eps_1=-1, &\eps_2=-1,
\end{array}
\right.
\end{align}
%%%
where $\eta=1$ if $q_1\equiv 1\,(\text{mod }4)$, and let $\eta =-1$ if $q_1\equiv 3\,(\text{mod }4)$. These suffice to prove Theorem \ref{MainResultForP0Even} in the events of $n=1$ or $n=2$. Indeed, if $n=1$ then $r_n=r_1=1$ and $s_n=s_1=0$, and so $T(r_n\pm \eta s_n,2s_n) = T(1,0)$ regardless of the value of $\eta$. It follows from \eqref{AuxEquation1InTheProveOfTheMainTheoremWhenP0IsEven} and \eqref{AuxEquation2InTheProveOfTheMainTheoremWhenP0IsEven} that  
$$\textstyle (\upsilon -\frac{1}{2}\sigma)(T(p,q)) = (\upsilon -\frac{1}{2}\sigma)(T(2q_1-2\eps_1,q_1))  =  \left\{ 
\begin{array}{llll}
1 & \qquad  \eps_1=1, &\eps_2=1, \cr
1 & \qquad  \eps_1=1, &\eps_2=-1,\cr
0  & \qquad  \eps_1=-1, &\eps_2=1,\cr
0 & \qquad  \eps_1=-1, &\eps_2=-1,
\end{array}
\right. $$
%%%
as claimed by Theorem \ref{MainResultForP0Even}.  

If $n=2$ then $r_n=r_2=m_1$ and $s_n=s_2=1$, and therefore, regardless of the value of $\eta$, the torus knot $T(r_n\pm \eta s_n, 2s_n) = T(m_1 \pm \eta, 2)$ is an alternating torus knot, on which the function $\upsilon -\frac{1}{2}\sigma $ vanishes~\cite{OSSUpsilon:2017}. It follows then yet again from \eqref{AuxEquation1InTheProveOfTheMainTheoremWhenP0IsEven} and \eqref{AuxEquation2InTheProveOfTheMainTheoremWhenP0IsEven} that  
%%%
$$\textstyle (\upsilon -\frac{1}{2}\sigma)(T(p,q)) = (\upsilon -\frac{1}{2}\sigma)(T(2q_2-2\eps_1r_2,q_2))  =  \left\{ 
\begin{array}{llll}
2 & \qquad ;  \eps_1=1, &\eps_2=1, \cr
1 & \qquad ; \eps_1=1, &\eps_2=-1,\cr
1  & \qquad ; \eps_1=-1, &\eps_2=1,\cr
0 & \qquad ; \eps_1=-1, &\eps_2=-1,
\end{array}
\right. $$
%%%
which is also as claimed by Theorem \ref{MainResultForP0Even}.  

If $n\ge 3$, we shall rely on Proposition \ref{PropositionThirdReductionStepForTheCaseOfP0BeingEven} which tells us how to compute $(\upsilon-\frac{1}{2}\sigma)(T(\rho_{k,n}\pm \rho_{k+1,n},2\rho_{k+1,n}))$ in term of $(\upsilon-\frac{1}{2}\sigma)(T(\rho_{k+1,n}\pm \rho_{k+2,n},2\rho_{k+2,n}))$. In particular, said proposition shows that if $\eps_{k+1}\eps_{k+2}=1$, then 
$$\textstyle  (\upsilon-\frac{1}{2}\sigma)(T(\rho_{k,n}\pm \rho_{k+1,n},2\rho_{k+1,n})) = (\upsilon-\frac{1}{2}\sigma)(T(\rho_{k+1,n}\pm \rho_{k+2,n},2\rho_{k+2,n})),$$
where the sign on the right-hand side depends on the sign on the left-hand side, and on the congruence class of $m_k$ modulo 4 (if $m_k\equiv 0\,(\text{mod }4)$ the signs are opposite, while if $m_k\equiv 2\,(\text{mod }4)$ the signs are the same). Conversely, if $\eps_{k+1}\eps_{k+2} = -1$, then Proposition \ref{PropositionThirdReductionStepForTheCaseOfP0BeingEven} shows that 
%%%
$$\textstyle  (\upsilon-\frac{1}{2}\sigma)(T(\rho_{k,n}\pm \rho_{k+1,n},2\rho_{k+1,n})) = -(\upsilon-\frac{1}{2}\sigma)(T(\rho_{k+1,n}\pm \rho_{k+2,n},2\rho_{k+2,n})) + (n-k-1).$$
%%%
Here too the signs in $\rho_{k,n}\pm \rho_{k+1,n}$ and $\rho_{k+1,n}\pm \rho_{k+2,n}$ depend on the congruence class of $m_k$ modulo 4 (if $m_k\equiv 0\,(\text{mod }4)$ the signs are the same, if $m_k\equiv 2\,(\text{mod }4)$ the signs are opposite). This prompts the definition of the set $\mathcal J$ as 
$$\mathcal J = \{k\in \{1,\dots, n-2\}\,|\, \eps_{k+1}\eps_{k+2} = -1\},$$ 
and we write $\mathcal J=\{k_{1}, k_{2}, \dots, k_{\ell}\}$ with $k_{1}<k_{2}<\dots<k_{\ell}$. It follows from this discussion that 
%%%
\begin{align*}
\textstyle (\upsilon-\frac{1}{2}\sigma)&(T(r_n\pm s_n,2s_n))  = \textstyle (\upsilon-\frac{1}{2}\sigma)(T(\rho_{1,n}\pm \rho_{2,n},2\rho_{2,n})) \cr
& = \textstyle \left( \sum _{k_j\in \mathcal J} (-1)^{j-1} (n-k_j-1)\right) + (\upsilon-\frac{1}{2}\sigma)(T(\rho_{n,n}\pm \rho_{n+1,n},2\rho_{n+1,n}))   \cr
& = \textstyle \left( \sum _{k_j\in \mathcal J} (-1)^{j-1} (n-k_j-1)\right) + (\upsilon-\frac{1}{2}\sigma)(T(1,0))  \cr
& = \textstyle\sum _{k_j\in \mathcal J} (-1)^{j-1} (n-k_j-1).
\end{align*}
%%%
Combined with \eqref{AuxEquation1InTheProveOfTheMainTheoremWhenP0IsEven} and \eqref{AuxEquation2InTheProveOfTheMainTheoremWhenP0IsEven} this completes the proof of Theorem \ref{MainResultForP0Even}, after a shift of indices replacing $k+1$ by $k$.
%%%
%%%
\begin{remark} \label{SecondRemarkAboutTheValueOfTheSumFromMainTheorem}
The value of $n-k_j$ is a decreasing function of $j\in\{2,\dots, n-1\}$ and therefore the sum $\sum_{k_j\in\mathcal J}(-1)^{j-1}(n-k_j)$ is nonnegative. For nonempty $\mathcal J$, the sum is maxmimized precisely when $\mathcal J=\{2\}$, in which case its value becomes $n-2$.   
\end{remark}
%%%%%%%%%%%%%%%%%%%%%%%%%%%%%%%%%%%%%%%%%%%%%%%%%%%%%%%%%%%%%%%%%%%%%%%%%%%%%%%%%%%%%%%%%%%%%%%%%%%%%%%%%%%%%%%
%%%%%%%%%%%%%%%%%%%%%%%%%%%%%%%%%%%%%%%%%%%%%%%%%%%%%%%%%%%%%%%%%%%%%%%%%%%%%%%%%%%%%%%%%%%%%%%%%%%%%%%%%%%%%%%
\subsection{Proof of Theorem \ref{TheoremOnPartialSolutionOfNonorientableMilnorConjecture}} \label{SectionOnPartialProofOfNonOrientableMilnorConjecture}
%%%%
Let $p,q>1$ be a pair of relatively prime integers of which $q$ is odd. Let $n$, $p_0$, $q_1$, $\{m_k\}_{k=1}^{n-1}$, $\{\eps_k\}_{k=1}^n$ be obtained from $(p,q)$ as in Theorem \ref{TheoremOnReversingPinchingMoves}. 

Consider firstly the case of $p$ odd, and let $\mathcal I$ be defined as in Theorem \ref{MainResultForP0Odd}. If $\mathcal I$ is not the empty set, write $\mathcal  I = \{ k_1, \dots, k_\ell\}$ with $k_1<k_2 < \dots < k_\ell$. Theorem \ref{MainResultForP0Odd} says that
$$\textstyle (\upsilon -\frac{1}{2}\sigma)(T(p,q)) = \left\{
\begin{array}{rl}
n - \sum _{k_i\in \mathcal I} (-1)^{i-1} (n-k_{i}) & \quad ; \quad q_1-\eps_1 \equiv 0\,(\text{mod }4), \cr & \cr
\sum _{k_i\in \mathcal I} (-1)^{i-1} (n-k_{i}) & \quad ; \quad q_1-\eps_1 \equiv 2\,(\text{mod }4),
\end{array}
\right.$$
with the sum over the empty set equaling zero. If $\mathcal I\ne \emptyset$, the alternating sum $\sum _{i\in \mathcal I} (-1)^{i-1} (n-k_{i})$ appearing in this formula, can take on values between 1 and $n-1$. It is greater than or equal to 1 because the terms in the sum are decreasing in absolute value, with the first one being positive. It is less than or equal to $n-1$ since its first term $n-k_1$ is less than or equal to $n-1$. Given this, the value of  $(\upsilon -\frac{1}{2}\sigma)(T(p,q))$ is less than or equal to $n-1$ regardless of the congruence class of $q_1-\eps_1$ modulo 4. 

If $\mathcal I = \emptyset$, which occurs precisely when $m_k\equiv 2\,(\text{mod }4)$ for $k=1,\dots, n-1$, then 
$$\textstyle (\upsilon -\frac{1}{2}\sigma)(T(p,q)) = \left\{
\begin{array}{rl}
n & \quad ; \quad q_1-\eps_1 \equiv 0\,(\text{mod }4), \cr & \cr
0 & \quad ; \quad q_1-\eps_1 \equiv 2\,(\text{mod }4),
\end{array}
\right.$$
proving part (a) of Theorem \ref{TheoremOnPartialSolutionOfNonorientableMilnorConjecture}. 
\vskip3mm
Secondly, consider the case of $p$ even, and let  $\mathcal J$ be defined as in Theorem \ref{MainResultForP0Even}. If $\mathcal J$ is not the empty set, write $\mathcal  J = \{ k_1, \dots, k_\ell\}$ with $k_1<k_2 < \dots < k_\ell$. If $n=1$ then part (b) of Theorem \ref{TheoremOnPartialSolutionOfNonorientableMilnorConjecture} is trivially true by Theorem \ref{MainResultForP0Even}. If $n\ge 2$, then Theorem \ref{MainResultForP0Even} says that 
%%%
$$\textstyle (\upsilon -\frac{1}{2}\sigma)(T(p,q)) = \left\{
\begin{array}{rll}
n - \sum _{k_j\in \mathcal J} (-1)^{j-1} (n-k_{j}) & \quad ; \quad \eps_1=1, &\eps_2=1, \cr
1 +\sum _{k_j\in \mathcal J} (-1)^{j-1} (n-k_{j}) & \quad ; \quad \eps_1=1, &\eps_2=-1, \cr
(n-1) - \sum _{k_j\in \mathcal J} (-1)^{j-1} (n-k_{j}) & \quad ; \quad \eps_1=-1, &\eps_2=1, \cr
\sum _{k_j\in \mathcal J} (-1)^{j-1} (n-k_{j}) & \quad ; \quad \eps_1=-1, &\eps_2=-1, \cr
\end{array}
\right.$$
%%%
with the sum over the empty set equaling zero. If $\mathcal J\ne \emptyset$, then the alternating sum $\sum _{j\in \mathcal J} (-1)^{j-1} (n-k_{j})$ appearing in all four cases in the above formula, lies between 1 and $n-2$. It is greater than or equal to 1 because the terms in the sum are decreasing in absolute value, with the first one being positive. It is less than or equal to $n-2$ since its first term $n-k_1$ is less than or equal to $n-2$. Given this, the value of  $(\upsilon -\frac{1}{2}\sigma)(T(p,q))$ is less than or equal to $n-1$ if $\eps_1 =1$, and is less than or equal to $n-2$ if $\eps_1=-1$. 

If $\mathcal J=\emptyset$, which can only occur if $n=2$ or if $n\ge 3$ and $\eps_k=\eps_2$ for all $k=3, \dots, n$, then    
%%%
$$\textstyle (\upsilon -\frac{1}{2}\sigma)(T(p,q)) = \left\{
\begin{array}{rll}
n  & \quad ; \quad \eps_1=1, &\eps_2=1, \cr
1  & \quad ; \quad \eps_1=1, &\eps_2=-1, \cr
(n-1) & \quad ; \quad \eps_1=-1, &\eps_2=1, \cr
0 & \quad ; \quad \eps_1=-1, &\eps_2=-1. \cr
\end{array}
\right.$$
%%%
We see that in this event we obatin $(\upsilon -\frac{1}{2}\sigma)(T(p,q)) = n$ if and only if $\eps_1=\eps_2=1$, proving part (b) of Theorem \ref{TheoremOnPartialSolutionOfNonorientableMilnorConjecture}. 
\vskip3mm
Similar arguments to those given above in the proof of Theorem \ref{TheoremOnPartialSolutionOfNonorientableMilnorConjecture}, can be used to obtain this characterization of knots $T(p,q)$ with $(\upsilon-\frac{1}{2}\sigma)(T(p,q)) = n-1$. 
%%%
%%%
\begin{proposition} \label{PropositionCharacterizingTheCaseOfUpsilonMinusSigmaOverTwoEqualsN-1}
Let $p,q>1$ be a pair of relatively prime integers of which $q$ is odd. Let $n$, $p_0$, $q_1$, $\{\eps_k\}_{k=1}^n$ be obtained from $(p,q)$ as in Theorem \ref{TheoremOnReversingPinchingMoves}.
%%%
\begin{itemize}
\item[(a)] If $p$ is odd then 
$$\upsilon (T(p,q)) -\textstyle \frac{1}{2}\sigma (T(p,q)) = n-1$$
if and only if one of the next two possibilities occurs.
\begin{itemize}
\item[(i)] $q_1-\eps_1 \equiv 0 \,(\text{mod }4)$ and either $m_{n-1}\equiv 0\,(\text{mod }4)$ (and $m_k\equiv 2\,(\text{mod }4)$ for all $k\ne n-1$), or if there are exactly two indices $k_1, k_2$, with $k_2=k_1+1$ and $m_{k_1},m_{k_2} \equiv 0 \,(\text{mod }4)$. %The choice of $k_1=n-2$ is allowed. 
\item[(ii)] $q_1-\eps_1 \equiv 2 \,(\text{mod }4)$ and $m_1\equiv 0\,(\text{mod }4)$ (and $m_k\equiv 2\,(\text{mod }4)$ for all $k>1$).
\end{itemize} 
For any such torus knot $T(p,q)$ one obtains 
\begin{equation} \label{EquationTorusKnotWithGamma4BetweenN-1AndN}
n-1\le \gamma_4(T(p,q)) \le n.
\end{equation}
%%%
\vskip2mm
\item[(b)] If $p$ is even then 
$$\upsilon (T(p,q)) -\textstyle \frac{1}{2}\sigma (T(p,q)) = n-1$$
if and only if there is exactly one index $k_0\in \{1,\dots, n\}$ with $\eps_{k_0}=-1$ (and $\eps_k=1$ for all $k\ne k_0$).  For any such torus knot $T(p,q)$ the double bound \eqref{EquationTorusKnotWithGamma4BetweenN-1AndN} holds. 
\end{itemize}
\end{proposition}
%%%
%%%
\begin{proof}
The proof of this proposition is in the spirit of the proof of Theorem \ref{TheoremOnPartialSolutionOfNonorientableMilnorConjecture} given at the beginning of this section, and we shall recycle some of the details (and all of the notation) given there. 

If $p$ is odd then  
$$\textstyle (\upsilon -\frac{1}{2}\sigma)(T(p,q)) = \left\{
\begin{array}{rl}
n - \sum _{k_i\in \mathcal I} (-1)^{i-1} (n-k_{i}) & \quad ; \quad q_1-\eps_1 \equiv 0\,(\text{mod }4), \cr & \cr
\sum _{k_i\in \mathcal I} (-1)^{i-1} (n-k_{i}) & \quad ; \quad q_1-\eps_1 \equiv 2\,(\text{mod }4),
\end{array}
\right.$$
with $1\le \sum _{k_i\in \mathcal I} (-1)^{i-1} (n-k_{i})\le n-1$, provided $\mathcal I \ne \emptyset$. 

If  $q_1-\eps_1\equiv 0 \,(\text{mod }4)$ then $(\upsilon - \frac{1}{2}\sigma) (T(p,q)) = n-1$ if and only if $\sum _{i\in \mathcal I} (-1)^{i-1} (n-k_{i}) = 1$. The latter equality occurs if and only if the sum consists of the single nonzero term $n-k_1$ (forcing $\ell=1$ and $k_1=n-1$) or of the two nonzero terms $(n-k_1)-(n-k_2)$ (forcing $\ell=2$ and $k_2-k_1=1$). These two cases correspond to $m_{n-1}\equiv 0\,(\text{mod }4)$ (and with $m_{k}\equiv 2\,(\text{mod }4)$ for all $k\ne n-1$), and to there being exactly two indices $k_1$, $k_2$ with $k_2=k_1+1$  and with $m_{k_1}, m_{k_2} \equiv 0\,(\text{mod }4)$, respectively. 

If $q_1-\eps_1\equiv 2 \,(\text{mod }4)$ then  $(\upsilon - \frac{1}{2}\sigma) (T(p,q)) = n-1$  if and only if $\sum _{i\in \mathcal I} (-1)^{i-1} (n-k_{i}) = n-1$, which occurs if and only if the sum consists of the single term $n-k_1$, leading to $\ell = 1$ and $k_1=1$. 
\vskip3mm
If $p$ is even then 
%%%
$$\textstyle (\upsilon -\frac{1}{2}\sigma)(T(p,q)) = \left\{
\begin{array}{rll}
n - \sum _{k_j\in \mathcal J} (-1)^{j-1} (n-k_{j}) & \quad ; \quad \eps_1=1, &\eps_2=1, \cr
1 +\sum _{k_j\in \mathcal J} (-1)^{j-1} (n-k_{j}) & \quad ; \quad \eps_1=1, &\eps_2=-1, \cr
(n-1) - \sum _{k_j\in \mathcal J} (-1)^{j-1} (n-k_{j}) & \quad ; \quad \eps_1=-1, &\eps_2=1, \cr
\sum _{k_j\in \mathcal J} (-1)^{j-1} (n-k_{j}) & \quad ; \quad \eps_1=-1, &\eps_2=-1, \cr
\end{array}
\right.$$
%%%
with $1\le \sum _{k_j\in \mathcal J} (-1)^{j-1} (n-k_{j}) \le n-2$ provided $\mathcal J\ne \emptyset$. 

If $\eps_1=\eps_2=1$ then $(\upsilon -\frac{1}{2}\sigma)(T(p,q)) = n-1$ if and only if $\sum _{k_j\in \mathcal J} (-1)^{j-1} (n-k_{j}) = 1$, which occurs if and only if either the sum consists of the single term $n-k_1$ (forcing $\ell=1$ and $k_1=n-1$), or it consists of the two terms $(n-k_1)-(n-k_2)$ (forcing $l=2$ and $k_2=k_1+1$). These cases respectively correspond to $\eps_n=-1$ (and $\eps_k=1$ for $k=1,\dots, n-1$), and to $\eps_{k_2+1} = -1$ (and $\eps_k=1$ for $k\ne k_2+1$).

If $\eps_1=1$ and $\eps_2=-1$ then $(\upsilon -\frac{1}{2}\sigma)(T(p,q)) = n-1$ if and only if $\sum _{k_j\in \mathcal J} (-1)^{j-1} (n-k_{j}) = n-2$, which can only happen if $\ell=1$ and $k_1=2$. This case occurs if and only if $\eps_2=-1$ and $\eps_k=1$ for $k\ne 2$. 

If $\eps_1=-1$ and $\eps_2=1$ then $(\upsilon -\frac{1}{2}\sigma)(T(p,q)) = n-1$ if and only if $\sum _{k_j\in \mathcal J} (-1)^{j-1} (n-k_{j}) = 0$, which occurs only when $\mathcal J = \emptyset$, which in turn occurs if and only if $\eps_1=-1$ and $\eps_k=1$ for $k>1$.  

If $\eps_1=-1=\eps_2$ then  $(\upsilon -\frac{1}{2}\sigma)(T(p,q)) \le n-2$. 
\end{proof}
%%%%%%%%%%%%%%%%%%%%%%%%%%%%%%%%%%%%%%%%%%%%%%%%%%%%%%%%%%%%%%%%%%%%%%%%%%%%%%%%%%%%%%%%%%%%%%%%%%%%%%%%%%%%%%%
%%%%%%%%%%%%%%%%%%%%%%%%%%%%%%%%%%%%%%%%%%%%%%%%%%%%%%%%%%%%%%%%%%%%%%%%%%%%%%%%%%%%%%%%%%%%%%%%%%%%%%%%%%%%%%%
%%%%%%%%%%%%%%%%%%%%%%%%%%%%%%%%%%%%%%%%%%%%%%%%%%%%%%%%%%%%%%%%%%%%%%%%%%%%%%%%%%%%%%%%%%%%%%%%%%%%%%%%%%%%%%%
\section{Counterexamples} \label{SectionAboutCounterexamples}
%%%
%%%
Lobb's discovery of the failure of Conjecture \ref{NonorientableAnalogueOfMilnorsConjecture} for the torus knot $T(4,9)$ can be used to produce infinitely many more counterexamples. To see how, recall the parameters 
$$n, \quad p_0, \quad q_1, \quad  \{m_k\}_{k=1}^{n-1}, \quad \text{ and } \quad \{\eps_k\}_{k=1}^n$$
from Theorem \ref{TheoremOnReversingPinchingMoves}, from which the relatively prime pairs $(p_k, q_k)$ are obtained as (with $q_0 = 1$) 
\begin{align*}
p_1& = q_1p_0-2\eps_1, \cr 
p_k & = m_{k-1}p_{k-1}-\eps_k\eps_{k-1} p_{k-2}, \quad \quad k\ge 2, \cr
q_k & = m_{k-1}q_{k-1}-\eps_k\eps_{k-1} q_{k-2}, \quad \quad\, k\ge 2. 
\end{align*}
Their significance is that each torus knot $T(p_k,q_k)$, $k\ge 0$ is obtained from the knot $T(p_{k+1}, q_{k+1})$ by a single pinch move with sign $\eps_{k+1}$, and the value of $\gamma_4(T(p_k,q_k))$, $k\ge 1$ predicted by Conjecture \ref{NonorientableAnalogueOfMilnorsConjecture} is precisely $k$. 

Pick an arbitrary $n\ge 3$ and choose $p_0=0$, $q_1=5$ and $\eps_1=-1$. For $2\le k\le n$ pick $m_k\ge 2$ even, and $\eps_k\in\{\pm1\}$ at will, allowing for infinitely many choices. Then $T(p_2,q_2) = T(4,9)$ and each knot $T(p_k,q_k)$, $k\ge 3$ is related to $T(4,9)$ by $k-2$ pinch moves. Since $T(4,9)$ bounds a M\"obius band in the 4-ball, it follows that $\gamma_4(T(p_k,q_k))$, $k\ge 3$ is at most $k-1$, proving Corollary \ref{CorollaryToLobbsResult}.  

The concrete choices of $m_k=2$ and $\eps_k=-1$ for all $k$ leads to 
$$T(p_n,q_n) = T(2n,4n+1).$$
Thus, these torus knots have $\gamma_4$ at most $n-1$, but are predicted to have $\gamma_4$ equal to $n$ by Conjecture \ref{NonorientableAnalogueOfMilnorsConjecture}.  It is interesting to ask whether counterexamples can be found for which the predicted and actual value of $\gamma_4$ differ by an arbitrarily large amount.   
%%%%%%%%%%%%%%%%%%%%%%%%%%%%%%%%%%%%%%%%%%%%%%%%%%%%%%%%
%%%%%%%%%%%%%%%%%%%%%%%%%%%%%%%%%%%%%%%%%%%%%%%%%%%%%%%%
%%%%%%%%%%%%%%%%%%%%%%%%%%%%%%%%%%%%%%%%%%%%%%%%%%%%%%%%
%%%%%%%%%%%%%%%%%%%%%%%%%%%%%%%%%%%%%%%%%%%%%%%%%%%%%%%%
%%%%%%%%%%%%%%%%%%%%%%%%%%%%%%%%%%%%%%%%%%%%%%%%%%%%%%%%
%%%%%%%%%%%%%%%%%%%%%%%%%%%%%%%%%%%%%%%%%%%%%%%%%%%%%%%%
%%%%%%%%%%%%%%%%%%%%%%%%%%%%%%%%%%%%%%%%%%%%%%%%%%%%%%%%
%%%%%%%%%%%%%%%%%%%%%%%%%%%%%%%%%%%%%%%%%%%%%%%%%%%%%%%%
%%%%%%%%%%%%%%%%%%%%%%%%%%%%%%%%%%%%%%%%%%%%%%%%%%%%%%%%
%%%%%%%%%%%%%%%%%%%%%%%%%%%%%%%%%%%%%%%%%%%%%%%%%%%%%%%%
%%%%%%%%%%%%%%%%%%%%%%%%%%%%%%%%%%%%%%%%%%%%%%%%%%%%%%%%
%%%%%%%%%%%%%%%%%%%%%%%%%%%%%%%%%%%%%%%%%%%%%%%%%%%%%%%%
%%%%%%%%%%%%%%%%%%%%%%%%%%%%%%%%%%%%%%%%%%%%%%%%%%%%%%%%

%%%%%%%%%%%%%%%%%%%%%%%%%%%%%%%%%%%%%%%%%%%%%%%%%%%%%%%%%%%%%%%%%%%%%%%%%%%%%%%%%%%%%%%%%%%%%%%%%%%%%%%%%%%%%%%%%%%
%%%%%%%%%%%%%%%%%%%%%%%%%%%%%%%%%%%%%%%%%%%%%%%%%%%%%%%%%%%%%%%%%%%%%%%%%%%%%%%%%%%%%%%%%%%%%%%%%%%%%%%%%%%%%%%%%%%
%%%%%%%%%%%%%%%%%%%%%%%%%%%%%%%%%%%%%%%%%%%%%%%%%%%%%%%%%%%%%%%%%%%%%%%%%%%%%%%%%%%%%%%%%%%%%%%%%%%%%%%%%%%%%%%%%%%
%%%%%%%%%%%%%%%%%%%%%%%%%%%%%%%%%%%%%%%%%%%%%%%%%%%%%%%%%%%%%%%%%%%%%%%%%%%%%
%%%%%%%%%%%%%%%%%%%%%%%%%%%%%%%%%%%%%%%%%%%%%%%%%%%%%%%%%%%%%%%%%%%%%%%%%%%%%
%%%%%%%%%%%%%%%%%%%%%%%%%%%%%%%%%%%%%%%%%%%%%%%%%%%%%%%%%%%%%%%%%%%%%%%%%%%%%

\end{document}